\documentclass[12pt,reqno]{amsart}

\usepackage{amsmath, amsfonts, amssymb, amsthm, mathrsfs}

\usepackage[utf8,utf8x]{inputenc}
\usepackage[pagebackref=true]{hyperref}
\hypersetup{%
	pdftitle = {},
	pdfkeywords = {},
	pdfsubject = {},
	pdfauthor = {},
	colorlinks = true,
	linktoc = page,
	citecolor = purple,
	linkcolor = blue 
}

\usepackage{xcolor}
\usepackage{caption}
\usepackage{tikz}
\usetikzlibrary{calc}
\usetikzlibrary{intersections}
\usetikzlibrary{cd}
\usetikzlibrary{patterns}

\tikzcdset{scale cd/.style={every label/.append style={scale=#1},
		cells={nodes={scale=#1}}}}

\usepackage{enumerate}
\usepackage{mfirstuc}

\usepackage{todonotes}

\newcommand{\CC }{\mathbb{C}}
\newcommand{\RR }{\mathbb{R}}

\newcommand{\ZZ }{\mathbb{Z}}

\newcommand{\Ac }{\mathcal{A}}

\newcommand{\Bc }{\mathcal{B}}

\newcommand{\Ec }{\mathcal{E}}

\newcommand{\Vc }{\mathcal{V}}

\newcommand{\Cc }{\mathcal{C}}

\newcommand{\Lc }{\mathcal{L}}

\newcommand{\OM }{\mathcal{M}}
\newcommand{\Fc }{\mathcal{F}}

\newcommand{\Tc }{\mathcal{T}}
\newcommand{\Sc }{\mathcal{S}}

\newcommand{\Zc }{\mathcal{Z}}

\newcommand{\isom }{\simeq}

\newcommand{\Xf }{\mathfrak{X}}

\newcommand{\Ff }{\mathfrak{F}}

\DeclareMathOperator{\sgn}{sgn}
\DeclareMathOperator{\mtc}{\underline{M}}

\DeclareMathOperator{\Zero}{\mathbf{0}}

\newcommand\wt[1]{\widetilde{#1}}

\DeclareMathOperator{\rk}{rk}

\DeclareMathOperator{\vts}{vert}

\DeclareMathOperator{\leLc}{\triangleleft\mspace{-1mu}\cdot}

\DeclareMathOperator{\Pos}{\mathtt{Pos}}

\DeclareMathOperator{\Top}{\mathtt{Top}}

\DeclareMathOperator{\HoFib}{HoFib}

\def\dual{\vee}

\DeclareMathOperator{\sd}{sd}
\DeclareMathOperator{\rksdS}{rksd\mathcal{S}}
\DeclareMathOperator{\subdiv}{\triangleleft}

\newcommand\afr{\mathfrak{a}}
\newcommand\bfr{\mathfrak{b}}

\numberwithin{equation}{section}

\theoremstyle{plain}
\newtheorem{lemma}[equation]{Lemma}
\newtheorem{theorem}[equation]{Theorem}

\newtheorem{corollary}[equation]{Corollary}
\newtheorem{proposition}[equation]{Proposition}
\theoremstyle{definition}
\newtheorem{definition}[equation]{Definition}
\newtheorem{remark}[equation]{Remark}

\newtheorem{example}[equation]{Example}
\newtheorem{problem}[equation]{Problem}

\usepackage{pgfplots, tikz-3dplot, ifthen}

\def\samples{15}

\newcommand{\GCArcAB}[7]{
	\pgfmathsetmacro{\Ax}{#1}
	\pgfmathsetmacro{\Ay}{#2}
	\pgfmathsetmacro{\Az}{#3}
	\pgfmathsetmacro{\Bx}{#4}
	\pgfmathsetmacro{\By}{#5}
	\pgfmathsetmacro{\Bz}{#6}
	
	\pgfmathsetmacro{\nA}{sqrt(\Ax*\Ax+\Ay*\Ay+\Az*\Az)}
	\pgfmathsetmacro{\nB}{sqrt(\Bx*\Bx+\By*\By+\Bz*\Bz)}
	
	\pgfmathsetmacro{\nAx}{\Ax/\nA}
	\pgfmathsetmacro{\nAy}{\Ay/\nA}
	\pgfmathsetmacro{\nAz}{\Az/\nA}
	\pgfmathsetmacro{\nBx}{\Bx/\nB}
	\pgfmathsetmacro{\nBy}{\By/\nB}
	\pgfmathsetmacro{\nBz}{\Bz/\nB}
	
	\pgfmathsetmacro{\a}{acos(\nAx*\nBx+\nAy*\nBy+\nAz*\nBz)}
	
	\draw[#7] plot[variable=\t, domain=1:0, samples=\samples]
	({\nAx*sin((1-\t)*\a)/sin(\a) +\nBx*sin(\t*\a)/sin(\a)},
	{\nAy*sin((1-\t)*\a)/sin(\a) + \nBy*sin(\t*\a)/sin(\a)},
	{\nAz*sin((1-\t)*\a)/sin(\a) + \nBz*sin(\t*\a)/sin(\a)});
}

\newcommand{\GCArcABfb}[7]{
	\pgfmathsetmacro{\Ax}{#1}
	\pgfmathsetmacro{\Ay}{#2}
	\pgfmathsetmacro{\Az}{#3}
	\pgfmathsetmacro{\Bx}{#4}
	\pgfmathsetmacro{\By}{#5}
	\pgfmathsetmacro{\Bz}{#6}
	
	\pgfmathsetmacro{\nA}{sqrt(\Ax*\Ax+\Ay*\Ay+\Az*\Az)}
	\pgfmathsetmacro{\nB}{sqrt(\Bx*\Bx+\By*\By+\Bz*\Bz)}
	
	\pgfmathsetmacro{\nAx}{\Ax/\nA}
	\pgfmathsetmacro{\nAy}{\Ay/\nA}
	\pgfmathsetmacro{\nAz}{\Az/\nA}
	\pgfmathsetmacro{\nBx}{\Bx/\nB}
	\pgfmathsetmacro{\nBy}{\By/\nB}
	\pgfmathsetmacro{\nBz}{\Bz/\nB}
	
	\let\p\tdplotmainphi
	\let\t\tdplotmaintheta
	\pgfmathsetmacro\azx{sin(\p)*sin(\t)}
	\pgfmathsetmacro\azy{-cos(\p)*sin(\t)}
	\pgfmathsetmacro\azz{cos(\t)}
	
	\pgfmathsetmacro{\sAZ}{\nAx*\azx+\nAy*\azy+\nAz*\azz}
	\pgfmathsetmacro{\sBZ}{\nBx*\azx+\nBy*\azy+\nBz*\azz}
	\pgfmathsetmacro{\tt}{\sBZ/(\sBZ-\sAZ)}
	
	\pgfmathsetmacro{\af}{\sAZ>=0}
	\pgfmathsetmacro{\ab}{\sAZ<0}
	\pgfmathsetmacro{\bf}{\sBZ>=0}
	\pgfmathsetmacro{\bb}{\sBZ<0}
	\pgfmathsetmacro{\abf}{\af+\bf>1}
	\pgfmathsetmacro{\abb}{\ab+\bb>1}

	\ifthenelse{\abf=1}{
		\GCArcAB{\Ax}{\Ay}{\Az}{\Bx}{\By}{\Bz}{#7}
	}{
		\ifthenelse{\abb=1}{%
			\GCArcAB{\Ax}{\Ay}{\Az}{\Bx}{\By}{\Bz}{opacity=0.25, #7}
		}{
			\pgfmathsetmacro{\Mx}{\tt*\nAx+(1-\tt)*\nBx}
			\pgfmathsetmacro{\My}{\tt*\nAy+(1-\tt)*\nBy}
			\pgfmathsetmacro{\Mz}{\tt*\nAz+(1-\tt)*\nBz}
			\pgfmathsetmacro{\nM}{sqrt(\Mx*\Mx+\My*\My+\Mz*\Mz)}
			\pgfmathsetmacro{\nMx}{\Mx/\nM}
			\pgfmathsetmacro{\nMy}{\My/\nM}
			\pgfmathsetmacro{\nMz}{\Mz/\nM}
			
			\ifthenelse{\af>0}{
				\pgfmathsetmacro{\aAM}{acos(\nAx*\nMx+\nAy*\nMy+\nAz*\nMz)}
				
				\draw[#7] plot[variable=\t, domain=1:0, samples=\samples]
				({\nAx*sin((1-\t)*\aAM)/sin(\aAM) +\nMx*sin(\t*\aAM)/sin(\aAM)},
				{\nAy*sin((1-\t)*\aAM)/sin(\aAM) + \nMy*sin(\t*\aAM)/sin(\aAM)},
				{\nAz*sin((1-\t)*\aAM)/sin(\aAM) + \nMz*sin(\t*\aAM)/sin(\aAM)});
				
				\pgfmathsetmacro{\aBM}{acos(\nBx*\nMx+\nBy*\nMy+\nBz*\nMz)}
				
				\draw[opacity=0.25, #7] plot[variable=\t, domain=1:0, samples=\samples]
				({\nBx*sin((1-\t)*\aBM)/sin(\aBM) +\nMx*sin(\t*\aBM)/sin(\aBM)},
				{\nBy*sin((1-\t)*\aBM)/sin(\aBM) + \nMy*sin(\t*\aBM)/sin(\aBM)},
				{\nBz*sin((1-\t)*\aBM)/sin(\aBM) + \nMz*sin(\t*\aBM)/sin(\aBM)});
			}{
				\pgfmathsetmacro{\aAM}{acos(\nAx*\nMx+\nAy*\nMy+\nAz*\nMz)}
				
				\draw[opacity=0.25, #7] plot[variable=\t, domain=1:0, samples=\samples]
				({\nAx*sin((1-\t)*\aAM)/sin(\aAM) +\nMx*sin(\t*\aAM)/sin(\aAM)},
				{\nAy*sin((1-\t)*\aAM)/sin(\aAM) + \nMy*sin(\t*\aAM)/sin(\aAM)},
				{\nAz*sin((1-\t)*\aAM)/sin(\aAM) + \nMz*sin(\t*\aAM)/sin(\aAM)});
				
				\pgfmathsetmacro{\aBM}{acos(\nBx*\nMx+\nBy*\nMy+\nBz*\nMz)}
				
				\draw[#7] plot[variable=\t, domain=1:0, samples=\samples]
				({\nBx*sin((1-\t)*\aBM)/sin(\aBM) +\nMx*sin(\t*\aBM)/sin(\aBM)},
				{\nBy*sin((1-\t)*\aBM)/sin(\aBM) + \nMy*sin(\t*\aBM)/sin(\aBM)},
				{\nBz*sin((1-\t)*\aBM)/sin(\aBM) + \nMz*sin(\t*\aBM)/sin(\aBM)});					
			}				
		}
	}		
}

\newcommand{\POnSfb}[4]{
	\pgfmathsetmacro{\Px}{#1}
	\pgfmathsetmacro{\Py}{#2}
	\pgfmathsetmacro{\Pz}{#3}
	\pgfmathsetmacro{\nP}{sqrt(\Px*\Px+\Py*\Py+\Pz*\Pz)}
	\pgfmathsetmacro{\nPx}{\Px/\nP}
	\pgfmathsetmacro{\nPy}{\Py/\nP}
	\pgfmathsetmacro{\nPz}{\Pz/\nP}
	
	\let\p\tdplotmainphi
	\let\t\tdplotmaintheta
	\pgfmathsetmacro\azx{sin(\p)*sin(\t)}
	\pgfmathsetmacro\azy{-cos(\p)*sin(\t)}
	\pgfmathsetmacro\azz{cos(\t)}
	
	\pgfmathsetmacro{\sPZ}{\nPx*\azx+\nPy*\azy+\nPz*\azz}
	
	\pgfmathsetmacro{\Pf}{\sPZ>=0}
	
	\ifthenelse{\Pf>0}{%
		\filldraw[#4] ({\nPx},{\nPy},{\nPz}) circle[radius=0.5pt];
	}{%
		\filldraw[opacity=0.25, #4] ({\nPx},{\nPy},{\nPz}) circle[radius=0.5pt];
	}
}

\title[Milnor fibrations and oriented matroids]
{Milnor fibrations and oriented matroids}

\author[P.~M\"ucksch]{Paul M\"ucksch}
\address
{Institut für Algebra, Zahlentheorie und Diskrete Mathematik, 
	Fakultät für Mathematik und Physik, 
	Leibniz Universität Hannover, 
	Welfengarten 1, D-30167 Hannover, Germany}
\email{muecksch@math.uni-hannover.de}

\author{Masahiko Yoshinaga}
\address{Masahiko Yoshinaga,
	Department of Mathematics,
	Graduate School of Science,
	The University of Osaka,
	Toyonaka, Osaka 560-0043, Japan}
\email{yoshinaga@math.sci.osaka-u.ac.jp}

\begin{document}
	
	\begin{abstract}
		We introduce a combinatorial model for the Milnor fibration of a complexified real arrangement
		using oriented matroids.
		It is a poset quasi-fibration, whose domain is a subdivision of the Salvetti complex
		stemming from a natural subdivision of the dual oriented matroid complex.
		This yields a concrete finite regular CW complex which
		is homotopy equivalent to the Milnor fiber of the complexified real arrangement 
		and implies that the homotopy type of the Milnor fiber
		of a complexified real arrangement only depends on the underlying combinatorial structure given
		by its oriented matroid. 
		Moreover, our construction works for any oriented matroid, disregarding realizability,
		so we obtain a notion of a combinatorial Milnor fibration for any oriented matroid.
	\end{abstract}
	
	
	\keywords{Milnor fiber, Oriented matroid, Salvetti complex, poset quasi-fibration}
	\subjclass[2020]{Primary: 32S55, 52C40, 52C35; Secondary: 32S22, 55R10, 57Q70}
	
	\maketitle
	
	
	\section{Introduction}
	\label{sec:Introduction}
	
	For a homogeneous polynomial $f \in \CC[x_1,\ldots,x_\ell]$ the evaluation of $f$
	restricted to the hypersurface complement 
	$\Xf(f) = \CC^\ell\setminus f^{-1}(0)$
	yields a locally trivial fibration
	\[
	f|_{\Xf(f)}:\Xf(f) \to \CC^\times, v\mapsto f(v),
	\]
	known as the \emph{Milnor fibration} of $f$,
	by a classical result of Milnor \cite{Milnor1968_SingularPoints}. 
	Its typical fiber $f^{-1}(1)$ is called the \emph{Milnor fiber} of $f$.
	If the singularity of $f$ at the origin is isolated, then $f^{-1}(1)$ has the homotopy type of a wedge of spheres,
	again thanks to Milnor \cite{Milnor1968_SingularPoints}.
	
	In contrast, if the origin is not an isolated singularity, then much less is known about the topology of the Milnor fiber
	and its study is an important topic at the crossroads of algebraic geometry, singularity theory and algebraic topology (see e.g.\ \cite{DP2003_Hypersurf, NS2007_Motivic, NS2012_Milnorboundary}).
	
	A typical and important instance of a most non-isolated singularity is given by a hyperplane arrangement $\Ac$,
	i.e., a finite collection of (linear) hyperplanes in $V = \CC^\ell$. 
	For a choice of defining linear forms $\alpha_H \in (\CC^\ell)^*$ ($H =\ker(\alpha_H) \in \Ac$), denote by $Q = \prod_{H \in \Ac}\alpha_H \in \CC[x_1,\ldots,x_\ell]$
	the corresponding defining polynomial of $\Ac$ and let $\Xf = V \setminus \bigcup\limits_{H \in \Ac}H$ be the arrangement complement.
	The Milnor fibration of $\Ac$ is 
	\[
	Q|_{\Xf}:\Xf \to \CC^\times, v \mapsto Q(v),
	\]
	and its Milnor fiber we denote by $\Ff := Q^{-1}(1)$.
	
	Regarding the total space $\Xf$ of the fibration,
	thanks to the seminal work of Orlik and Solomon \cite{OrlikSolomon1980_OSAlgebra}, 
	its cohomology can be described completely in combinatorial terms using the intersection lattice of $\Ac$.
	Moreover, for complexified real arrangements, 
	the foundational work of Salvetti \cite{Salvetti1987_SalCpx} provided a combinatorial model in the form of the \emph{Salvetti complex}, 
	a finite regular CW complex whose homotopy type depends only on the oriented matroid of the arrangement (cf.\ Subsection \ref{ssec:SalvettiCpx}). 
	This construction has proven immensely influential and has been extended and refined in various ways 
	(see e.g., \cite{GelfandRybnikov1990_AlgTopInvOfOMs, BjoeZie1992_CombStrat, DelucchiFalk2017_EquivModelCompl}).
	However, extending such combinatorial models to Milnor fibers has posed persistent challenges. 
	While there has been substantial progress in understanding topological invariants and monodromy of Milnor fibers 
	(e.g., see \cite{CS1995_MilnorFib, FY2007_Multinets, BDS2011_FirstMilnor, Suciu2014_MilnorFibrations, 
		DS2014_Multinets, PS2017_modular, Yoshinaga2020_Milnor2Torsion}), existing approaches often rely on algebraic or Morse-theoretic techniques.
	So far, a concrete model for the homotopy type of $\Ff$ is available only in the special cases of real reflection arrangements, thanks to Brady, Falk, and Watt
	\cite{BFW2018_MilnorCoxeterNCP} and the generic case due to Orlik and Randell \cite{OrRan1993_MilnorFGeneric}.
	Despite decades of research, a full combinatorial understanding of the Milnor fiber 
	-- analogous to the combinatorial models available for the complement of an arrangement -- has remained elusive.
	
	In this paper, we provide a purely combinatorial model of the Milnor fibration for any real hyperplane arrangement. 
	Our construction is based on a new notion of poset quasi-fibration (see Definition \ref{def:pqf}),
	which serves as a combinatorial analogue of a topological fibration. 
	The total space of our poset quasi-fibration is built from a regular CW complex that is a natural PL subdivision of the Salvetti complex,
	called the \emph{tope-rank subdivision} (see Definition \ref{def:TopeRankSubdivSalvetti} and Theorem \ref{thm:rksdSubdivMap}). 
	This subdivision, in turn, arises from a canonical refinement of the dual complex of the oriented matroid associated to the arrangement 
	(see Definition \ref{def:TopeRankSubdivCovectors}, Theorem \ref{thm:rkBsdLcCWposet} and Corollary \ref{coro:HomeoRkBSubdiv}).
	
	Firstly, Theorem \ref{thm:MilnorOMPosetQuasiFib} shows that our new combinatorial map indeed has the desired
	homotopy theoretic properties. The proof uses Forman's discrete Mores theory \cite{Forman1998_DiscrMorse}.
	Moreover, it fits nicely into a homotopy commutative square with the geometric Milnor fibration (see Theorem \ref{thm:MilnorOMHomotopyEquivToMilnorGeom}).
	As a consequence we obtain our main result, Theorem \ref{thm:MilnorFiberHtEquiv}: the poset fiber of this poset quasi-fibration map
	-- a concrete finite subcomplex of the rank-subdivided Salvetti complex --
	is homotopy equivalent to the geometric Milnor fiber of the complexified arrangement.
	From this, we derive an important conclusion: 
	the homotopy type of the Milnor fiber depends solely on the combinatorial structure encoded by the oriented matroid. 
	This not only affirms a long-standing intuition in the field but also provides a concrete finite model 
	that enables the computation of topological invariants of the Milnor fiber from oriented matroid data.
	
	Moreover, our construction applies beyond realizable arrangements: 
	it works for any oriented matroid, regardless of whether it arises from a geometric configuration of hyperplanes. 
	This leads naturally to the notion of a combinatorial Milnor fibration for arbitrary oriented matroids, 
	suggesting a new direction in which to generalize classical topological notions to a broader combinatorial context. 
	In this way, our work opens the door to the study of ``non-realizable Milnor fibers''
	-- topological spaces associated to combinatorial objects that have no geometric realization, yet behave analogously to classical Milnor fibers.
	
	To our knowledge, this is the first fully combinatorial model of the Milnor fibration 
	that simultaneously captures its homotopy type and admits extension to non-realizable configurations. 
	It thus provides a novel framework that unifies and extends previous work on both the topology of hyperplane complements 
	and the subtle structure of their Milnor fibers.
	
	\medskip\noindent
	The article is organized as follows.
	In Section \ref{sec:Preliminaries} we recall all combinatorial, topological and geometric notions and results
	which we require later on.
	The next Section \ref{sec:TopeRankSubdiv} is devoted to the construction of our new subdivision of the dual oriented matroid complex
	and the Salvetti complex.
	Section \ref{sec:MilnorFibOM} then provides the definition of our combinatorial Milnor fibration for oriented matroids
	and gives the main results of our work, in particular, capturing the homotopy type of the Milnor fiber for any complexified real arrangement.


	\section{Preliminaries}
	\label{sec:Preliminaries}
	
	In this preliminary section we review all the objects and notions required 
	for our study.
	
	\subsection{Posets, regular CW complexes and PL topology}
	\label{ssec:PosetTopology}
	
	We introduce some notation and collect basic information on the combinatorial topology of posets, regular CW complexes
	and facts from piecewise-linear (PL) topology.
	Throughout this paper, all \emph{partially ordered sets} (or \emph{posets} for short) and all complexes are assumed to be finite.
	A concise reference for us is \cite[App.~4.7]{BLSWZ1999_OrientedMatroids}.
	
	We start by recalling some basic poset terminology.
	Let $P = (P,\leq)$ be a poset. Its \emph{dual} or \emph{opposite} poset $P^\dual$
	has the same ground set as $P$ but its order relation $\leq^\dual$ 
	is reversed: $x \leq^\dual y :\iff y \leq x$.
	
	To emphasize the order relation of a particular poset $P$ we occasionally use the notation $\leq_P$.
	
	We denote by $\lessdot$ the \emph{cover relations} of $P$.
	That is $x \lessdot y$ if and only if $x < y$, and for all $z \in P$:
	$x < z \leq y$ implies $z=y$.
	
	A subset $I \subseteq P$ is called an \emph{order ideal} if $y \in I$ and $x \in P$ with
	$x \leq y$ implies $x \in I$.
	Dually, $F \subseteq P$ is called an \emph{order filter} if $F^\dual$ is an order ideal in $P^\dual$.
	We write $P_{\leq y} = \{x \in P \mid x \leq y\}$ for the \emph{principal} order ideal generated by $y \in P$
	and similarly $P_{\geq x} = \{ y \in P \mid x \leq y\}$ for the principal order filter
	generated by $x \in P$.
	
	A map $f:P \to Q$ between two posets is \emph{order preserving} or a \emph{poset map}
	if for all $x \leq_P y$ $(x,y \in P)$ we have $f(x) \leq_Q f(y)$.
	If $q \in Q$ then we write $(f\downarrow q) := f^{-1}(Q_{\leq q})$ for the \emph{poset fibers} of $f$.
	
	\bigskip
	We assume familiarity with the notion of \emph{CW complexes}, e.g.\ see \cite{Hatcher2002_AT}.
	Important for our context are the following special instances.
	
	\begin{definition}[{\cite[Def.~4.7.4]{BLSWZ1999_OrientedMatroids}}]
		\label{def:regCWcpx}
		A finite \emph{regular cell complex} $\Sigma$ is a finite collection of balls $\sigma$ in a Hausdorff space
		$|\Sigma| = \bigcup\limits_{\sigma \in \Sigma}\sigma$ such that
		\begin{enumerate}[(i)]
			\item the interiors $\mathring{\sigma}$ partition $|\Sigma|$ and
			
			\item for all $\sigma \in \Sigma$ their boundary $\partial\sigma$ is a union of some members of $|\Sigma|$.
		\end{enumerate}
	\end{definition}

	A finite regular cell complex $\Sigma$ is determined up to homeomorphism by its \emph{face poset} $\Fc(\Sigma)$
	whose elements are the cells $\sigma \in \Sigma$ ordered by inclusion (\cite[Prop.~4.7.8]{BLSWZ1999_OrientedMatroids}).
	Cells of dimension zero are referred to as \emph{vertices} and we write $\vts(\sigma)$
	for the set of vertices of $\Sigma$ contained in $\sigma$.
	
	Conversely, to every poset $P$ is associated its \emph{order complex}
	$\Delta(P)$. 
	It is the simplicial complex with 
	$n$-simplices given by all chains of length $n$ in $P$, i.e.
	\[
	\Delta(P)_n := \{\{x_0,\ldots,x_n\} \mid x_i \in P\text{ with }x_0 < x_1 < \ldots < x_n\}.
	\]
	
	For a simplicial complex $\Delta$ we denote its topological realization by $|\Delta|$.
	The composition of forming the order complex and taking its realization yields the 
	\emph{simplicial realization} functor 
	\[
	|\Delta(-)|:\Pos \to \Top
	\]
	from the category $\Pos$ of posets and order preserving maps
	to the category $\Top$ of topological spaces and continuous maps. 
	A regular cell complex $\Sigma$ is naturally homeomorphic to the image of its face poset under this functor: $|\Delta(\Fc(\Sigma))| \cong |\Sigma|$.
	
	We call an order preserving map $f:P \to Q$ a homotopy equivalence
	provided $|\Delta(f)|$ is a homotopy equivalence.
	
	
	Maps between finite regular cell complexes are given as order preserving maps
	between their face posets.	
	A poset map between the face posets of
	regular cell complexes which is a homotopy equivalence in the above sense
	is indeed a homotopy equivalence of their defining topological spaces.
	
	A regular cell complex $\Sigma$ is called \emph{pure} if
	all maximal cells have the same dimension.
	If for all $\sigma, \tau \in \Sigma$ with nonempty intersection
	we have $\sigma\cap\tau \in \Sigma$ then $\Sigma$ has the \emph{intersection property}.
	
	A subset $\Gamma \subseteq \Sigma$ is called a \emph{(closed) subcomplex}
	if with a cell $\sigma \in \Gamma$ all of its faces also belong to $\Gamma$,
	i.e.\ $\Gamma$ is an order ideal of $\Sigma$.
	For a subset of cells $S \subseteq \Sigma$
	we denote by $\Sigma(S) := \{\tau \in \Sigma \mid \tau \subseteq \sigma$ for a $\sigma \in S\}$ the subcomplex
	of all faces of cells in $S$.
	In poset terminology, $\Fc(\Sigma(S)) = \bigcup_{\sigma \in S} \Fc(\Sigma)_{\leq \sigma} $ is the order ideal generated by $S$.
	For a cell $\sigma \in \Sigma$ we denote by $\partial\sigma$ the subcomplex consisting of
	all proper faces of $\sigma$.
	
	We have the following special class of posets which exactly correspond to regular cell complexes.
	
	\begin{definition}[{\cite{Bjoerner1984_CWPosets}}]
		\label{def:CW-poset}
		Let $P$ be a graded poset with grading $\rho:P \to \ZZ_{\geq 0}$.
		Then $P$ is a \emph{CW-poset} if for all $x \in P$ the order complex $|\Delta(P_{<x})|$ is homeomorphic to
		a $(\rho(x)-1)$-sphere. 
	\end{definition}
	
	\begin{lemma}[{\cite{McCrory1975_ConeCpxs}, \cite{Bjoerner1984_CWPosets}}]
		A poset $P$ is a CW-poset if and only if it is the face poset of a regular cell complex.
	\end{lemma}
	
	\begin{definition}
		\label{def:SubdivPL homeo}
		\begin{enumerate}[(a)]
			\item
			Let $\Sigma$ and $\Gamma$ be two regular cell complexes. Then
			$\Gamma$ is a \emph{subdivision} of $\Sigma$ if $|\Gamma| = |\Sigma|$ and
			every closed cell of $\Gamma$ is a subset of some closed cell of $\Sigma$.
			In that case we write $\Gamma \subdiv \Sigma$. 
			
			\item 
			Let $\Delta$ and $\Delta'$ be two simplicial complexes. Then $\Delta$
			is \emph{PL homeomorphic} to $\Delta'$ if there are simplicial subdivisions $\Gamma \subdiv \Delta$,
			$\Gamma' \subdiv \Delta'$ and a simplicial isomorphism $\Gamma \to \Gamma'$
			or in other words, $\Delta$ and $\Delta'$ are PL homeomorphic if and only if they have common simplicial subdivisions.
			
			\item 
			Let $P$ and $Q$ be two posets. Then $P$ is PL homeomorphic to $Q$ if
			$\Delta(P)$ is PL homeomorphic to $\Delta(Q)$.
			
			\item
			Let $\Sigma$ and $\Gamma$ be two regular cell complexes. 
			Then $\Sigma$ is PL homeomorphic to $\Gamma$ if $\Fc(\Sigma)$ is PL homeomorphic to $\Fc(\Gamma)$.
			
			\item
			A simplicial complex $\Delta$ is called a
			\begin{enumerate}[(i)]
				\item \emph{PL $d$-ball} if $\Delta$ is PL homeomorphic to the $d$-simplex,
				\item \emph{PL $d$-sphere} if $\Delta$ is PL homeomorphic to the boundary of the $(d+1)$-simplex.
			\end{enumerate}
			More generally, a regular cell complex $\Sigma$ is a
			\begin{enumerate}[(i)]
				\setcounter{enumii}{2}
				\item \emph{PL $d$-ball} if $\Delta(\Fc(\Sigma))$ is PL homeomorphic to the $d$-simplex,
				\item \emph{PL $d$-sphere} if $\Delta(\Fc(\Sigma))$ is PL homeomorphic to the boundary of the $(d+1)$-simplex.
			\end{enumerate}	
		\end{enumerate}
	\end{definition}
	
	\begin{theorem}[{\cite[Thm.~4.7.21(i),(ii),(v)]{BLSWZ1999_OrientedMatroids}}]\hfill
		\label{thm:PLBallsAndSpheres}
		\begin{enumerate}[(i)]
			\item The union of two PL $d$-balls, whose intersection is a PL $(d-1)$-ball lying in the boundary of each, is a PL $d$-ball.
			\item The union of two PL $d$-balls, which intersect along their entire boundaries, is a PL $d$-sphere.
			\item The cone over a PL $d$-sphere is a PL $(d+1)$-ball.
		\end{enumerate}
	\end{theorem}
	
	An important fact about PL spheres is that they admit \emph{dual cell structures}.
	
	\begin{lemma}[{\cite[Prop.~4.7.26(iii),(iv)]{BLSWZ1999_OrientedMatroids}}]
		\label{lem:PL sphereDual}
		Let $\Sigma$ be a regular cell complex which is a PL $d$-sphere. Then:
		\begin{enumerate}[(i)]
			\item 
			every closed cell $\sigma \in \Sigma$ is a PL ball.
			
			\item 
			there is a regular cell complex $\Sigma^\dual$, also a PL $d$-sphere,
			with $|\Sigma| = |\Sigma^\dual|$ and $\Fc(\Sigma^\dual) = \Fc(\Sigma)^\dual$.
		\end{enumerate}
	\end{lemma}
	
	\begin{definition}[Cones]
		\label{def:ConePoset}
		For a poset $P$ denote by 
		\begin{enumerate}[(i)]
			\item $P\ast v$ the poset where one maximal element $v$ is added to $P$ with $x < v$ for all $x\in P$, and
			\item $CP$ the \emph{cone poset} over $P$ which
			is defined as $P \times I \sqcup \{v\}$ where $I$ is the poset $0 < 1$ with 
			the usual product order on $P \times I$ and $v \leq (x,1)$ for all $x \in P$.	
		\end{enumerate}
	\end{definition}
	
	\begin{lemma}
		\label{lem:ConeCWPoset}
		If $P$ is the face poset of a regular cell complex $\Sigma$ then so is $CP$
		for which we write $C\Sigma$. Moreover, $CP$ is PL homeomorphic to $P\ast v$.
	\end{lemma}
	\begin{proof}
		See eg.\ \cite[Prop.1.1(ii)]{McCrory1975_ConeCpxs}.
	\end{proof}
	
	\begin{lemma}[{\cite[Prop.~4.7.13]{BLSWZ1999_OrientedMatroids}}]
		\label{lem:IntersectionProperty}
		If $\Sigma$ has the intersection property, then
		for $\sigma,\tau \in \Sigma$ we have
		\begin{enumerate}[(i)]
			\item $\sigma = \tau$ if and only if $\vts(\sigma) = \vts(\tau)$,
			\item $\sigma \subseteq \tau$ if and only if $\vts(\sigma) \subseteq \vts(\tau)$.
		\end{enumerate}
		Consequently, $\Fc(\Sigma)$ is isomorphic to $\{\vts(\sigma)\mid \sigma \in \Sigma\}$ ordered by inclusion.
	\end{lemma}

	
	\subsection{Combinatorial models of fibrations}
	\label{ssec:PosetQuasiFib}
	
	Firstly, we record the following theorem from Quillen's foundational
	work \cite{Quillen1973_KTheory1} on algebraic $K$-Theory, originally expressed in terms
	of small categories and functors.
	We state it in its special incarnation in terms of posets and order preserving maps.
	
	\begin{theorem}[Theorem B for posets {\cite[pp.~89]{Quillen1973_KTheory1}}]
		\label{thm:TheoremB}
		Let $f:P \to Q$ be an order preserving map between posets.
		If for all $a \leq b$ $(a,b \in Q)$ the inclusion
		$(f\downarrow a) \hookrightarrow (f\downarrow b)$ is a homotopy equivalence,
		then for $x \in P$ with $f(x)=a$ the homotopy fiber
		$\HoFib(|\Delta(f)|,a)$ of $|\Delta(f)|$ is homotopy equivalent to $|\Delta(f\downarrow a)|$.
		
		Consequently, we have a long exact homotopy sequence		
		\begin{center}
			\begin{tikzcd}
				&\ldots \ar[r] &\pi_{i+1}(|\Delta(Q)|,a) \ar[r] &\pi_{i}(|\Delta(f\downarrow a)|,x) \\
				&	\ar[r, "\pi_i(|\Delta(j)|)"] &\pi_{i}(|\Delta(P)|,x) \ar[r, "\pi_i(|\Delta(f)|)"] &\pi_{i}(|\Delta(Q)|,a) \ar[r] &\ldots
			\end{tikzcd}
		\end{center}
		where $j:(f\downarrow a) \hookrightarrow P$ is the natural inclusion.
	\end{theorem}
	
	The following notion is motivated by the preceding theorem.
	It was firstly used in a combinatorial setting in \cite{Bab1993_CombFlagSpace} (see also \cite[App.~B]{AD2002_CombGrassmanians}) 
	and recently for the study of supersolvable oriented matroids in \cite{Mue24_ModFlatsOMs}.
	
	\begin{definition}
		\label{def:pqf}
		Let $f:P \to Q$ be a poset map such that for all $a \leq b$ $(a,b \in Q)$
		the inclusion $(f\downarrow a) \hookrightarrow (f\downarrow b)$ 
		is a homotopy equivalence.
		Then $f$ is called a \emph{poset quasi-fibration}.
	\end{definition}
	
	\begin{definition}
		\label{def:ModelFibration}
		Let $\varphi:X \to Y$ be a continuous map between topological spaces and $f:P \to Q$ a poset map.
		Then we say that $f$ is a \emph{combinatorial model} for $\varphi$ if we have homotopy equivalences $|\Delta(P)| \isom X$,
		$|\Delta(Q)| \isom Y$ which fit into a (homotopy) commutative diagram:
		\begin{center}
			\begin{tikzcd}[column sep=16mm, scale cd=1]
				{|\Delta(P)|} \ar[r, "{|\Delta(f)|}"]\ar[d,"\isom", swap] & {|\Delta(Q)|} \ar[d,"\isom"]\\
				X \ar[r,"\varphi"] & Y.
			\end{tikzcd}
		\end{center}
	\end{definition}
	
	With the previous notion we have the following Lemma.
	\begin{lemma}
		\label{lem:CombFiberHtEquiv}
		Assume that $\varphi:X \to Y$ is a topological fibration with a combinatorial model $f:P \to Q$.
		Assume further that $f$ is also a poset quasi-fibration.
		Then the fiber $F = \varphi^{-1}(y)$ for $y \in Y$ is homotopy equivalent to any poset fiber $|\Delta(f\downarrow q)|, (q \in Q)$.
	\end{lemma}
	\begin{proof}
		Firstly, since $\varphi$ is a fibration, we may assume that the diagram is strictly commutative
		after replacing the homotopy equivalence $|\Delta(P)| \to X$ with another homotopic one, 
		say $h:|\Delta(P)| \to X$ (see \cite[5.29]{Strom11_HomotopyThy}).
		
		Since the map $f$ is a poset quasi-fibration it has a corresponding long
		exact homotopy sequence by Theorem \ref{thm:TheoremB} which includes $|\Delta(f\downarrow q)| =: F^c$ as the homotopy fiber of $|\Delta(f)|$.
		The topological fibration $\varphi$ also has a long exact homotopy sequence.
		From the commutative diagram we get an induced map $h|:F^c \to F$ 
		between the fibers and also a map between these two long exact sequences.
		The five-lemma yields that the maps $\pi_i(h|):\pi_i(F^c) \to \pi_i(F)$ are isomorphisms for all $i$.
		By Whitehead's Theorem, $h|$ is a homotopy equivalence.
		
		Alternatively: Use homotopy limits, i.e.\ homotopy fibers of both maps $|\Delta(f)|$ and $\varphi$ (e.g., see \cite[Thm.~6.76]{Strom11_HomotopyThy}).
	\end{proof}

	\subsection{Oriented Matroids}
	\label{ssec:OrientedMatroids}
	
	We recall some basics from oriented matroid theory.
	An oriented matroid can be regarded as a combinatorial abstraction of
	a real hyperplane arrangement, i.e.\ a finite set of hyperplanes in a finite dimensional real vector space.
	The standard reference is the book by Bj\"orner et al.\ \cite{BLSWZ1999_OrientedMatroids}.
	In our study, we adopt the point of view that oriented matroids constitute a
	special class of regular cell complexes, see Remark~\ref{rem:OMTopRep} below.
	
	First, let us recall some notation for sign vectors.
	Let $E$ be some finite set and consider a sign vector $\sigma \in \{+,-,0\}^E$.
	We denote by $z(\sigma) := \{ e \in E \mid \sigma_e = 0\}$ its \emph{zero set}.
	The \emph{opposite} vector $-\sigma$ is defined by
	\[
	(-\sigma)_e := \begin{cases}
		- & \text{ if } \sigma_e = +, \\
		+ & \text{ if } \sigma_e = -, \\
		0 & \text{ if } \sigma_e = 0.
	\end{cases}
	\]
	The \emph{composition} $\sigma \circ \tau$ of two sign vectors $\sigma, \tau \in \{+,-,0\}^E$  
	is defined by
	\[
	(\sigma \circ \tau)_e := \begin{cases}
		\sigma_e &\text{ if } \sigma_e \neq 0, \\
		\tau_e & \text{ else},
	\end{cases}
	\]
	and their \emph{separating set} $S(\sigma,\tau)$ is defined as
	\[
	S(\sigma,\tau) := \{e \in E \mid \sigma_e = -\tau_e \neq 0\}.
	\]
	If $A \subseteq E$ and $\sigma \in \{+,-,0\}^E$ we write $\sigma|_A := (\sigma_e \mid e \in A) \in \{+,-,0\}^A$
	for its restriction to $A$.
	
	\begin{definition}
		\label{def:OMAxioms}
		An \emph{oriented matroid} $\OM = (E,\Lc)$ is given by a finite
		ground set $E$ and a subset of sign vectors $\Lc \subseteq \{+,-,0\}^E$ called \emph{covectors} of $\OM$,
		subject to satisfying the following axioms:
		\begin{enumerate}[(1)]
			\item $\mathbf{0} := (0,0,\ldots,0) \in \Lc$,
			\item if $\sigma \in \Lc$, then $-\sigma\in \Lc$,
			\item if $\sigma, \tau \in \Lc$, then $\sigma \circ \tau \in \Lc$,
			\item if $\sigma, \tau \in \Lc$ and $e \in S(\sigma,\tau)$
			there exists a $\eta \in \Lc$ such that $\eta_e = 0$ and
			$\eta_f = (\sigma \circ \tau)_f = (\tau \circ \sigma)_f$ for all $f \in E \setminus S(\sigma,\tau)$.
		\end{enumerate}
		
		Let $\leq$ be the partial order on the set $\{+,-,0\}$ defined by $0 < +$, $0 < -$, with
		$+$ and $-$ incomparable. This yields the product partial order on $\{+,-,0\}^E$ in which sign vectors
		are compared component wise.
		The \emph{face poset} or covector poset $(\Lc, \leq)$ of $\OM$ is obtained by considering the induced partial order on the
		subset $\Lc \subseteq \{+,-,0\}^E$.
		
		The \emph{rank} of $\OM$ is defined as the length of a maximal chain in $\Lc$.
		
		The maximal elements in $\Lc$ are called \emph{topes}; they are denoted by $\Tc = \Tc(\OM)$.
	\end{definition}
	
	Note that this is only one of several equivalent ways to define oriented matroids, cf.\ \cite[Ch.~3]{BLSWZ1999_OrientedMatroids}.
	
	An element $e \in E$ is called a \emph{loop} of the oriented matroid $\OM =(E,\Lc)$
	if $\sigma_e = 0$ for all $\sigma \in \Lc$.
	Two elements $e,f \in E$ which are not loops are \emph{parallel} if $\sigma_e = 0$
	if and only if $\sigma_f=0$ for all $\sigma \in \Lc$. 
	
	In the rest of the paper, without loss of generality, we always assume an oriented matroid to
	be simple, i.e.\ it contains no loops or parallel elements.
	This is justified by the fact that the covector poset of an oriented matroid and its simplification are isomorphic, 
	cf.\ \cite[Lem.~4.1.11]{BLSWZ1999_OrientedMatroids}.
	
	We have the following important examples of oriented matroids.
	\begin{example}
		\label{ex:OMRealArrangement}
		(a)
		Let $\Ac$ be an arrangement of hyperplanes in a real vector space $V \isom \RR^\ell$.
		For $H \in \Ac$ choose defining linear forms $\alpha_H \in V^*$ with $\ker(\alpha_H) = H$.
		Then $\OM(\Ac) = (\Ac,\Lc(\Ac))$ with 
		\[
		\Lc(\Ac) := \{ (\sgn(\alpha_H(v)) \mid H \in \Ac) \mid v \in V\} \subseteq \{+,-,0\}^\Ac
		\]
		is an oriented matroid.
		
		(b) The dual poset $\Lc^\dual(\Ac)$ is isomorphic to the face lattice
		of the \emph{Zonotope} $\Zc(\Ac)$ associated to $\Ac$, which is defined as the Minkowski sum $\Zc(\Ac) := \sum_{H\in\Ac}[-\alpha_H,\alpha_H] \subseteq V^*$.
	\end{example}
	
	\begin{remark}
		\label{rem:OMTopRep}
		For our purposes, oriented matroids are best thought of as special instances of regular cell complexes.
		This is justified by the topological representation theorem due to Folkman and Lawrence \cite{FolkmanLawrence1978_OMs},
		see also \cite[Thm.~5.2.1]{BLSWZ1999_OrientedMatroids}.
		Its core implication that the reduced covector poset $\Lc\setminus\{\Zero\}$ 
		is the face poset of a regular cell decomposition of a sphere induced by
		an arrangement of tamely embedded codimension one subspheres suffices for our study (see also Theorem~\ref{thm:OMCovectorShellable}).
		Moreover, $\Lc\setminus\{\Zero\}$ is a PL sphere, so $\Lc^\dual$ represents a PL regular cell decomposition
		of a ball by Lemma \ref{lem:PL sphereDual}(ii) (cf.\ Figure \ref{fig:A_OMcpxs}).
	\end{remark}
	
	\bigskip
	\noindent
	There is another poset associated to an oriented matroid $(E,\Lc)$.
	\begin{definition}
		The \emph{geometric lattice} $L(\OM)$ of $\OM$ is defined as
		\[
		L(\OM) := \{ z(\sigma) \mid \sigma \in \Lc \} \subseteq 2^E,
		\]
		with partial order given by inclusion.
		Elements of $L(\OM)$ are called \emph{flats}.
		For two elements $X,Y \in L(\OM) $, their join is $X \vee Y = \inf\{Z \in L(\OM) \mid X\cup Y \subseteq Z\}$,
		their meet is $X \wedge Y = X \cap Y$.
		
		Moreover, for $X,Y \in L=L(\OM)$ we write $L_X := L_{\leq X}$ for the lattice of the \emph{localization} at $X$,
		$L^Y := L_{\geq Y}$ for the lattice of the \emph{contraction} to $Y$ 
		and $L_X^Y := [X,Y] := L_{\leq X} \cap L_{\geq Y}$ for an interval.
	\end{definition}
	
	The lattice $L=L(\OM)$ is graded by \emph{rank}, i.e.\
	the rank of $X \in L$ is the length of a maximal chain in $L_{\leq X}$.
	
	Note that the map $z:\Lc \to L$ is a cover and rank preserving, order reversing surjection \cite[Prop.~4.1.13]{BLSWZ1999_OrientedMatroids}.
	The associated order preserving map $\Lc^\vee \to L$ is also denoted by $z$.
	

	\bigskip
	\noindent
	Recall that we denote the topes of an oriented matroid $\OM = (E,\Lc)$,
	i.e.\ the maximal elements of $\Lc$ by $\Tc$.
	As remarked earlier, without loss of generality, we assume that $\OM$ is simple.
	
	\begin{definition}
		\label{def:TopePoset}
		Let $B \in \Tc$.
		We define a partial order on $\Tc$ by
		\[
		R \leq_B T :\iff S(B,R) \subseteq S(B,T).
		\]
		The resulting ranked poset $\Tc_B = (\Tc,\leq_B)$  with rank function $\rk_B(T) := |S(B,T)|$
		is called the \emph{tope poset} with respect to $B$.
	\end{definition}
	
	
	\subsection{Shellability and subcomplexes}
	\label{ssec:Shellability}
	
	We briefly elaborate on shellability of regular cell complexes, cf.\ \cite[App.~4.7]{BLSWZ1999_OrientedMatroids}.

	\begin{definition}
		\label{def:ShellableCpx}
		Let $\Sigma$ be a pure $d$-dimensional regular cell complex.
		A linear ordering $\sigma_1,\sigma_2,\ldots,\sigma_t$ of its maximal cells is called
		a \emph{shelling} if either $d=0$, or $d\geq 1$ and the following conditions are satisfied:
		\begin{enumerate}[(i)]
			\item $\delta\sigma_j \cap (\bigcup_{i=1}^{j-1}\delta\sigma_i)$ is pure of dimension $d-1$ for $2\leq j \leq t$,
			
			\item $\delta\sigma_j$ has a shelling in which the $(d-1)$-cells of $\delta\sigma_j \cap (\bigcup_{i=1}^{j-1}\delta\sigma_i)$ 
			come first for $2 \leq j \leq t$,
			
			\item $\delta\sigma_1$ has a shelling.
		\end{enumerate} 
		A complex which admits a shelling is called \emph{shellable}.
	\end{definition}
	
	For our investigation, the principal example of a shellable regular cell complex is the
	covector complex $\Lc$ of an oriented matroid.
	More precisely, we have the following theorem.
	\begin{theorem}[{\cite[Thm.~4.3.3]{BLSWZ1999_OrientedMatroids}}]
		\label{thm:OMCovectorShellable}
		The reduced covector poset $\Lc \setminus \{\Zero\}$ of an oriented matroid of rank $r$ is isomorphic to
		the face poset of a shellable and hence PL regular cell decomposition of the $(r-1)$-sphere.
		Furthermore, for any $B \in \Tc$, every linear extension of $\Tc_B$ is a shelling of $\Lc \setminus \{\Zero\}$.
	\end{theorem}
	
	Moreover, we have
	\begin{proposition}[{\cite[Prop.~4.3.1]{BLSWZ1999_OrientedMatroids}}]
		\label{prop:ShellingTopeSubcpx}
		Let $R \in \Tc$ be a tope. Then $\Lc_{< R} \setminus \{\Zero\}$ is a shellable regular cell decomposition of the $(r-2)$-sphere.
		Furthermore, every maximal chain $R = B_0 < R_1 < \ldots < R_n = -R$ in the tope poset $\Tc_R$ yields a
		total order $<$ on the ground set of $E$ and thus a total order on the facets of $\Lc_{< R} \setminus \{\Zero\}$ by restriction which is a shelling.
	\end{proposition}
	
	We list the following useful fact, cf.\ \cite[Prop.~4.7.26]{BLSWZ1999_OrientedMatroids}.
	\begin{proposition}
		\label{prop:ShellingSphereSubsequenceBall}
		Let $\Sigma$ be a regular cell decomposition of the $d$-sphere which has a shelling
		$\sigma_1,\ldots,\sigma_t$.
		Then for $k <t$ the subcomplex $\Sigma(\{\sigma_1,\ldots,\sigma_k\})$
		is a shellable and hence PL $d$-ball with shelling $\sigma_1,\ldots,\sigma_k$.
	\end{proposition}

	\begin{proposition}[{\cite[Thm.~4.1.14]{BLSWZ1999_OrientedMatroids}}]
		\label{prop:OMCpxIntersectionProperty}
		The regular cell complex $\Lc\setminus\{\Zero\}$ has the intersection property and so does $\Lc^\dual$.
		Equivalently, both $\Lc$ and $\Lc^\dual$ adjoint with an additional unique maximal respectively minimal element are atomic lattices. 
	\end{proposition}
	
	
	\subsection{The Salvetti complex of an oriented matroid}
	\label{ssec:SalvettiCpx}
	
	We start by defining a certain poset which was first constructed by
	Salvetti \cite{Salvetti1987_SalCpx} for an oriented matroid $\OM(\Ac)$
	associated to a real hyperplane arrangement $\Ac$.
	Gel'fand and Rybnikov \cite{GelfandRybnikov1990_AlgTopInvOfOMs} observed,
	that the same construction applies more generally to oriented matroids.
	
	\begin{definition}
		\label{def:SalvettiCpxOM}
		Let $\Lc$ be the poset of covectors of an oriented matroid $\OM$.
		Then the \emph{Salvetti poset} $\Sc = \Sc(\OM)$ is defined
		as
		\[
		\Sc := \{ (\sigma,T) \mid T \in \Tc\text{ and } \sigma \in \Lc_{\leq T} \} \subseteq \Lc \times \Tc,
		\]
		with partial order
		\[
		(\sigma, T) \leq_\Sc (\tau, R) :\iff \sigma \geq_\Lc \tau \text{ and } \sigma \circ R = T. 
		\]
	\end{definition}
	
	The following theorem due to Salvetti \cite{Salvetti1987_SalCpx} 
	is a fundamental result in the homotopy theory of complements
	of complexified real arrangements. 
	
	\begin{theorem}[{\cite{Salvetti1987_SalCpx}}]
		\label{thm:SalvettiHoEquiv}
		The Salvetti poset is the face poset of a regular cell complex which is also denoted by $\Sc$.
		If $\OM=\OM(\Ac)$ is the oriented matroid associated to a hyperplane arrangement $\Ac$
		in a real vector space $V$,
		then the Salvetti complex of $\OM$ (or $\Ac$) is homotopy equivalent to the complement of the complexified arrangement:
		\[
		|\Sc| \cong V\otimes\CC \setminus \left( \bigcup\limits_{H \in \Ac} H\otimes\CC \right).
		\]
	\end{theorem}
	
	The subcomplex of $\Sc$ consisting of all the faces of a maximal cell can be identified with the dual covector complex.
	\begin{lemma}
		\label{lem:SalPrincipialIdealCovectors}
		Let $(\Zero,T) \in \Sc$ be a maximal element of $\Sc$.
		Then $\Sc_{\leq (\Zero,T)} \isom \Lc^\dual$.
	\end{lemma}
	\begin{proof}
		The maps $\Sc_{\leq (\Zero,T)} \to \Lc^\dual, (\sigma,R) \mapsto \sigma$ 
		and $\Lc^\dual \to \Sc_{\leq (\Zero,T)}, \sigma \mapsto (\sigma,\sigma\circ T)$ are 
		mutually inverse and order preserving.
	\end{proof}
	
	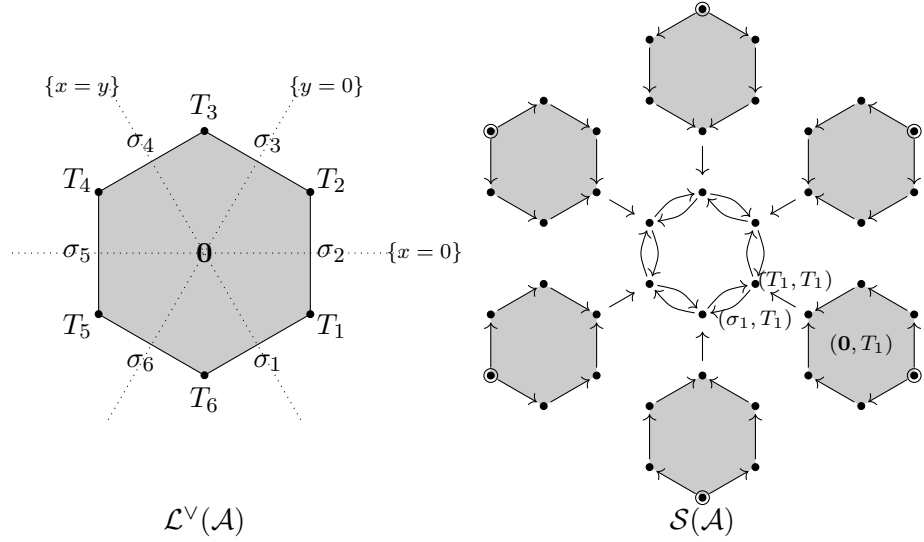
\begin{figure}
		\begin{tikzpicture}[scale=0.65]
			\draw[dotted] (4.,0.) -- (-4.,0.);  
			\node at (4.5,0.) {\tiny $\{x=0\}$}; 		
			\draw[dotted] (2.,3.4641016151377548) -- (-2.,-3.4641016151377548);  
			\node at (2.5,3.4) {\tiny $\{y=0\}$}; 		
			\draw[dotted] (-2.,3.4641016151377548) -- (2.,-3.4641016151377548);  
			\node at (-2.5,3.4) {\tiny $\{x=y\}$}; 		
			
			\node (l1) at (0.0,-2.) {};
			\node (l2) at (1.7320508075688774,-1.) {};
			\node (l3) at (1.7320508075688774,1.) {};
			\node (l4) at (0.0,2.) {};
			\node (l5) at (-1.7320508075688774,1.) {};
			\node (l6) at (-1.7320508075688774,-1.) {};
			\node (zero) at (0.0,0.) {};
			
			\node at ($1.5*(l1)$) {\small$T_6$};
			\node at ($1.5*(l2)$) {\small$T_1$};
			\node at ($1.5*(l3)$) {\small$T_2$};
			\node at ($1.5*(l4)$) {\small$T_3$};
			\node at ($1.5*(l5)$) {\small$T_4$};
			\node at ($1.5*(l6)$) {\small$T_5$};

			\node at ($0.75*(l1)+0.75*(l2)$) {\small$\sigma_1$};
			\node at ($0.75*(l2)+0.75*(l3)$) {\small$\sigma_2$};
			\node at ($0.75*(l3)+0.75*(l4)$) {\small$\sigma_3$};
			\node at ($0.75*(l4)+0.75*(l5)$) {\small$\sigma_4$};
			\node at ($0.75*(l5)+0.75*(l6)$) {\small$\sigma_5$};
			\node at ($0.75*(l6)+0.75*(l1)$) {\small$\sigma_6$};

			\draw ($1.25*(l1)$) -- ($1.25*(l2)$) -- ($1.25*(l3)$) -- ($1.25*(l4)$) -- ($1.25*(l5)$) -- ($1.25*(l6)$) -- ($1.25*(l1)$);
			
			\fill[color=black, opacity=0.2] ($1.25*(l1)$) -- ($1.25*(l2)$) -- ($1.25*(l3)$) -- ($1.25*(l4)$) -- ($1.25*(l5)$) -- ($1.25*(l6)$) -- ($1.25*(l1)$);
			
			\node[color=black] at ($0.25*1.25*(l2) + 0.25*1.25*(l3) + 0.25*1.25*(l5) + 0.25*1.25*(l6)$) {\small$\Zero$};
			
			\filldraw ($1.25*(l1)$) circle[radius=2pt];
			\filldraw ($1.25*(l2)$) circle[radius=2pt];
			\filldraw ($1.25*(l3)$) circle[radius=2pt];
			\filldraw ($1.25*(l4)$) circle[radius=2pt];
			\filldraw ($1.25*(l5)$) circle[radius=2pt];
			\filldraw ($1.25*(l6)$) circle[radius=2pt];
			
			\node at (0,-5.5) {\small$\Lc^\dual(\Ac)$};
		\end{tikzpicture}
		\begin{tikzpicture}[scale=0.65]		
			
			\node (l1) at (1.7320508075688774,-1.) {};
			\node (l2) at (1.7320508075688774,1.) {};
			\node (l3) at (0,2.) {};
			\node (l4) at (-1.7320508075688774,1.) {};
			\node (l5) at (-1.7320508075688774,-1.) {};
			\node (l6) at (0,-2.) {};
			\node (zero) at (0,0) {};
			
			
			\node (l1_0) at ($0.5*(1.7320508075688774,-1.)$) {};
			\node (l2_0) at ($0.5*(1.7320508075688774,1.)$) {};
			\node (l3_0) at ($0.5*(0,2.)$) {};
			\node (l4_0) at ($0.5*(-1.7320508075688774,1.)$){};
			\node (l5_0) at ($0.5*(-1.7320508075688774,-1.)$) {};
			\node (l6_0) at ($0.5*(0,-2.)$) {};
			
			\draw[->,shorten >=3pt, shorten <=3pt] ($1.25*(l1_0)$) .. controls ($0.5*(l1_0)+0.5*(l2_0)$)  .. ($1.25*(l2_0)$);;
			\draw[->,shorten >=3pt, shorten <=3pt] ($1.25*(l2_0)$) .. controls ($0.75*(l1_0)+0.75*(l2_0)$)  .. ($1.25*(l1_0)$);;
			\draw[->,shorten >=3pt, shorten <=3pt] ($1.25*(l2_0)$) .. controls ($0.5*(l2_0)+0.5*(l3_0)$)  .. ($1.25*(l3_0)$);;
			\draw[->,shorten >=3pt, shorten <=3pt] ($1.25*(l3_0)$) .. controls ($0.75*(l2_0)+0.75*(l3_0)$)  .. ($1.25*(l2_0)$);;
			\draw[->,shorten >=3pt, shorten <=3pt] ($1.25*(l3_0)$) .. controls ($0.5*(l3_0)+0.5*(l4_0)$)  .. ($1.25*(l4_0)$);;
			\draw[->,shorten >=3pt, shorten <=3pt] ($1.25*(l4_0)$) .. controls ($0.75*(l3_0)+0.75*(l4_0)$)  .. ($1.25*(l3_0)$);;
			\draw[->,shorten >=3pt, shorten <=3pt] ($1.25*(l4_0)$) .. controls ($0.5*(l4_0)+0.5*(l5_0)$)  .. ($1.25*(l5_0)$);;
			\draw[->,shorten >=3pt, shorten <=3pt] ($1.25*(l5_0)$) .. controls ($0.75*(l4_0)+0.75*(l5_0)$)  .. ($1.25*(l4_0)$);;
			\draw[->,shorten >=3pt, shorten <=3pt] ($1.25*(l5_0)$) .. controls ($0.5*(l5_0)+0.5*(l6_0)$)  .. ($1.25*(l6_0)$);;
			\draw[->,shorten >=3pt, shorten <=3pt] ($1.25*(l6_0)$) .. controls ($0.75*(l5_0)+0.75*(l6_0)$)  .. ($1.25*(l5_0)$);;
			\draw[->,shorten >=3pt, shorten <=3pt] ($1.25*(l6_0)$) .. controls ($0.5*(l6_0)+0.5*(l1_0)$)  .. ($1.25*(l1_0)$);;
			\draw[->,shorten >=3pt, shorten <=3pt] ($1.25*(l1_0)$) .. controls ($0.75*(l6_0)+0.75*(l1_0)$)  .. ($1.25*(l6_0)$);;
			\filldraw ($1.25*(l1_0)$) circle[radius=2pt];
			\filldraw ($1.25*(l2_0)$) circle[radius=2pt];
			\filldraw ($1.25*(l3_0)$) circle[radius=2pt];
			\filldraw ($1.25*(l4_0)$) circle[radius=2pt];
			\filldraw ($1.25*(l5_0)$) circle[radius=2pt];
			\filldraw ($1.25*(l6_0)$) circle[radius=2pt];
			
			\node at ($1.5*(l1_0)+(0.6,0.2)$) {\tiny{$(T_1,T_1)$}};
			\node at ($0.9*(l1_0) + 0.9*(l6_0)+(0.3,0)$) {\tiny{$(\sigma_1,T_1)$}};
			
			
			\node (l1_1) at ($1.5*(l1)+0.5*(1.7320508075688774,-1.)$) {};
			\node (l2_1) at ($1.5*(l1)+0.5*(1.7320508075688774,1.)$) {};
			\node (l3_1) at ($1.5*(l1)+0.5*(0,2.)$) {};
			\node (l4_1) at ($1.5*(l1)+0.5*(-1.7320508075688774,1.)$){};
			\node (l5_1) at ($1.5*(l1)+0.5*(-1.7320508075688774,-1.)$) {};
			\node (l6_1) at ($1.5*(l1)+0.5*(0,-2.)$) {};
			
			\draw[->,shorten >=3pt, shorten <=3pt] ($1.25*(l1_1)$) -- ($1.25*(l2_1)$);
			\draw[->,shorten >=3pt, shorten <=3pt] ($1.25*(l2_1)$) -- ($1.25*(l3_1)$);
			\draw[->,shorten >=3pt, shorten <=3pt] ($1.25*(l3_1)$) -- ($1.25*(l4_1)$);
			\draw[->,shorten >=3pt, shorten <=3pt] ($1.25*(l1_1)$) -- ($1.25*(l6_1)$);
			\draw[->,shorten >=3pt, shorten <=3pt] ($1.25*(l6_1)$) -- ($1.25*(l5_1)$);
			\draw[->,shorten >=3pt, shorten <=3pt] ($1.25*(l5_1)$) -- ($1.25*(l4_1)$);
			\fill[color=black, opacity=0.2] ($1.25*(l1_1)$) -- ($1.25*(l2_1)$) -- ($1.25*(l3_1)$) -- ($1.25*(l4_1)$) -- ($1.25*(l5_1)$) -- ($1.25*(l6_1)$) -- ($1.25*(l1_1)$);
			\filldraw ($1.25*(l1_1)$) circle[radius=2pt];
			\draw ($1.25*(l1_1)$) circle[radius=4pt];
			\filldraw ($1.25*(l2_1)$) circle[radius=2pt];
			\filldraw ($1.25*(l3_1)$) circle[radius=2pt];
			\filldraw ($1.25*(l4_1)$) circle[radius=2pt];
			\filldraw ($1.25*(l5_1)$) circle[radius=2pt];
			\filldraw ($1.25*(l6_1)$) circle[radius=2pt];

			\draw[->] ($1.1*(l1)$) -- ($0.8*(l1)$);
			
			
			\node (l1_2) at ($1.5*(l2)+0.5*(1.7320508075688774,-1.)$) {};
			\node (l2_2) at ($1.5*(l2)+0.5*(1.7320508075688774,1.)$) {};
			\node (l3_2) at ($1.5*(l2)+0.5*(0,2.)$) {};
			\node (l4_2) at ($1.5*(l2)+0.5*(-1.7320508075688774,1.)$){};
			\node (l5_2) at ($1.5*(l2)+0.5*(-1.7320508075688774,-1.)$) {};
			\node (l6_2) at ($1.5*(l2)+0.5*(0,-2.)$) {};
			
			\draw[->,shorten >=3pt, shorten <=3pt] ($1.25*(l2_2)$) -- ($1.25*(l3_2)$);
			\draw[->,shorten >=3pt, shorten <=3pt] ($1.25*(l3_2)$) -- ($1.25*(l4_2)$);
			\draw[->,shorten >=3pt, shorten <=3pt] ($1.25*(l4_2)$) -- ($1.25*(l5_2)$);
			\draw[->,shorten >=3pt, shorten <=3pt] ($1.25*(l2_2)$) -- ($1.25*(l1_2)$);
			\draw[->,shorten >=3pt, shorten <=3pt] ($1.25*(l1_2)$) -- ($1.25*(l6_2)$);
			\draw[->,shorten >=3pt, shorten <=3pt] ($1.25*(l6_2)$) -- ($1.25*(l5_2)$);
			\fill[color=black, opacity=0.2] ($1.25*(l1_2)$) -- ($1.25*(l2_2)$) -- ($1.25*(l3_2)$) -- ($1.25*(l4_2)$) -- ($1.25*(l5_2)$) -- ($1.25*(l6_2)$) -- ($1.25*(l1_2)$);
			\filldraw ($1.25*(l1_2)$) circle[radius=2pt];
			\filldraw ($1.25*(l2_2)$) circle[radius=2pt];
			\draw ($1.25*(l2_2)$) circle[radius=4pt];
			\filldraw ($1.25*(l3_2)$) circle[radius=2pt];
			\filldraw ($1.25*(l4_2)$) circle[radius=2pt];
			\filldraw ($1.25*(l5_2)$) circle[radius=2pt];
			\filldraw ($1.25*(l6_2)$) circle[radius=2pt];
			
			\draw[->] ($1.1*(l2)$) -- ($0.8*(l2)$);
			
			
			\node (l1_3) at ($1.5*(l3)+0.5*(1.7320508075688774,-1.)$) {};
			\node (l2_3) at ($1.5*(l3)+0.5*(1.7320508075688774,1.)$) {};
			\node (l3_3) at ($1.5*(l3)+0.5*(0,2.)$) {};
			\node (l4_3) at ($1.5*(l3)+0.5*(-1.7320508075688774,1.)$){};
			\node (l5_3) at ($1.5*(l3)+0.5*(-1.7320508075688774,-1.)$) {};
			\node (l6_3) at ($1.5*(l3)+0.5*(0,-2.)$) {};
			
			\draw[->,shorten >=3pt, shorten <=3pt] ($1.25*(l3_3)$) -- ($1.25*(l4_3)$);
			\draw[->,shorten >=3pt, shorten <=3pt] ($1.25*(l4_3)$) -- ($1.25*(l5_3)$);
			\draw[->,shorten >=3pt, shorten <=3pt] ($1.25*(l5_3)$) -- ($1.25*(l6_3)$);
			\draw[->,shorten >=3pt, shorten <=3pt] ($1.25*(l3_3)$) -- ($1.25*(l2_3)$);
			\draw[->,shorten >=3pt, shorten <=3pt] ($1.25*(l2_3)$) -- ($1.25*(l1_3)$);
			\draw[->,shorten >=3pt, shorten <=3pt] ($1.25*(l1_3)$) -- ($1.25*(l6_3)$);
			\fill[color=black, opacity=0.2] ($1.25*(l1_3)$) -- ($1.25*(l2_3)$) -- ($1.25*(l3_3)$) -- ($1.25*(l4_3)$) -- ($1.25*(l5_3)$) -- ($1.25*(l6_3)$) -- ($1.25*(l1_3)$);
			\filldraw ($1.25*(l1_3)$) circle[radius=2pt];
			\filldraw ($1.25*(l2_3)$) circle[radius=2pt];
			\filldraw ($1.25*(l3_3)$) circle[radius=2pt];
			\draw ($1.25*(l3_3)$) circle[radius=4pt];
			\filldraw ($1.25*(l4_3)$) circle[radius=2pt];
			\filldraw ($1.25*(l5_3)$) circle[radius=2pt];
			\filldraw ($1.25*(l6_3)$) circle[radius=2pt];
			
			\draw[->] ($1.1*(l3)$) -- ($0.8*(l3)$);
			
			
			\node (l1_4) at ($1.5*(l4)+0.5*(1.7320508075688774,-1.)$) {};
			\node (l2_4) at ($1.5*(l4)+0.5*(1.7320508075688774,1.)$) {};
			\node (l3_4) at ($1.5*(l4)+0.5*(0,2.)$) {};
			\node (l4_4) at ($1.5*(l4)+0.5*(-1.7320508075688774,1.)$){};
			\node (l5_4) at ($1.5*(l4)+0.5*(-1.7320508075688774,-1.)$) {};
			\node (l6_4) at ($1.5*(l4)+0.5*(0,-2.)$) {};

			\draw[->,shorten >=3pt, shorten <=3pt] ($1.25*(l4_4)$) -- ($1.25*(l5_4)$);
			\draw[->,shorten >=3pt, shorten <=3pt] ($1.25*(l5_4)$) -- ($1.25*(l6_4)$);
			\draw[->,shorten >=3pt, shorten <=3pt] ($1.25*(l6_4)$) -- ($1.25*(l1_4)$);
			\draw[->,shorten >=3pt, shorten <=3pt] ($1.25*(l4_4)$) -- ($1.25*(l3_4)$);
			\draw[->,shorten >=3pt, shorten <=3pt] ($1.25*(l3_4)$) -- ($1.25*(l2_4)$);
			\draw[->,shorten >=3pt, shorten <=3pt] ($1.25*(l2_4)$) -- ($1.25*(l1_4)$);
			\fill[color=black, opacity=0.2] ($1.25*(l1_4)$) -- ($1.25*(l2_4)$) -- ($1.25*(l3_4)$) -- ($1.25*(l4_4)$) -- ($1.25*(l5_4)$) -- ($1.25*(l6_4)$) -- ($1.25*(l1_4)$);
			\filldraw ($1.25*(l1_4)$) circle[radius=2pt];
			\filldraw ($1.25*(l2_4)$) circle[radius=2pt];
			\filldraw ($1.25*(l3_4)$) circle[radius=2pt];
			\filldraw ($1.25*(l4_4)$) circle[radius=2pt];
			\draw ($1.25*(l4_4)$) circle[radius=4pt];
			\filldraw ($1.25*(l5_4)$) circle[radius=2pt];
			\filldraw ($1.25*(l6_4)$) circle[radius=2pt];
			
			\draw[->] ($1.1*(l4)$) -- ($0.8*(l4)$);
			
			
			\node (l1_5) at ($1.5*(l5)+0.5*(1.7320508075688774,-1.)$) {};
			\node (l2_5) at ($1.5*(l5)+0.5*(1.7320508075688774,1.)$) {};
			\node (l3_5) at ($1.5*(l5)+0.5*(0,2.)$) {};
			\node (l4_5) at ($1.5*(l5)+0.5*(-1.7320508075688774,1.)$){};
			\node (l5_5) at ($1.5*(l5)+0.5*(-1.7320508075688774,-1.)$) {};
			\node (l6_5) at ($1.5*(l5)+0.5*(0,-2.)$) {};
			
			\draw[->,shorten >=3pt, shorten <=3pt] ($1.25*(l5_5)$) -- ($1.25*(l6_5)$);
			\draw[->,shorten >=3pt, shorten <=3pt] ($1.25*(l6_5)$) -- ($1.25*(l1_5)$);
			\draw[->,shorten >=3pt, shorten <=3pt] ($1.25*(l1_5)$) -- ($1.25*(l2_5)$);
			\draw[->,shorten >=3pt, shorten <=3pt] ($1.25*(l5_5)$) -- ($1.25*(l4_5)$);
			\draw[->,shorten >=3pt, shorten <=3pt] ($1.25*(l4_5)$) -- ($1.25*(l3_5)$);
			\draw[->,shorten >=3pt, shorten <=3pt] ($1.25*(l3_5)$) -- ($1.25*(l2_5)$);
			\fill[color=black, opacity=0.2] ($1.25*(l1_5)$) -- ($1.25*(l2_5)$) -- ($1.25*(l3_5)$) -- ($1.25*(l4_5)$) -- ($1.25*(l5_5)$) -- ($1.25*(l6_5)$) -- ($1.25*(l1_5)$);
			\filldraw ($1.25*(l1_5)$) circle[radius=2pt];
			\filldraw ($1.25*(l2_5)$) circle[radius=2pt];
			\filldraw ($1.25*(l3_5)$) circle[radius=2pt];
			\filldraw ($1.25*(l4_5)$) circle[radius=2pt];
			\filldraw ($1.25*(l5_5)$) circle[radius=2pt];
			\draw ($1.25*(l5_5)$) circle[radius=4pt];
			\filldraw ($1.25*(l6_5)$) circle[radius=2pt];
			
			\draw[->] ($1.1*(l5)$) -- ($0.8*(l5)$);
			
			
			\node (l1_6) at ($1.5*(l6)+0.5*(1.7320508075688774,-1.)$) {};
			\node (l2_6) at ($1.5*(l6)+0.5*(1.7320508075688774,1.)$) {};
			\node (l3_6) at ($1.5*(l6)+0.5*(0,2.)$) {};
			\node (l4_6) at ($1.5*(l6)+0.5*(-1.7320508075688774,1.)$){};
			\node (l5_6) at ($1.5*(l6)+0.5*(-1.7320508075688774,-1.)$) {};
			\node (l6_6) at ($1.5*(l6)+0.5*(0,-2.)$) {};
			
			\draw[->,shorten >=3pt, shorten <=3pt] ($1.25*(l6_6)$) -- ($1.25*(l1_6)$);
			\draw[->,shorten >=3pt, shorten <=3pt] ($1.25*(l1_6)$) -- ($1.25*(l2_6)$);
			\draw[->,shorten >=3pt, shorten <=3pt] ($1.25*(l2_6)$) -- ($1.25*(l3_6)$);
			\draw[->,shorten >=3pt, shorten <=3pt] ($1.25*(l6_6)$) -- ($1.25*(l5_6)$);
			\draw[->,shorten >=3pt, shorten <=3pt] ($1.25*(l5_6)$) -- ($1.25*(l4_6)$);
			\draw[->,shorten >=3pt, shorten <=3pt] ($1.25*(l4_6)$) -- ($1.25*(l3_6)$);
			\fill[color=black, opacity=0.2] ($1.25*(l1_6)$) -- ($1.25*(l2_6)$) -- ($1.25*(l3_6)$) -- ($1.25*(l4_6)$) -- ($1.25*(l5_6)$) -- ($1.25*(l6_6)$) -- ($1.25*(l1_6)$);
			\filldraw ($1.25*(l1_6)$) circle[radius=2pt];
			\filldraw ($1.25*(l2_6)$) circle[radius=2pt];
			\filldraw ($1.25*(l3_6)$) circle[radius=2pt];
			\filldraw ($1.25*(l4_6)$) circle[radius=2pt];
			\filldraw ($1.25*(l5_6)$) circle[radius=2pt];
			\filldraw ($1.25*(l6_6)$) circle[radius=2pt];
			\draw ($1.25*(l6_6)$) circle[radius=4pt];
			
			\node at ($1.25*1.5*(l1)$) {\tiny{$(\Zero,T_1)$}};
			
			\draw[->] ($1.1*(l6)$) -- ($0.8*(l6)$);

			\node at (0,-5.5) {\small$\Sc(\Ac)$};
			
		\end{tikzpicture}
		\caption{The dual oriented matroid complex $\Lc^\dual(\Ac)$ respectively zonotope and the Salvetti complex $\Sc(\Ac)$ of the arrangement $\Ac = \{\{x=0\},\{y=0\},\{x=y\}\}$.}
		\label{fig:ExSalA2}
	\end{figure}
	The following example illustrates the geometric picture behind the construction of $\Sc$.
	\begin{example}
		\label{ex:SalA2}
		Let $\Ac$ be the arrangement in $\RR^2$ with defining polynomial $Q(\Ac) = xy(x-y)$.
		Consider Figure \ref{fig:ExSalA2}.
		Then the Salvetti complex $\Sc(\Ac)$ has
		\begin{itemize}
			\item 
			one vertex for each of the $6$ chambers or topes of $\Ac$,
			
			\item 
			two edges for each pair of adjacent chambers separated by one line, drawn by two arrows pointing in opposite directions,
			
			\item 
			and one $2$-cell for each chamber (a hexagon, corresponding to the dual covector complex or zonotope of $\Ac$), 
			whose boundary is identified with the edges corresponding to
			the appropriate arrows pointing away from the associated base chamber.
		\end{itemize}
	\end{example}
	
	We list the following further example of the Salvetti complex which will be important for our later considerations.
	\begin{example}
		\label{ex:SalBooleanParam}
		Let $\Bc_n := \{ \{x_1=0\}, \{x_2=0\}, \ldots, \{x_n=0\}\}$ be the Boolean arrangement consisting of the coordinate hyperplanes
		in $\RR^n$. The complement $\Xf(\Bc_n) = (\CC^\times)^n$ is the complex $n$-torus which deformation retracts to $(S^1)^n$
		where $S^1  = \{z \in \CC \mid |z|=1\}$ is the unit circle in $\CC^\times = \CC\setminus\{0\}$.
		The poset $\Lc^\dual = \{0,+,-\}^n$ is isomorphic to the face lattice of the $n$-cube $[0,1]^n$ by identifying $\sigma \in \Lc^\dual$
		with the face $\{x \in [0,1]^n \mid x_i=0$ if $\sigma_i=-$, $x_i=1$ if $\sigma_i=+\} =:\psi(\sigma) \subseteq [0,1]^n$.
		
		By Lemma \ref{lem:SalPrincipialIdealCovectors}, the Salvetti complex $\Sc(\Bc_n)$ is obtained by gluing together $n$-cubes.
		We can describe this explicitly geometrically as follows.
		We have the following concrete parametrization of the cells in $\Sc(\Bc_n)$,
		yielding the cell-decomposition of $(S^1)^n \hookrightarrow (\CC^\times)^n$, by
		$\varphi: \Sc(\Bc_n) \to 2^{(S^1)^n}, (\sigma,T) \mapsto 
		\{(\exp(T_1(t_1-1)\pi i),\ldots,\exp(T_n(t_n-1)\pi i)) \mid (t_1,\ldots,t_n) \in \psi(\sigma)\}$
		and thus an identification $|\Sc(\Bc_n)| = (S^1)^n \subseteq \CC^n$.
		This amounts to a concrete description of the inclusion of $|\Sc(\Bc_n)|$ 
		into $\Xf(\Bc_n)$ which we also denote by $\varphi:|\Sc(\Bc_n)| = (S^1)^n \hookrightarrow (\CC^\times) = \Xf(\Bc_n)$
		for later reference.
	\end{example}
	
	
	\subsection{Discrete Morse theory}
	\label{ssec:DiscMorseTheory}
	
	We review some basic concepts from Forman's discrete Morse theory \cite{Forman1998_DiscrMorse} from the viewpoint
	of acyclic matchings on posets respectively regular cell complexes adopted by Chari \cite{Chari2000_DiscreteMorseCombDecomp}. 
	For a comprehensive exposition of the theory we refer to 
	Kozlov's book \cite{Kozlov2008_CombAlgTop}.
	
	\begin{definition}
		Let $P = (P,\leq)$ be a poset.
		We associate a directed graph $G(\Sigma) = (\Vc,\Ec)$ with vertex set $\Vc = P$ the ground set
		of $P$ and directed edges $\Ec := \{ (a,b) \mid a \lessdot b\}$ its cover relations,
		i.e. the Hasse diagram of the poset.
		
		A subset $\mtc \subseteq \Ec$ is called a \emph{matching} on $P$ if every element of $P$ appears
		in at most one $(a,b) \in \mtc$.
		
		Let $G(P,\mtc) = (\Vc,\Ec')$ be the new directed graph with $\Ec' := \Ec\setminus\mtc \cup \{(b,a) \mid (a,b) \in \mtc\}$, 
		i.e.\ the directed graph constructed from $G(P)$ by inverting all edges in $\mtc$.
		A matching $\mtc$ on $\Sigma$ is called \emph{acyclic} if $G(P,\mtc)$ has no directed cycles.
		
		For an acyclic matching $\mtc$ on $P$ its \emph{critical elements} $C(\mtc) \subseteq P$ are defined
		as
		\[
		C(\mtc) := \{ a \in P \mid a \notin e\text{ for all }e \in \mtc\}.
		\]
		
		In the same way we define acyclic matchings for a regular cell complex $\Sigma$  as matchings on $\Fc(\Sigma)$.
	\end{definition}
	
	We record the following special case of the main theorem of discrete Morse theory,
	cf.\ \cite{Forman1998_DiscrMorse}, \cite{Chari2000_DiscreteMorseCombDecomp}, \cite[Ch.~11]{Kozlov2008_CombAlgTop}.
	
	\begin{theorem}
		\label{thm:MainThmDMTRegularSubcomplex}
		Let $\Sigma$ be a regular cell complex and let $\Gamma \subseteq \Sigma$ be a subcomplex.
		Assume that $\mtc$ is an acyclic matching on $\Fc(\Sigma)$ with $C(\mtc) = \Gamma$.
		Then $|\Gamma|$ is a strong deformation retract of $|\Sigma|$.
		In particular, the inclusion $|\Gamma| \hookrightarrow |\Sigma|$ is a homotopy equivalence.
	\end{theorem}
	
	We conclude by recalling the following useful tool to construct acyclic matchings on posets, and hence on cell complexes.
	\begin{theorem}[Patchwork theorem {\cite[Thm.~11.10]{Kozlov2008_CombAlgTop}}]
		\label{thm:Patchwork}
		Let $f:P \to Q$ be a poset map. Assume that for all $q \in Q$ we have an acyclic matching $\mtc(q)$
		on $f^{-1}(q)$. Then $\mtc = \bigcup\limits_{q \in Q}\mtc(q)$ is an acyclic matching on $P$.
	\end{theorem}
	
	
	\section{The tope-rank subdivision}
	\label{sec:TopeRankSubdiv}
	
	Throughout this section assume that $\OM = (E,\Lc)$ is a simple oriented matroid with $|E|=n$ of rank $d\geq 2$.
	
	We define a certain subdivision of the covector complex of an oriented matroid with respect to a base tope.
	This also results in a corresponding subdivision of the Salvetti complex.
	
	The combinatorial underpinnings are as follows.
	
	\begin{definition}
		\label{def:TopeRankSubdivCovectors}
		Let $\sigma \in \Lc^\dual$ be a cell in the dual covector complex and $B \in \Tc$ a tope.
		Recall that we set $\Tc(\sigma) := \Tc \cap \Lc^\dual_{\leq \sigma}$ 
		which can be identified with $\vts(\sigma)$ by Proposition \ref{prop:OMCpxIntersectionProperty} and define
		\begin{itemize}
			\item 
			$\sigma^B_k := \{T \in \Tc(\sigma) \mid \rk_B(T) = k\}$,
			
			\item 
			$\sigma^B_{[k,k+1]} := \sigma^B_k \cup \sigma^B_{k+1}$,
			
			\item
			define the \emph{($B$-)rank subdivision} of $\sigma$ as:
			\begin{align*}
				\rk_B\sd(\sigma) := &\{ \sigma^B_k \mid k \in \rk_B(\Tc(\sigma))\} \\
				&\cup \{\sigma^B_{[k,k+1]} \mid  k \in \rk_B(\Tc(\sigma))\setminus \{\rk_B(\sigma\circ(-B))\}\}.
			\end{align*}
		\end{itemize}
		Then the \emph{($B$-)rank subdivision} of $\Lc^\dual$ is the poset defined by:
		\[
		\rk_B\sd \Lc^\dual := \bigcup\limits_{\sigma \in \Lc^\dual} \rk_B\sd(\sigma) \subseteq 2^\Tc
		\] with partial order by inclusion.
		
		Thanks to Proposition \ref{prop:OMCpxIntersectionProperty}, we have a poset map 
		$p:2^\Tc \to \Lc^\dual, \afr \mapsto \min\{\sigma \in \Lc^\dual \mid \afr \subseteq \Tc(\sigma)\}$.
		Thus, by restriction of $p$ to $\rk_B\sd\Lc^\dual$ we obtain a poset map of the subdivision to the original complex:
		\[
		p_B := p|_{\rk_B\sd\Lc^\dual}: \rk_B\sd\Lc^\dual \to \Lc^\dual, \Tc \supseteq \afr \mapsto \min\{\sigma \in \Lc^\dual \mid \afr \subseteq \Tc(\sigma)\}.
		\]
	\end{definition}

	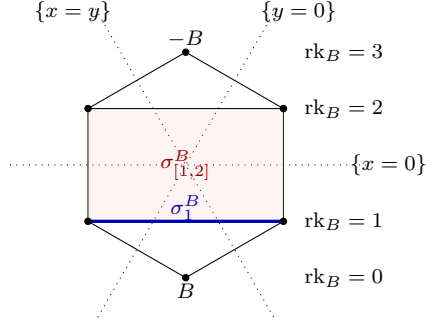
\begin{figure}
		\begin{tikzpicture}[scale=0.6]
			\draw[dotted] (4.,0.) -- (-4.,0.);  
			\node at (4.5,0.) {\tiny $\{x=0\}$}; 		
			\draw[dotted] (2.,3.4641016151377548) -- (-2.,-3.4641016151377548);  
			\node at (2.5,3.4) {\tiny $\{y=0\}$}; 		
			\draw[dotted] (-2.,3.4641016151377548) -- (2.,-3.4641016151377548);  
			\node at (-2.5,3.4) {\tiny $\{x=y\}$}; 		
			
			\node (l1) at (0.0,-2.) {};
			\node (l2) at (1.7320508075688774,-1.) {};
			\node (l3) at (1.7320508075688774,1.) {};
			\node (l4) at (0.0,2.) {};
			\node (l5) at (-1.7320508075688774,1.) {};
			\node (l6) at (-1.7320508075688774,-1.) {};
			\node (zero) at (0.0,0.) {};
			
			\node at ($(l1)+(0,-0.8)$) {\tiny$B$};
			\node at ($(l4)+(0,0.8)$) {\tiny$-B$};

			\draw ($1.25*(l1)$) -- ($1.25*(l2)$) -- ($1.25*(l3)$) -- ($1.25*(l4)$) -- ($1.25*(l5)$) -- ($1.25*(l6)$) -- ($1.25*(l1)$);

			\fill[color=red!20!white, opacity=0.2] ($1.25*(l2)$) -- ($1.25*(l3)$) -- ($1.25*(l5)$) -- ($1.25*(l6)$);
			
			\draw[line width=1.25pt, color=black!30!blue] ($1.25*(l2)$) -- ($1.25*(l6)$);
			\draw ($1.25*(l3)$) -- ($1.25*(l5)$);
			
			\node[color=black!40!red] at ($0.25*1.25*(l2) + 0.25*1.25*(l3) + 0.25*1.25*(l5) + 0.25*1.25*(l6)$) {\tiny$\sigma^B_{[1,2]}$};
			\node[color=black!30!blue] at ($0.5*1.25*(l2) + 0.5*1.25*(l6) +(0,0.3)$) {\tiny$\sigma^B_{1}$};
			
			
			
			\filldraw ($1.25*(l1)$) circle[radius=2pt];
			\filldraw ($1.25*(l2)$) circle[radius=2pt];
			\filldraw ($1.25*(l3)$) circle[radius=2pt];
			\filldraw ($1.25*(l4)$) circle[radius=2pt];
			\filldraw ($1.25*(l5)$) circle[radius=2pt];
			\filldraw ($1.25*(l6)$) circle[radius=2pt];
			
			
			\node at (3.5,-2.5) {\tiny$\rk_B = 0$};
			\node at (3.5,-1.25) {\tiny$\rk_B = 1$};
			\node at (3.5,1.25) {\tiny$\rk_B = 2$};
			\node at (3.5,2.5) {\tiny$\rk_B = 3$};
			
		\end{tikzpicture}
		\caption{The $B$-rank subdivision $\rk_B\sd\Lc^\dual(\Ac)$ for the arrangement $\Ac$ defined by $Q(\Ac) = xy(x-y)$.}
		\label{fig:TopeRkSubdivCovec}
	\end{figure}
	
	From the preceding definition it is not clear that our new poset $\rk_B\sd\Lc^\dual$ is the face poset of a regular cell complex,
	i.e.\ it is a CW-poset.
	This is the following first main theorem of this section.
	\begin{theorem}
		\label{thm:rkBsdLcCWposet}
		The tope rank subdivision poset $\rk_B\sd\Lc^\dual$ 
		is the face poset of a regular $PL$-cell decomposition of the $d$-ball. 
	\end{theorem}
	
	\begin{proof}
		
		The proof is in three steps.
		First, for $0\leq k \leq n-1$ we construct a new complex $\Sigma_{[k,k+1]}$ with respect to $B \in \Tc$ from a certain subcomplex of $\Lc\setminus \{\Zero\}$ which
		is the face poset of a regular cell decomposition of the PL $(d-1)$-sphere (Step 1).
		Then, we show that $\Fc(\Sigma_{[k,k+1]}^\dual)$ is isomorphic to $\rk_B\sd\Lc^\dual_{<x}$ for a maximal element $x = \Zero^B_{[k,k+1]} \in \rk_B\sd\Lc^\dual$
		which is also a PL $(d-1)$-sphere thanks to Lemma \ref{lem:PL sphereDual}(ii) (Step 2).
		Finally, the successive application of Theorem \ref{thm:PLBallsAndSpheres} concludes our proof (Step 3).
		
		Let $B \in \Tc$ be a base tope and $k \in \{0,\ldots, n-1\}$.
		
		\textbf{Step 1.} 
		By Theorem \ref{thm:OMCovectorShellable} and Proposition \ref{prop:ShellingSphereSubsequenceBall}, 
		the subcomplex $\Lc(\Tc_{B,\leq k})\setminus \{\Zero\}$ is a shellable ball
		and in particular a PL $(d-1)$-ball, where $\Tc_{B,\leq k}$ denotes the set of topes with $\rk_B \leq k$. 
		Let $\Gamma$ be its boundary subcomplex which is a PL $(d-2)$-sphere (cf.\ the left hand side of Figure \ref{fig:A_sdCpx1}).
		Note that if $\Tc_{B,k}$ is the set of topes of rank $k$ in $\Tc_B$ then for the cone over $\Gamma$ we have
		\[
		C\Gamma = \bigcup\limits_{R \in \Tc_{B,k}} C(\Lc(R)\cap \Gamma),
		\]
		where each smaller cone $C(\Lc(R)\cap \Gamma)$, constructed with the same cone point $v$,
		is a PL $(d-1)$-ball by Proposition \ref{prop:ShellingTopeSubcpx}, Proposition \ref{prop:ShellingSphereSubsequenceBall}. 
		Note that two such distinct balls either have empty intersection or intersect only along their boundaries.
		If $k= 0$, set $\Sigma_k := \Lc(B)$ and for $k\geq 1$, let $\Sigma_k$
		be the regular cell complex with one maximal cell for each $C(\Lc(R)\cap \Gamma)$ ($R \in \Tc_{B,k}$)
		union with their boundary subcomplexes. 
		Then $C\Gamma \triangleleft \Sigma_k$ is a subdivision and thus $\Sigma_k$
		is a PL $(d-1)$-ball with boundary $\Gamma$ by Theorem \ref{thm:PLBallsAndSpheres}(iii).
		
		Similarly, we have
		\[
		C\Gamma = \bigcup\limits_{R \in \Tc_{B,k+1}} C(\Lc(R)\cap \Gamma),
		\]
		and for $k<n-1$ let $\Sigma_{k+1}$ be this partial cone constructed in the same way from the set of topes or rank $k+1$ 
		with one maximal cell for each $C(\Lc(R)\cap \Gamma)$ ($R \in \Tc_{B,k+1}$)
		union with their boundary subcomplexes; for $k=n-1$ let $\Sigma_{k+1} := \Lc(-B)\setminus \{\Zero\}$.
		Then, $\Sigma_{k+1}$ is a PL $(d-1)$-ball with boundary $\Gamma$ again thanks to Theorem \ref{thm:PLBallsAndSpheres}(iii). 
		Let $\Sigma_{[k,k+1]} := \Sigma_{k+1} \cup \Sigma_k$ (cf.\ the right hand side of Figure \ref{fig:A_sdCpx1}). 
		By construction, we have $\Sigma_{k+1} \cap \Sigma_k = \Gamma$, 
		thus $\Sigma_{[k,k+1]}$ is a regular cell complex and a PL $(d-1)$-sphere by Theorem \ref{thm:PLBallsAndSpheres}(ii).

		\textbf{Step 2.}
		Since $\Gamma$ is a closed subcomplex of $\Lc\setminus \{\Zero\}$, it inherits the intersection property from $\Lc\setminus \{\Zero\}$ (see Proposition \ref{prop:OMCpxIntersectionProperty}).
		Consequently, the complex $\Sigma_{[k,k+1]}$ also has the intersection property and so does $\Sigma_{[k,k+1]}^\dual$ 
		which is a PL $(d-1)$-sphere by Lemma \ref{lem:PL sphereDual}(ii).
		Moreover, by our construction, the vertices of $\Sigma_{[k,k+1]}^\dual$ are in one to one correspondence with $\Tc_{B,[k,k+1]}$ the set of topes
		of rank $k$ and $k+1$ in $\Tc_B$ which are exactly the minimal elements contained in $\rk_B\sd\Lc^\dual_{< x}$ where $x=\Zero^B_{[k,k+1]}$.
		By Lemma \ref{lem:IntersectionProperty} we can identify $\Fc(\Sigma_{[k,k+1]}^\dual)$ with the set $\{\vts(\sigma) \mid \sigma \in \Sigma_{[k,k+1]}^\dual\}$
		and under this identification we have $\{\vts(\sigma) \mid \sigma \in \Sigma_{[k,k+1]}^\dual\} = \rk_B\sd\Lc^\dual_{< x}$ as sets of subsets of $\Tc$. 
		Thus, $\Fc(\Sigma_{[k,k+1]}^\dual) \isom \rk_B\sd\Lc^\dual_{< x}$.
		
		\textbf{Step 3.}
		Since $\rk_B\sd\Lc^\dual_{\leq \Zero^B_{[k-1,k]}} \cap \rk_B\sd\Lc^\dual_{\leq \Zero^B_{[k,k+1]}} = \rk_B\sd\Lc^\dual_{\leq \Zero^B_{k}}$
		are PL $(d-1)$-balls by Lemma \ref{lem:PL sphereDual}(i) 
		which are contained in the boundary of $\rk_B\sd\Lc^\dual_{\leq \Zero^B_{[k-1,k]}}$ respectively $\rk_B\sd\Lc^\dual_{\leq \Zero^B_{[k,k+1]}}$,
		inductively, by Theorem \ref{thm:PLBallsAndSpheres}(i) we can finally conclude that
		$\rk_B\sd\Lc^\dual$ is the face poset of a regular PL cell decomposition of the $d$-ball.
	\end{proof}
	
	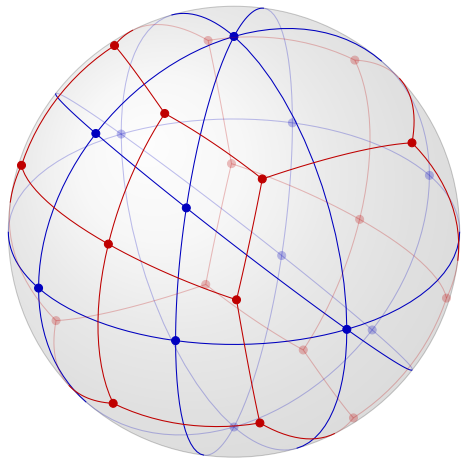
\begin{figure}	
		
		\tdplotsetmaincoords{60}{60}
		\begin{tikzpicture}[scale=3,tdplot_main_coords]
				
				%
				\draw[ball color = gray!40, opacity = 0.2] (0,0,0) circle (1cm);
				
				\def\OMcolor{blue!75!black}
				{\color{\OMcolor}
						\foreach \i/\vx/\vy/\vz in {
								1/0./0./1., 
								2/0./1./0., 
								3/1./0./0., 
								4/-0.70710678118654746/0.70710678118654746/0., 
								5/0./-0.70710678118654746/0.70710678118654746, 
								6/0.57735026918962584/-0.57735026918962584/0.57735026918962584, 
								7/0./0./-1., 
								8/0./-1./0., 
								9/-1./0./0., 
								10/0.70710678118654746/-0.70710678118654746/0., 
								11/0./0.70710678118654746/-0.70710678118654746, 
								12/-0.57735026918962584/0.57735026918962584/-0.57735026918962584	
							}{
								\node (\i) at (\vx,\vy,\vz) {};
								\POnSfb{\vx}{\vy}{\vz}{}
							}
						%
						\foreach \ax/\ay/\az/\bx/\by/\bz in {
								0./0./1./0./1./0., 
								0./0./1./1./0./0., 
								0./0./1./-0.70710678118654746/0.70710678118654746/0., 
								0./0./1./0./-0.70710678118654746/0.70710678118654746, 
								0./0./1./0.57735026918962584/-0.57735026918962584/0.57735026918962584, 
								0./0./1./-1./0./0., 
								0./1./0./1./0./0., 
								0./1./0./-0.70710678118654746/0.70710678118654746/0., 
								0./1./0./0./0.70710678118654746/-0.70710678118654746, 
								1./0./0./0.57735026918962584/-0.57735026918962584/0.57735026918962584, 
								1./0./0./0./0./-1., 
								1./0./0./0.70710678118654746/-0.70710678118654746/0., 
								1./0./0./0./0.70710678118654746/-0.70710678118654746, 
								-0.70710678118654746/0.70710678118654746/0./-1./0./0., 
								-0.70710678118654746/0.70710678118654746/0./-0.57735026918962584/0.57735026918962584/-0.57735026918962584, 
								0./-0.70710678118654746/0.70710678118654746/0.57735026918962584/-0.57735026918962584/0.57735026918962584, 
								0./-0.70710678118654746/0.70710678118654746/0./-1./0., 
								0./-0.70710678118654746/0.70710678118654746/-1./0./0., 
								0.57735026918962584/-0.57735026918962584/0.57735026918962584/0.70710678118654746/-0.70710678118654746/0., 
								0./0./-1./0./-1./0., 
								0./0./-1./-1./0./0., 
								0./0./-1./0.70710678118654746/-0.70710678118654746/0., 
								0./0./-1./0./0.70710678118654746/-0.70710678118654746, 
								0./0./-1./-0.57735026918962584/0.57735026918962584/-0.57735026918962584, 
								0./-1./0./-1./0./0., 
								0./-1./0./0.70710678118654746/-0.70710678118654746/0., 
								-1./0./0./-0.57735026918962584/0.57735026918962584/-0.57735026918962584, 
								0./0.70710678118654746/-0.70710678118654746/-0.57735026918962584/0.57735026918962584/-0.57735026918962584 
							}{
								\GCArcABfb{\ax}{\ay}{\az}{\bx}{\by}{\bz}{color=\OMcolor}
							}
					}
				
				
				\def\dualOMcolor{red!75!black}
				{\color{\dualOMcolor}
						\foreach \i/\vx/\vy/\vz in {
								1/-0.47596314947796792/-0.33655677059077743/0.8125199200687454, 
								2/0.47596314947796792/0.33655677059077743/-0.8125199200687454, 
								3/0.21513724867401407/-0.47862549063280968/0.85125413593678212, 
								4/-0.21513724867401407/0.47862549063280968/-0.85125413593678212, 
								5/-0.8125199200687454/0.33655677059077743/0.47596314947796792, 
								6/0.8125199200687454/-0.33655677059077743/-0.47596314947796792, 
								7/-0.47596314947796792/-0.8125199200687454/0.33655677059077743, 
								8/0.47596314947796792/0.8125199200687454/-0.33655677059077743, 
								9/0.68455031944529765/-0.25056280708573159/0.68455031944529765, 
								10/-0.68455031944529765/0.25056280708573159/-0.68455031944529765, 
								11/0.36700100107065714/-0.85476343535873778/0.36700100107065714, 
								12/-0.36700100107065714/0.85476343535873778/-0.36700100107065714, 
								13/-0.85125413593678212/0.47862549063280968/-0.21513724867401407, 
								14/0.85125413593678212/-0.47862549063280968/0.21513724867401407, 
								15/-0.33655677059077743/0.8125199200687454/0.47596314947796792, 
								16/0.33655677059077743/-0.8125199200687454/-0.47596314947796792, 
								17/-0.57735026918962573/-0.57735026918962573/-0.57735026918962573, 
								18/0.57735026918962573/0.57735026918962573/0.57735026918962573	
							}{
								\node (\i) at (\vx,\vy,\vz) {};
								\POnSfb{\vx}{\vy}{\vz}{}
							}
						%
						\foreach \ax/\ay/\az/\bx/\by/\bz in {
								-0.47596314947796792/-0.33655677059077743/0.8125199200687454/0.21513724867401407/-0.47862549063280968/0.85125413593678212, 
								-0.47596314947796792/-0.33655677059077743/0.8125199200687454/-0.8125199200687454/0.33655677059077743/0.47596314947796792, 
								-0.47596314947796792/-0.33655677059077743/0.8125199200687454/-0.47596314947796792/-0.8125199200687454/0.33655677059077743, 
								0.47596314947796792/0.33655677059077743/-0.8125199200687454/-0.21513724867401407/0.47862549063280968/-0.85125413593678212, 
								0.47596314947796792/0.33655677059077743/-0.8125199200687454/0.8125199200687454/-0.33655677059077743/-0.47596314947796792, 
								0.47596314947796792/0.33655677059077743/-0.8125199200687454/0.47596314947796792/0.8125199200687454/-0.33655677059077743, 
								0.21513724867401407/-0.47862549063280968/0.85125413593678212/0.68455031944529765/-0.25056280708573159/0.68455031944529765, 
								0.21513724867401407/-0.47862549063280968/0.85125413593678212/0.36700100107065714/-0.85476343535873778/0.36700100107065714, 
								-0.21513724867401407/0.47862549063280968/-0.85125413593678212/-0.68455031944529765/0.25056280708573159/-0.68455031944529765, 
								-0.21513724867401407/0.47862549063280968/-0.85125413593678212/-0.36700100107065714/0.85476343535873778/-0.36700100107065714, 
								-0.8125199200687454/0.33655677059077743/0.47596314947796792/-0.85125413593678212/0.47862549063280968/-0.21513724867401407, 
								-0.8125199200687454/0.33655677059077743/0.47596314947796792/-0.33655677059077743/0.8125199200687454/0.47596314947796792, 
								0.8125199200687454/-0.33655677059077743/-0.47596314947796792/0.85125413593678212/-0.47862549063280968/0.21513724867401407, 
								0.8125199200687454/-0.33655677059077743/-0.47596314947796792/0.33655677059077743/-0.8125199200687454/-0.47596314947796792, 
								-0.47596314947796792/-0.8125199200687454/0.33655677059077743/0.36700100107065714/-0.85476343535873778/0.36700100107065714, 
								-0.47596314947796792/-0.8125199200687454/0.33655677059077743/-0.57735026918962573/-0.57735026918962573/-0.57735026918962573, 
								0.47596314947796792/0.8125199200687454/-0.33655677059077743/-0.36700100107065714/0.85476343535873778/-0.36700100107065714, 
								0.47596314947796792/0.8125199200687454/-0.33655677059077743/0.57735026918962573/0.57735026918962573/0.57735026918962573, 
								0.68455031944529765/-0.25056280708573159/0.68455031944529765/0.85125413593678212/-0.47862549063280968/0.21513724867401407, 
								0.68455031944529765/-0.25056280708573159/0.68455031944529765/0.57735026918962573/0.57735026918962573/0.57735026918962573, 
								-0.68455031944529765/0.25056280708573159/-0.68455031944529765/-0.85125413593678212/0.47862549063280968/-0.21513724867401407, 
								-0.68455031944529765/0.25056280708573159/-0.68455031944529765/-0.57735026918962573/-0.57735026918962573/-0.57735026918962573, 
								0.36700100107065714/-0.85476343535873778/0.36700100107065714/0.85125413593678212/-0.47862549063280968/0.21513724867401407, 
								0.36700100107065714/-0.85476343535873778/0.36700100107065714/0.33655677059077743/-0.8125199200687454/-0.47596314947796792, 
								-0.36700100107065714/0.85476343535873778/-0.36700100107065714/-0.85125413593678212/0.47862549063280968/-0.21513724867401407, 
								-0.36700100107065714/0.85476343535873778/-0.36700100107065714/-0.33655677059077743/0.8125199200687454/0.47596314947796792, 
								-0.33655677059077743/0.8125199200687454/0.47596314947796792/0.57735026918962573/0.57735026918962573/0.57735026918962573, 
								0.33655677059077743/-0.8125199200687454/-0.47596314947796792/-0.57735026918962573/-0.57735026918962573/-0.57735026918962573 
							}{
								\GCArcABfb{\ax}{\ay}{\az}{\bx}{\by}{\bz}{}
							}
					}	
			\end{tikzpicture}
		
		\caption{The oriented matroid complex $\Lc \setminus \{\Zero\}$ (blue) and its dual $\Lc^\dual$ (red) of the arrangement $\Ac$
			with defining polynomial $xyz(x+y)(y+z)$.}
		\label{fig:A_OMcpxs}
	\end{figure}
	
	\begin{figure}		
		\tdplotsetmaincoords{60}{60}
		\begin{tikzpicture}[scale=2.75,tdplot_main_coords]
				
				
				\draw[ball color = gray!40, opacity = 0.25] (0,0,0) circle (1cm);
				
				\def\OMcolor{blue!75!black}
				{\color{\OMcolor}
						\foreach \i/\vx/\vy/\vz in {
								3/1./0./0., 
								6/0.57735026918962584/-0.57735026918962584/0.57735026918962584, 
								7/0./0./-1., 
								8/0./-1./0., 
								10/0.70710678118654746/-0.70710678118654746/0., 
								11/0./0.70710678118654746/-0.70710678118654746 
							}{
								\node (\i) at (\vx,\vy,\vz) {};
								\POnSfb{\vx}{\vy}{\vz}{color=\OMcolor}
							}
						
						\POnSfb{1}{0}{0}{}
						%
						\foreach \ax/\ay/\az/\bx/\by/\bz in {
								0./0./1./0./1./0., 
								0./0./1./1./0./0., 
								0./0./1./-0.70710678118654746/0.70710678118654746/0., 
								0./0./1./0./-0.70710678118654746/0.70710678118654746, 
								0./0./1./0.57735026918962584/-0.57735026918962584/0.57735026918962584, 
								0./0./1./-1./0./0., 
								0./1./0./1./0./0., 
								0./1./0./-0.70710678118654746/0.70710678118654746/0., 
								0./1./0./0./0.70710678118654746/-0.70710678118654746, 
								1./0./0./0.57735026918962584/-0.57735026918962584/0.57735026918962584, 
								1./0./0./0./0./-1., 
								1./0./0./0.70710678118654746/-0.70710678118654746/0., 
								1./0./0./0./0.70710678118654746/-0.70710678118654746, 
								-0.70710678118654746/0.70710678118654746/0./-1./0./0., 
								-0.70710678118654746/0.70710678118654746/0./-0.57735026918962584/0.57735026918962584/-0.57735026918962584, 
								0./-0.70710678118654746/0.70710678118654746/0.57735026918962584/-0.57735026918962584/0.57735026918962584, 
								0./-0.70710678118654746/0.70710678118654746/0./-1./0., 
								0./-0.70710678118654746/0.70710678118654746/-1./0./0., 
								0.57735026918962584/-0.57735026918962584/0.57735026918962584/0.70710678118654746/-0.70710678118654746/0., 
								0./0./-1./0./-1./0., 
								0./0./-1./-1./0./0., 
								0./0./-1./0.70710678118654746/-0.70710678118654746/0., 
								0./0./-1./0./0.70710678118654746/-0.70710678118654746, 
								0./0./-1./-0.57735026918962584/0.57735026918962584/-0.57735026918962584, 
								0./-1./0./-1./0./0., 
								0./-1./0./0.70710678118654746/-0.70710678118654746/0., 
								-1./0./0./-0.57735026918962584/0.57735026918962584/-0.57735026918962584, 
								0./0.70710678118654746/-0.70710678118654746/-0.57735026918962584/0.57735026918962584/-0.57735026918962584 
							}{
								\GCArcABfb{\ax}{\ay}{\az}{\bx}{\by}{\bz}{color=\OMcolor}
							}

						\foreach \ax/\ay/\az/\bx/\by/\bz in {
								1./0./0./0.57735026918962584/-0.57735026918962584/0.57735026918962584, 
								1./0./0./0./0.70710678118654746/-0.70710678118654746, 
								0.57735026918962584/-0.57735026918962584/0.57735026918962584/0.70710678118654746/-0.70710678118654746/0., 
								0./0./-1./0./-1./0., 
								0./0./-1./0./0.70710678118654746/-0.70710678118654746, 
								0./-1./0./0.70710678118654746/-0.70710678118654746/0. 
							}{
								\GCArcABfb{\ax}{\ay}{\az}{\bx}{\by}{\bz}{line width = 0.5mm, color=\OMcolor}
							}
					}
				
				
				\def\dualOMcolor{red!75!black}
				{\color{\dualOMcolor}
						\foreach \i/\vx/\vy/\vz in {
								1/-0.47596314947796792/-0.33655677059077743/0.8125199200687454, 
								2/0.47596314947796792/0.33655677059077743/-0.8125199200687454, 
								3/0.21513724867401407/-0.47862549063280968/0.85125413593678212, 
								4/-0.21513724867401407/0.47862549063280968/-0.85125413593678212, 
								5/-0.8125199200687454/0.33655677059077743/0.47596314947796792, 
								7/-0.47596314947796792/-0.8125199200687454/0.33655677059077743, 
								8/0.47596314947796792/0.8125199200687454/-0.33655677059077743, 
								9/0.68455031944529765/-0.25056280708573159/0.68455031944529765, 
								10/-0.68455031944529765/0.25056280708573159/-0.68455031944529765, 
								11/0.36700100107065714/-0.85476343535873778/0.36700100107065714, 
								12/-0.36700100107065714/0.85476343535873778/-0.36700100107065714, 
								13/-0.85125413593678212/0.47862549063280968/-0.21513724867401407, 
								14/0.85125413593678212/-0.47862549063280968/0.21513724867401407, 
								15/-0.33655677059077743/0.8125199200687454/0.47596314947796792, 
								16/0.33655677059077743/-0.8125199200687454/-0.47596314947796792, 
								17/-0.57735026918962573/-0.57735026918962573/-0.57735026918962573, 
								18/0.57735026918962573/0.57735026918962573/0.57735026918962573	
							}{
								\node (\i) at (\vx,\vy,\vz) {};
								\POnSfb{\vx}{\vy}{\vz}{}
							}

						\foreach \i/\vx/\vy/\vz in {
								2/0.47596314947796792/0.33655677059077743/-0.8125199200687454, 
								4/-0.21513724867401407/0.47862549063280968/-0.85125413593678212, 
								8/0.47596314947796792/0.8125199200687454/-0.33655677059077743, 
								9/0.68455031944529765/-0.25056280708573159/0.68455031944529765, 
								11/0.36700100107065714/-0.85476343535873778/0.36700100107065714, 
								14/0.85125413593678212/-0.47862549063280968/0.21513724867401407, 
								16/0.33655677059077743/-0.8125199200687454/-0.47596314947796792, 
								17/-0.57735026918962573/-0.57735026918962573/-0.57735026918962573 
							}{
								\node (\i) at (\vx,\vy,\vz) {};
								\draw (\i) circle[radius=1pt];
								\POnSfb{\vx}{\vy}{\vz}{}
							}
						%
						\foreach \ax/\ay/\az/\bx/\by/\bz in {
								-0.47596314947796792/-0.33655677059077743/0.8125199200687454/0.21513724867401407/-0.47862549063280968/0.85125413593678212, 
								-0.47596314947796792/-0.33655677059077743/0.8125199200687454/-0.8125199200687454/0.33655677059077743/0.47596314947796792, 
								-0.47596314947796792/-0.33655677059077743/0.8125199200687454/-0.47596314947796792/-0.8125199200687454/0.33655677059077743, 
								0.47596314947796792/0.33655677059077743/-0.8125199200687454/-0.21513724867401407/0.47862549063280968/-0.85125413593678212, 
								0.47596314947796792/0.33655677059077743/-0.8125199200687454/0.8125199200687454/-0.33655677059077743/-0.47596314947796792, 
								0.47596314947796792/0.33655677059077743/-0.8125199200687454/0.47596314947796792/0.8125199200687454/-0.33655677059077743, 
								0.21513724867401407/-0.47862549063280968/0.85125413593678212/0.68455031944529765/-0.25056280708573159/0.68455031944529765, 
								0.21513724867401407/-0.47862549063280968/0.85125413593678212/0.36700100107065714/-0.85476343535873778/0.36700100107065714, 
								-0.21513724867401407/0.47862549063280968/-0.85125413593678212/-0.68455031944529765/0.25056280708573159/-0.68455031944529765, 
								-0.21513724867401407/0.47862549063280968/-0.85125413593678212/-0.36700100107065714/0.85476343535873778/-0.36700100107065714, 
								-0.8125199200687454/0.33655677059077743/0.47596314947796792/-0.85125413593678212/0.47862549063280968/-0.21513724867401407, 
								-0.8125199200687454/0.33655677059077743/0.47596314947796792/-0.33655677059077743/0.8125199200687454/0.47596314947796792, 
								0.8125199200687454/-0.33655677059077743/-0.47596314947796792/0.85125413593678212/-0.47862549063280968/0.21513724867401407, 
								0.8125199200687454/-0.33655677059077743/-0.47596314947796792/0.33655677059077743/-0.8125199200687454/-0.47596314947796792, 
								-0.47596314947796792/-0.8125199200687454/0.33655677059077743/0.36700100107065714/-0.85476343535873778/0.36700100107065714, 
								-0.47596314947796792/-0.8125199200687454/0.33655677059077743/-0.57735026918962573/-0.57735026918962573/-0.57735026918962573, 
								0.47596314947796792/0.8125199200687454/-0.33655677059077743/-0.36700100107065714/0.85476343535873778/-0.36700100107065714, 
								0.47596314947796792/0.8125199200687454/-0.33655677059077743/0.57735026918962573/0.57735026918962573/0.57735026918962573, 
								0.68455031944529765/-0.25056280708573159/0.68455031944529765/0.85125413593678212/-0.47862549063280968/0.21513724867401407, 
								0.68455031944529765/-0.25056280708573159/0.68455031944529765/0.57735026918962573/0.57735026918962573/0.57735026918962573, 
								-0.68455031944529765/0.25056280708573159/-0.68455031944529765/-0.85125413593678212/0.47862549063280968/-0.21513724867401407, 
								-0.68455031944529765/0.25056280708573159/-0.68455031944529765/-0.57735026918962573/-0.57735026918962573/-0.57735026918962573, 
								0.36700100107065714/-0.85476343535873778/0.36700100107065714/0.85125413593678212/-0.47862549063280968/0.21513724867401407, 
								0.36700100107065714/-0.85476343535873778/0.36700100107065714/0.33655677059077743/-0.8125199200687454/-0.47596314947796792, 
								-0.36700100107065714/0.85476343535873778/-0.36700100107065714/-0.85125413593678212/0.47862549063280968/-0.21513724867401407, 
								-0.36700100107065714/0.85476343535873778/-0.36700100107065714/-0.33655677059077743/0.8125199200687454/0.47596314947796792, 
								-0.33655677059077743/0.8125199200687454/0.47596314947796792/0.57735026918962573/0.57735026918962573/0.57735026918962573, 
								0.33655677059077743/-0.8125199200687454/-0.47596314947796792/-0.57735026918962573/-0.57735026918962573/-0.57735026918962573 
							}{
								\GCArcABfb{\ax}{\ay}{\az}{\bx}{\by}{\bz}{}
							}
						\foreach \ax/\ay/\az/\bx/\by/\bz in {				
								0.47596314947796792/0.33655677059077743/-0.8125199200687454/-0.21513724867401407/0.47862549063280968/-0.85125413593678212, 
								0.47596314947796792/0.33655677059077743/-0.8125199200687454/0.47596314947796792/0.8125199200687454/-0.33655677059077743, 
								0.68455031944529765/-0.25056280708573159/0.68455031944529765/0.85125413593678212/-0.47862549063280968/0.21513724867401407, 
								0.36700100107065714/-0.85476343535873778/0.36700100107065714/0.85125413593678212/-0.47862549063280968/0.21513724867401407, 
								0.36700100107065714/-0.85476343535873778/0.36700100107065714/0.33655677059077743/-0.8125199200687454/-0.47596314947796792, 
								0.33655677059077743/-0.8125199200687454/-0.47596314947796792/-0.57735026918962573/-0.57735026918962573/-0.57735026918962573 
							}{
								\GCArcABfb{\ax}{\ay}{\az}{\bx}{\by}{\bz}{thick}
							}
						\node (B) at (0.8125199200687454,-0.33655677059077743,-0.47596314947796792) {};
						\draw[line width = 0.1mm, fill=red!50!black] (B) circle[radius=1.5pt];
						\node at (B) {{\color{white}\scriptsize{$B$}}};
					}
				
				\node[tdplot_screen_coords] at (1.25,0) {$\to$};
			\end{tikzpicture}
		%
		\tdplotsetmaincoords{60}{60}
		\begin{tikzpicture}[scale=2.75,tdplot_main_coords]
				
				
				\draw[ball color = gray!40, opacity = 0.25] (0,0,0) circle (1cm);
				
				\def\OMcolor{blue!75!black}
				\def\Ccolor{red!50!blue}
				{
						\foreach \i/\vx/\vy/\vz in {
								11/0.27470724490123433/0.60687867481766011/-0.74581110453675459, 
								3/0.98127077961345444/0.16562279925117029/-0.098370450060014367, 
								6/0.68421145908413106/-0.56955779747676116/0.45547622834938933, 
								10/0.58625691928517709/-0.80913538945506969/0.040034311804323403, 
								8/0.053691776889334054/-0.95491834300865708/-0.29197286051971855, 
								7/0.0072131982308458941/-0.20692549617572864/-0.97833011238728163}{
								\node (\i) at (\vx,\vy,\vz) {};
								\POnSfb{\vx}{\vy}{\vz}{color=\OMcolor}
							}
						\foreach \ax/\ay/\az/\bx/\by/\bz in {
								0.0072131982308458941/-0.20692549617572864/-0.97833011238728163/0.053691776889334054/-0.95491834300865708/-0.29197286051971855, 
								0.0072131982308458941/-0.20692549617572864/-0.97833011238728163/0.27470724490123433/0.60687867481766011/-0.74581110453675459, 
								0.053691776889334054/-0.95491834300865708/-0.29197286051971855/0.58625691928517709/-0.80913538945506969/0.040034311804323403, 
								0.98127077961345444/0.16562279925117029/-0.098370450060014367/0.68421145908413106/-0.56955779747676116/0.45547622834938933, 
								0.98127077961345444/0.16562279925117029/-0.098370450060014367/0.27470724490123433/0.60687867481766011/-0.74581110453675459, 
								0.68421145908413106/-0.56955779747676116/0.45547622834938933/0.58625691928517709/-0.80913538945506969/0.040034311804323403 
							}{
								\GCArcABfb{\ax}{\ay}{\az}{\bx}{\by}{\bz}{line width = 0.5mm, color=\OMcolor}
							}

						\foreach \i/\vx/\vy/\vz in {
								1/0.756929/-0.434287/-0.488316,
								2/-0.756929/0.434287/0.488316
							}{
								\node (\i) at (\vx,\vy,\vz) {};
								\POnSfb{\vx}{\vy}{\vz}{color=\Ccolor}
							}
						\foreach \ax/\ay/\az/\bx/\by/\bz in {
								0.756929/-0.434287/-0.488316/0.0072131982308458941/-0.20692549617572864/-0.97833011238728163, 
								0.756929/-0.434287/-0.488316/0.98127077961345444/0.16562279925117029/-0.098370450060014367, 
								0.756929/-0.434287/-0.488316/0.58625691928517709/-0.80913538945506969/0.040034311804323403, 
								-0.756929/0.434287/0.488316/0.053691776889334054/-0.95491834300865708/-0.2919728605197185,
								-0.756929/0.434287/0.488316/0.0072131982308458941/-0.20692549617572864/-0.97833011238728163,
								-0.756929/0.434287/0.488316/0.27470724490123433/0.60687867481766011/-0.74581110453675459,
								-0.756929/0.434287/0.488316/0.98127077961345444/0.16562279925117029/-0.098370450060014367,
								-0.756929/0.434287/0.488316/0.68421145908413106/-0.56955779747676116/0.45547622834938933
							}{
								\GCArcABfb{\ax}{\ay}{\az}{\bx}{\by}{\bz}{line width = 0.5mm, color=\Ccolor}
							}
					}
				%
				
				
				{\color{red!75!black}
						\foreach \i/\vx/\vy/\vz in {
								2/0.47596314947796792/0.33655677059077743/-0.8125199200687454, 
								4/-0.21513724867401407/0.47862549063280968/-0.85125413593678212, 
								8/0.47596314947796792/0.8125199200687454/-0.33655677059077743, 
								9/0.68455031944529765/-0.25056280708573159/0.68455031944529765, 
								11/0.36700100107065714/-0.85476343535873778/0.36700100107065714, 
								14/0.85125413593678212/-0.47862549063280968/0.21513724867401407, 
								16/0.33655677059077743/-0.8125199200687454/-0.47596314947796792, 
								17/-0.57735026918962573/-0.57735026918962573/-0.57735026918962573 
							}{
								\node (\i) at (\vx,\vy,\vz) {};
								\POnSfb{\vx}{\vy}{\vz}{}
							}
						%
						\foreach \ax/\ay/\az/\bx/\by/\bz in {				
								0.47596314947796792/0.33655677059077743/-0.8125199200687454/-0.21513724867401407/0.47862549063280968/-0.85125413593678212, 
								0.47596314947796792/0.33655677059077743/-0.8125199200687454/0.47596314947796792/0.8125199200687454/-0.33655677059077743, 
								0.68455031944529765/-0.25056280708573159/0.68455031944529765/0.85125413593678212/-0.47862549063280968/0.21513724867401407, 
								0.36700100107065714/-0.85476343535873778/0.36700100107065714/0.85125413593678212/-0.47862549063280968/0.21513724867401407, 
								0.36700100107065714/-0.85476343535873778/0.36700100107065714/0.33655677059077743/-0.8125199200687454/-0.47596314947796792, 
								0.33655677059077743/-0.8125199200687454/-0.47596314947796792/-0.57735026918962573/-0.57735026918962573/-0.57735026918962573 
							}{
								\GCArcABfb{\ax}{\ay}{\az}{\bx}{\by}{\bz}{thick}
							}
					}	
				
				{\color{green!50!black}
						\foreach \ax/\ay/\az/\bx/\by/\bz in {		
								0.47596314947796792/0.33655677059077743/-0.8125199200687454/0.85125413593678212/-0.47862549063280968/0.21513724867401407, 
								0.47596314947796792/0.33655677059077743/-0.8125199200687454/0.33655677059077743/-0.8125199200687454/-0.47596314947796792, 
								0.85125413593678212/-0.47862549063280968/0.21513724867401407/0.33655677059077743/-0.8125199200687454/-0.47596314947796792, 
								-0.21513724867401407/0.47862549063280968/-0.85125413593678212/0.47596314947796792/0.8125199200687454/-0.33655677059077743, 
								0.47596314947796792/0.8125199200687454/-0.33655677059077743/0.68455031944529765/-0.25056280708573159/0.68455031944529765, 
								0.68455031944529765/-0.25056280708573159/0.68455031944529765/0.36700100107065714/-0.85476343535873778/0.36700100107065714, 
								0.36700100107065714/-0.85476343535873778/0.36700100107065714/-0.57735026918962573/-0.57735026918962573/-0.57735026918962573, 
								-0.21513724867401407/0.47862549063280968/-0.85125413593678212/-0.57735026918962573/-0.57735026918962573/-0.57735026918962573 
							}{
								\GCArcABfb{\ax}{\ay}{\az}{\bx}{\by}{\bz}{thick}
							}
					}
				
			\end{tikzpicture}	
		\caption{The boundary complex $\Gamma = \partial\Lc(\Tc_{B,\leq 1})\setminus \{\Zero\}$ (thick blue), the double cone construction $\Sigma_{[1,2]}$ 
			from Step 1 of the proof of Theorem \ref{thm:rkBsdLcCWposet} and its dual (red  and green)  homeomorphic to $\rk_B\sd\Lc^\dual_{<\Zero^B_{[1,2]}}$.}
		\label{fig:A_sdCpx1}
	\end{figure}
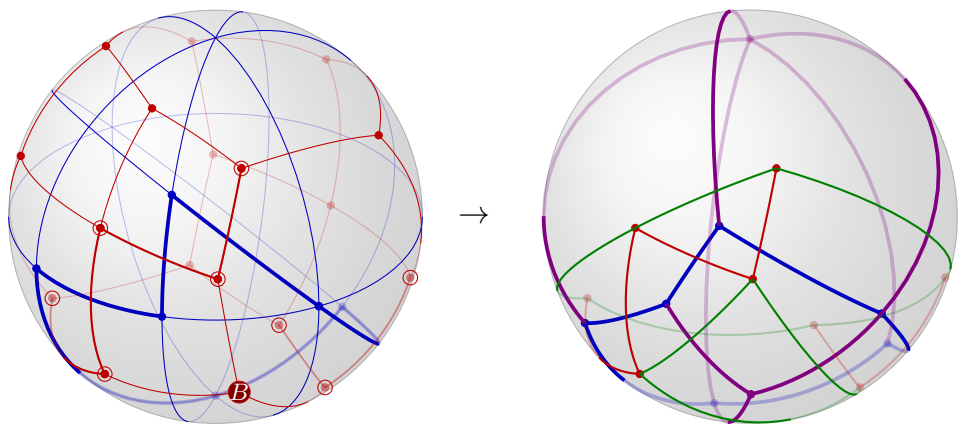

	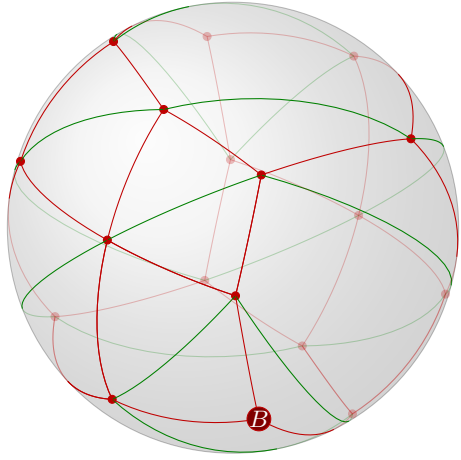
\begin{figure}	
 draw the oriented matroid cpx and its dual of xyz(x+y)(y+z)
		\tdplotsetmaincoords{60}{60}
		\begin{tikzpicture}[scale=3,tdplot_main_coords]
				
				
				\draw[ball color = gray!40, opacity = 0.25] (0,0,0) circle (1cm);
				
				
				\def\dualOMcolor{red!75!black}
				{\color{\dualOMcolor}
						\foreach \i/\vx/\vy/\vz in {
								1/-0.47596314947796792/-0.33655677059077743/0.8125199200687454, 
								2/0.47596314947796792/0.33655677059077743/-0.8125199200687454, 
								3/0.21513724867401407/-0.47862549063280968/0.85125413593678212, 
								4/-0.21513724867401407/0.47862549063280968/-0.85125413593678212, 
								5/-0.8125199200687454/0.33655677059077743/0.47596314947796792, 
								7/-0.47596314947796792/-0.8125199200687454/0.33655677059077743, 
								8/0.47596314947796792/0.8125199200687454/-0.33655677059077743, 
								9/0.68455031944529765/-0.25056280708573159/0.68455031944529765, 
								10/-0.68455031944529765/0.25056280708573159/-0.68455031944529765, 
								11/0.36700100107065714/-0.85476343535873778/0.36700100107065714, 
								12/-0.36700100107065714/0.85476343535873778/-0.36700100107065714, 
								13/-0.85125413593678212/0.47862549063280968/-0.21513724867401407, 
								14/0.85125413593678212/-0.47862549063280968/0.21513724867401407, 
								15/-0.33655677059077743/0.8125199200687454/0.47596314947796792, 
								16/0.33655677059077743/-0.8125199200687454/-0.47596314947796792, 
								17/-0.57735026918962573/-0.57735026918962573/-0.57735026918962573, 
								18/0.57735026918962573/0.57735026918962573/0.57735026918962573	
							}{
								\node (\i) at (\vx,\vy,\vz) {};
								\POnSfb{\vx}{\vy}{\vz}{}
							}
						
						{\color{green!50!black}
								\foreach \ax/\ay/\az/\bx/\by/\bz in {		
										0.47596314947796792/0.33655677059077743/-0.8125199200687454/0.85125413593678212/-0.47862549063280968/0.21513724867401407, 
										0.47596314947796792/0.33655677059077743/-0.8125199200687454/0.33655677059077743/-0.8125199200687454/-0.47596314947796792, 
										0.85125413593678212/-0.47862549063280968/0.21513724867401407/0.33655677059077743/-0.8125199200687454/-0.47596314947796792, 
										-0.21513724867401407/0.47862549063280968/-0.85125413593678212/0.47596314947796792/0.8125199200687454/-0.33655677059077743, 
										0.47596314947796792/0.8125199200687454/-0.33655677059077743/0.68455031944529765/-0.25056280708573159/0.68455031944529765, 
										0.68455031944529765/-0.25056280708573159/0.68455031944529765/0.36700100107065714/-0.85476343535873778/0.36700100107065714, 
										0.36700100107065714/-0.85476343535873778/0.36700100107065714/-0.57735026918962573/-0.57735026918962573/-0.57735026918962573, 
										-0.21513724867401407/0.47862549063280968/-0.85125413593678212/-0.57735026918962573/-0.57735026918962573/-0.57735026918962573, 
										0.21513724867401407/-0.47862549063280968/0.85125413593678212/-0.47596314947796792/-0.8125199200687454/0.33655677059077743, 
										-0.47596314947796792/-0.8125199200687454/0.33655677059077743/-0.68455031944529765/0.25056280708573159/-0.68455031944529765, 
										-0.68455031944529765/0.25056280708573159/-0.68455031944529765/-0.36700100107065714/0.85476343535873778/-0.3670010010706571, 
										-0.36700100107065714/0.85476343535873778/-0.36700100107065714/0.57735026918962573/0.57735026918962573/0.57735026918962573, 
										0.21513724867401407/-0.47862549063280968/0.85125413593678212/0.57735026918962573/0.57735026918962573/0.57735026918962573, 
										-0.47596314947796792/-0.33655677059077743/0.8125199200687454/-0.85125413593678212/0.47862549063280968/-0.21513724867401407, 
										-0.85125413593678212/0.47862549063280968/-0.21513724867401407/-0.33655677059077743/0.8125199200687454/0.47596314947796792, 
										-0.47596314947796792/-0.33655677059077743/0.8125199200687454/-0.33655677059077743/0.8125199200687454/0.47596314947796792 
									}{
										\GCArcABfb{\ax}{\ay}{\az}{\bx}{\by}{\bz}{}
									}
							}
						
						%
						\foreach \ax/\ay/\az/\bx/\by/\bz in {
								-0.47596314947796792/-0.33655677059077743/0.8125199200687454/0.21513724867401407/-0.47862549063280968/0.85125413593678212, 
								-0.47596314947796792/-0.33655677059077743/0.8125199200687454/-0.8125199200687454/0.33655677059077743/0.47596314947796792, 
								-0.47596314947796792/-0.33655677059077743/0.8125199200687454/-0.47596314947796792/-0.8125199200687454/0.33655677059077743, 
								0.47596314947796792/0.33655677059077743/-0.8125199200687454/-0.21513724867401407/0.47862549063280968/-0.85125413593678212, 
								0.47596314947796792/0.33655677059077743/-0.8125199200687454/0.8125199200687454/-0.33655677059077743/-0.47596314947796792, 
								0.47596314947796792/0.33655677059077743/-0.8125199200687454/0.47596314947796792/0.8125199200687454/-0.33655677059077743, 
								0.21513724867401407/-0.47862549063280968/0.85125413593678212/0.68455031944529765/-0.25056280708573159/0.68455031944529765, 
								0.21513724867401407/-0.47862549063280968/0.85125413593678212/0.36700100107065714/-0.85476343535873778/0.36700100107065714, 
								-0.21513724867401407/0.47862549063280968/-0.85125413593678212/-0.68455031944529765/0.25056280708573159/-0.68455031944529765, 
								-0.21513724867401407/0.47862549063280968/-0.85125413593678212/-0.36700100107065714/0.85476343535873778/-0.36700100107065714, 
								-0.8125199200687454/0.33655677059077743/0.47596314947796792/-0.85125413593678212/0.47862549063280968/-0.21513724867401407, 
								-0.8125199200687454/0.33655677059077743/0.47596314947796792/-0.33655677059077743/0.8125199200687454/0.47596314947796792, 
								0.8125199200687454/-0.33655677059077743/-0.47596314947796792/0.85125413593678212/-0.47862549063280968/0.21513724867401407, 
								0.8125199200687454/-0.33655677059077743/-0.47596314947796792/0.33655677059077743/-0.8125199200687454/-0.47596314947796792, 
								-0.47596314947796792/-0.8125199200687454/0.33655677059077743/0.36700100107065714/-0.85476343535873778/0.36700100107065714, 
								-0.47596314947796792/-0.8125199200687454/0.33655677059077743/-0.57735026918962573/-0.57735026918962573/-0.57735026918962573, 
								0.47596314947796792/0.8125199200687454/-0.33655677059077743/-0.36700100107065714/0.85476343535873778/-0.36700100107065714, 
								0.47596314947796792/0.8125199200687454/-0.33655677059077743/0.57735026918962573/0.57735026918962573/0.57735026918962573, 
								0.68455031944529765/-0.25056280708573159/0.68455031944529765/0.85125413593678212/-0.47862549063280968/0.21513724867401407, 
								0.68455031944529765/-0.25056280708573159/0.68455031944529765/0.57735026918962573/0.57735026918962573/0.57735026918962573, 
								-0.68455031944529765/0.25056280708573159/-0.68455031944529765/-0.85125413593678212/0.47862549063280968/-0.21513724867401407, 
								-0.68455031944529765/0.25056280708573159/-0.68455031944529765/-0.57735026918962573/-0.57735026918962573/-0.57735026918962573, 
								0.36700100107065714/-0.85476343535873778/0.36700100107065714/0.85125413593678212/-0.47862549063280968/0.21513724867401407, 
								0.36700100107065714/-0.85476343535873778/0.36700100107065714/0.33655677059077743/-0.8125199200687454/-0.47596314947796792, 
								-0.36700100107065714/0.85476343535873778/-0.36700100107065714/-0.85125413593678212/0.47862549063280968/-0.21513724867401407, 
								-0.36700100107065714/0.85476343535873778/-0.36700100107065714/-0.33655677059077743/0.8125199200687454/0.47596314947796792, 
								-0.33655677059077743/0.8125199200687454/0.47596314947796792/0.57735026918962573/0.57735026918962573/0.57735026918962573, 
								0.33655677059077743/-0.8125199200687454/-0.47596314947796792/-0.57735026918962573/-0.57735026918962573/-0.57735026918962573 
							}{
								\GCArcABfb{\ax}{\ay}{\az}{\bx}{\by}{\bz}{}
							}
						\foreach \ax/\ay/\az/\bx/\by/\bz in {				
								0.47596314947796792/0.33655677059077743/-0.8125199200687454/-0.21513724867401407/0.47862549063280968/-0.85125413593678212, 
								0.47596314947796792/0.33655677059077743/-0.8125199200687454/0.47596314947796792/0.8125199200687454/-0.33655677059077743, 
								0.68455031944529765/-0.25056280708573159/0.68455031944529765/0.85125413593678212/-0.47862549063280968/0.21513724867401407, 
								0.36700100107065714/-0.85476343535873778/0.36700100107065714/0.85125413593678212/-0.47862549063280968/0.21513724867401407, 
								0.36700100107065714/-0.85476343535873778/0.36700100107065714/0.33655677059077743/-0.8125199200687454/-0.47596314947796792, 
								0.33655677059077743/-0.8125199200687454/-0.47596314947796792/-0.57735026918962573/-0.57735026918962573/-0.57735026918962573 
							}{
								\GCArcABfb{\ax}{\ay}{\az}{\bx}{\by}{\bz}{}
							}
						\node (B) at (0.8125199200687454,-0.33655677059077743,-0.47596314947796792) {};
						\draw[line width = 0.1mm, fill=red!50!black] (B) circle[radius=1.5pt];
						\node at (B) {{\color{white}\scriptsize{$B$}}};
					}
			\end{tikzpicture}
		\caption{The $B$-rank subdivision complex $\rk_B\sd\Lc^\dual$ of the arrangement $\Ac$
			with defining polynomial $xyz(x+y)(y+z)$.}
		\label{fig:A_sdOMFullCpxs}
	\end{figure}
	
	\begin{example}
		\label{ex:A3_rkBsd_Oonstr}
		Let $\Ac$ be the arrangement with defining polynomial
		$Q(\Ac) = xyz(x+y)(y+z)$.
		The regular PL subdivision of the sphere induced by $\Ac$ whose face poset is $\Lc(\Ac) \setminus \{\Zero\}$ 
		and its dual are displayed in Figure \ref{fig:A_OMcpxs}.
		
		The double cone construction in the first step in the proof of Theorem \ref{thm:rkBsdLcCWposet}
		for the arrangement $\Ac$ with respect to the base chamber $B$ and $k=1$ 
		is illustrated in Figure \ref{fig:A_sdCpx1}.
		The additional edges of the two cones over the boundary complex $\Gamma$ glued together along $\Gamma$ are presented by the purple rays.
		They are dual to the additional green edges in the dual complex $\Sigma_{[1,2]}^\dual$ whose face poset
		is isomorphic to $\rk_B\sd\Lc_{< \Zero^B_{[1,2]}}^\dual$.
		
		The full subdivided complex $\rk_B\sd\Lc^\dual$ with respect to the base chamber $B$ is then displayed in Figure \ref{fig:A_sdOMFullCpxs}.
	\end{example}
	
	\begin{corollary}
		\label{coro:HomeoRkBSubdiv}
		There is a PL homeomorphism $\varphi:|\rk_B\sd\Lc^\dual| \to |\Lc^\dual|$ which restricts to a PL homeomorphism 
		$\varphi|_{\varphi^{-1}(|\Lc^\dual_{\leq x}|)}:\varphi^{-1}(|\Lc^\dual_{\leq x}|) \to |\Lc^\dual_{\leq x}|$ for any $x \in \Lc^\dual$.
		Thus, we can regard $\rk_B\sd\Lc^\dual$ as a PL subdivision of the dual oriented matroid complex $\Lc^\dual$
		via $\rk_B\sd\Lc^\dual \ni x \mapsto \varphi(|x|) \subseteq |p_B(x)|$.
	\end{corollary}
	
	\begin{proof}
		We argue by induction on $d=\rk(\OM)$. For $d=1$ we have $\rk_B\sd\Lc^\dual = \Lc^\dual$ and the statement is trivially true.
		
		Now, let $\rk(\Lc) = d \geq 2$.
		Note that for $x \leLc \Zero$, $\Lc^\dual_{\leq x}$ is the dual oriented matroid complex of an oriented matroid of rank $d-1$.
		By the induction hypothesis we have PL homeomorphisms $\varphi_x:|\rk_B\sd(\Lc_{\leq x})| \to |\Lc^\dual_{\leq x}|$
		which restrict to PL homeomorphisms 
		$\varphi_x|_{\varphi_x^{-1}(|\Lc^\dual_{\leq y}|)}:\varphi_x^{-1}(|\Lc^\dual_{\leq y}|) \to |\Lc^\dual_{\leq y}|$ for $y \leq x$
		and can thus be glued together to a PL homeomorphism $\varphi':\partial|\rk_B\sd\Lc| \to \partial|\Lc^\dual|$ between spheres.
		By \cite[Lem.~1.10]{RourkeSanderson1972_PLTop}, $\varphi'$ extends to our desired PL homeomorphism $\varphi:|\rk_B\sd\Lc^\dual| \to |\Lc^\dual|$.
	\end{proof}

	Now, we construct our subdivision of the Salvetti complex.
	\begin{definition}
		\label{def:TopeRankSubdivSalvetti}
		Let $\Sc = \Sc(\OM)$ be the Salvetti complex of $\OM$.
		Then the \emph{tope-rank subdivision} of a cell $(\sigma,T) \in \Sc$ is defined as:
		\[
		\rk\sd(\sigma,T) := \{ (\afr,T) \mid \afr \in \rk_T\sd(\sigma))\},
		\]
		and the \emph{tope-rank subdivision} of $\Sc$ is defined by
		\[
		\rksdS:= \bigcup\limits_{x \in \Sc} \rk\sd(x),
		\]
		with partial order given by
		\[
		(\afr,T) \leq (\bfr,R) :\iff \afr\subseteq \bfr \text{ and } p_T(\afr)\circ R = T.
		\]
		
		We also have a poset map $\wt{p}:\rksdS\to \Sc, (\afr,T) \mapsto (p_T(\afr),T)$.
	\end{definition}

	\begin{example}
		\label{ex:TopeRkSubdivBoolean}
		Recall from Example \ref{ex:SalBooleanParam}
		the explicit parametrization of the cells in $\Sc(\Bc_n)$:
		$\varphi((\sigma,T)) = 
		\{(\exp(T_1(t_1-1)\pi i),\ldots,\exp(T_n(t_n-1)\pi i)) \mid (t_1,\ldots,t_n) \in \psi(\sigma)\}$
		where
		$\psi(\sigma) = \{x \in [0,1]^n \mid x_i=0$ if $\sigma_i=-$, $x_i=1$ if $\sigma_i=+\} \subseteq [0,1]^n$
		for $\sigma \in \Lc^\dual$.
		This yields the inclusion $\varphi:|\Sc(\Bc_n)| \hookrightarrow \Xf(\Bc_n)$ which is a homotopy equivalence.
		
		The rank subdivision of $\Sc(\Bc_n)$
		can be similarly parametrized.
		
		Firstly, we have the following explicit description of the subdivision of $\sigma \in \Lc^\dual$:
		a cell $\afr = \sigma^B_k \in  \rk_B\sd\Lc^\dual$ is identified with
		$\wt{\psi}(\afr):=\psi(\sigma) \cap \{\sum_{i=1}^n x_i = k\} \subseteq [0,1]^n$;
		a cell $\bfr = \sigma^B_{[k,k+1]}$ is identified with 
		$\wt{\psi}(\bfr):=\psi(\sigma) \cap \{k \leq \sum_{i=1}^n x_i \leq k+1\} \subseteq [0,1]^n$.
		Analogous to Example \ref{ex:SalBooleanParam},
		the parametrization of the tope-rank subdivision is defined by identifying a cell $(\afr,T) \in \rksdS(\Bc_n)$
		with 
		\begin{align*}
			&\wt{\varphi}((\afr,T)) := \\
			&\quad \{(\exp(T_1(t_1-1)\pi i),\ldots,\exp(T_n(t_n-1)\pi i)) \mid (t_1,\ldots,t_n) \in \wt{\psi}(\afr)\}.
		\end{align*}
		
	\end{example}

	\begin{theorem}
		\label{thm:rksdSubdivMap}
		The poset $\rksdS$ is the face poset of a regular cell complex $PL$-homeomorphic to $\Sc$ 
		and if $\Sc = \Sc(\Ac)$ is the Salvetti-complex of a real arrangement $\Ac$,
		then $|\rksdS|$ is homotopy equivalent to the complexified complement $\Xf(\Ac)$. 
	\end{theorem}
	\begin{proof}
		The subdivision $\rk_B\sd\Lc^\dual$ of $\Lc^\dual$ which constitute the maximal cells in $\Sc$ by Lemma \ref{lem:SalPrincipialIdealCovectors} respects 
		their gluing in $\Sc$ respectively $\rksdS$. Hence, the statement follows with Theorem \ref{thm:rkBsdLcCWposet},
		Corollary \ref{coro:HomeoRkBSubdiv} and the usual homeomorphisms gluing argument.
		By Theorem \ref{thm:SalvettiHoEquiv}, $|\Sc|$ is homotopy equivalent to $\Xf(\Ac)$ and so is $|\rksdS|$.
	\end{proof}


	\section{Milnor fibration of an oriented matroid}
	\label{sec:MilnorFibOM}
	
	Let $\Cc$ denote the Salvetti complex of the rank $1$ oriented matroid with covectors $\{+,-,0\}$,
	i.e.\ the face poset of $\Cc$ is given by the set $\{(+,+),(0,+),(-,-),(0,-)\}$, a regular cell decomposition of the circle $S^1$.

	\begin{definition}
		\label{def:MilnorFibrationOM}
		Define the map $Q:\Tc \to \{+,-\}, T \mapsto \prod_{e \in E}T_e$,
		and the poset map $\wt{Q}:\rksdS \to \Cc$ by:
		\[
		\wt{Q}((\sigma^T_k,T)) := \begin{cases}
			(+,+) &\text{ if } Q(\sigma^T_k) = \{+\},\\
			(-,-) &\text{ if } Q(\sigma^T_k) = \{-\},
		\end{cases}
		\]
		and
		\[
		\wt{Q}((\sigma^T_{[k,k+1]},T)) := \begin{cases}
			(0,+) &\text{ if } Q(\sigma^T_k) = \{+\},\\
			(0,-) &\text{ if } Q(\sigma^T_k) = \{-\}.
		\end{cases}
		\]
		
		We define the \emph{(combinatorial) Milnor fiber} of $\OM$ by 
		$$\wt{\Ff}(\OM) := \wt{Q}^{-1}((+,+)).$$
		(It is a consequence of the  Theorem \ref{thm:MilnorOMPosetQuasiFib} below that $\wt{Q}^{-1}((+,+))$ and $\wt{Q}^{-1}((-,-))$ are homotopy equivalent.)
	\end{definition}

	\begin{figure}
		\begin{tikzpicture}[scale=0.6]		
			
			\node (l1) at (1.7320508075688774,-1.) {};
			\node (l2) at (1.7320508075688774,1.) {};
			\node (l3) at (0,2.) {};
			\node (l4) at (-1.7320508075688774,1.) {};
			\node (l5) at (-1.7320508075688774,-1.) {};
			\node (l6) at (0,-2.) {};
			\node (zero) at (0,0) {};
			
			
			\node (l1_0) at ($0.5*(1.7320508075688774,-1.)$) {};
			\node (l2_0) at ($0.5*(1.7320508075688774,1.)$) {};
			\node (l3_0) at ($0.5*(0,2.)$) {};
			\node (l4_0) at ($0.5*(-1.7320508075688774,1.)$){};
			\node (l5_0) at ($0.5*(-1.7320508075688774,-1.)$) {};
			\node (l6_0) at ($0.5*(0,-2.)$) {};
			
			\draw[->,shorten >=3pt, shorten <=3pt] ($1.25*(l1_0)$) .. controls ($0.5*(l1_0)+0.5*(l2_0)$)  .. ($1.25*(l2_0)$);;
			\draw[->,shorten >=3pt, shorten <=3pt, color=black!50!green] ($1.25*(l2_0)$) .. controls ($0.75*(l1_0)+0.75*(l2_0)$)  .. ($1.25*(l1_0)$);;
			\draw[->,shorten >=3pt, shorten <=3pt, color=black!50!green] ($1.25*(l2_0)$) .. controls ($0.5*(l2_0)+0.5*(l3_0)$)  .. ($1.25*(l3_0)$);;
			\draw[->,shorten >=3pt, shorten <=3pt] ($1.25*(l3_0)$) .. controls ($0.75*(l2_0)+0.75*(l3_0)$)  .. ($1.25*(l2_0)$);;
			\draw[->,shorten >=3pt, shorten <=3pt] ($1.25*(l3_0)$) .. controls ($0.5*(l3_0)+0.5*(l4_0)$)  .. ($1.25*(l4_0)$);;
			\draw[->,shorten >=3pt, shorten <=3pt, color=black!50!green] ($1.25*(l4_0)$) .. controls ($0.75*(l3_0)+0.75*(l4_0)$)  .. ($1.25*(l3_0)$);;
			\draw[->,shorten >=3pt, shorten <=3pt, color=black!50!green] ($1.25*(l4_0)$) .. controls ($0.5*(l4_0)+0.5*(l5_0)$)  .. ($1.25*(l5_0)$);;
			\draw[->,shorten >=3pt, shorten <=3pt] ($1.25*(l5_0)$) .. controls ($0.75*(l4_0)+0.75*(l5_0)$)  .. ($1.25*(l4_0)$);;
			\draw[->,shorten >=3pt, shorten <=3pt] ($1.25*(l5_0)$) .. controls ($0.5*(l5_0)+0.5*(l6_0)$)  .. ($1.25*(l6_0)$);;
			\draw[->,shorten >=3pt, shorten <=3pt, color=black!50!green] ($1.25*(l6_0)$) .. controls ($0.75*(l5_0)+0.75*(l6_0)$)  .. ($1.25*(l5_0)$);;
			\draw[->,shorten >=3pt, shorten <=3pt, color=black!50!green] ($1.25*(l6_0)$) .. controls ($0.5*(l6_0)+0.5*(l1_0)$)  .. ($1.25*(l1_0)$);;
			\draw[->,shorten >=3pt, shorten <=3pt] ($1.25*(l1_0)$) .. controls ($0.75*(l6_0)+0.75*(l1_0)$)  .. ($1.25*(l6_0)$);;
			\filldraw[color=blue] ($1.25*(l1_0)$) circle[radius=2pt];
			\filldraw[color=red] ($1.25*(l2_0)$) circle[radius=2pt];
			\filldraw[color=blue] ($1.25*(l3_0)$) circle[radius=2pt];
			\filldraw[color=red] ($1.25*(l4_0)$) circle[radius=2pt];
			\filldraw[color=blue] ($1.25*(l5_0)$) circle[radius=2pt];
			\filldraw[color=red] ($1.25*(l6_0)$) circle[radius=2pt];
			
			
			\node (l1_1) at ($1.5*(l1)+0.5*(1.7320508075688774,-1.)$) {};
			\node (l2_1) at ($1.5*(l1)+0.5*(1.7320508075688774,1.)$) {};
			\node (l3_1) at ($1.5*(l1)+0.5*(0,2.)$) {};
			\node (l4_1) at ($1.5*(l1)+0.5*(-1.7320508075688774,1.)$){};
			\node (l5_1) at ($1.5*(l1)+0.5*(-1.7320508075688774,-1.)$) {};
			\node (l6_1) at ($1.5*(l1)+0.5*(0,-2.)$) {};
			
			\draw[->,shorten >=3pt, shorten <=3pt] ($1.25*(l1_1)$) -- ($1.25*(l2_1)$);
			\draw[->,shorten >=3pt, shorten <=3pt, color=black!50!green] ($1.25*(l2_1)$) -- ($1.25*(l3_1)$);
			\draw[->,shorten >=3pt, shorten <=3pt] ($1.25*(l3_1)$) -- ($1.25*(l4_1)$);
			\draw[->,shorten >=3pt, shorten <=3pt] ($1.25*(l1_1)$) -- ($1.25*(l6_1)$);
			\draw[->,shorten >=3pt, shorten <=3pt, color=black!50!green] ($1.25*(l6_1)$) -- ($1.25*(l5_1)$);
			\draw[->,shorten >=3pt, shorten <=3pt] ($1.25*(l5_1)$) -- ($1.25*(l4_1)$);
			\fill[color=black!50!green, opacity=0.2] ($1.25*(l2_1)$) -- ($1.25*(l3_1)$) -- ($1.25*(l5_1)$) -- ($1.25*(l6_1)$);
			\fill[color=black, opacity=0.2] ($1.25*(l1_1)$) -- ($1.25*(l2_1)$) -- ($1.25*(l6_1)$);
			\fill[color=black, opacity=0.2] ($1.25*(l3_1)$) -- ($1.25*(l4_1)$) -- ($1.25*(l5_1)$);
			\filldraw[color=blue] ($1.25*(l1_1)$) circle[radius=2pt];
			\draw ($1.25*(l1_1)$) circle[radius=4pt];
			\filldraw[color=red] ($1.25*(l2_1)$) circle[radius=2pt];
			\filldraw[color=blue] ($1.25*(l3_1)$) circle[radius=2pt];
			\filldraw[color=red] ($1.25*(l4_1)$) circle[radius=2pt];
			\filldraw[color=blue] ($1.25*(l5_1)$) circle[radius=2pt];
			\filldraw[color=red] ($1.25*(l6_1)$) circle[radius=2pt];
			\draw[color=red] ($1.25*(l6_1)$) -- ($1.25*(l2_1)$);
			\draw[color=blue] ($1.25*(l5_1)$) -- ($1.25*(l3_1)$);

			\draw[->] ($1.1*(l1)$) -- ($0.8*(l1)$);
			
			
			\node (l1_2) at ($1.5*(l2)+0.5*(1.7320508075688774,-1.)$) {};
			\node (l2_2) at ($1.5*(l2)+0.5*(1.7320508075688774,1.)$) {};
			\node (l3_2) at ($1.5*(l2)+0.5*(0,2.)$) {};
			\node (l4_2) at ($1.5*(l2)+0.5*(-1.7320508075688774,1.)$){};
			\node (l5_2) at ($1.5*(l2)+0.5*(-1.7320508075688774,-1.)$) {};
			\node (l6_2) at ($1.5*(l2)+0.5*(0,-2.)$) {};
			
			\draw[->,shorten >=3pt, shorten <=3pt, color=black!50!green] ($1.25*(l2_2)$) -- ($1.25*(l3_2)$);
			\draw[->,shorten >=3pt, shorten <=3pt] ($1.25*(l3_2)$) -- ($1.25*(l4_2)$);
			\draw[->,shorten >=3pt, shorten <=3pt, color=black!50!green] ($1.25*(l4_2)$) -- ($1.25*(l5_2)$);
			\draw[->,shorten >=3pt, shorten <=3pt, color=black!50!green] ($1.25*(l2_2)$) -- ($1.25*(l1_2)$);
			\draw[->,shorten >=3pt, shorten <=3pt] ($1.25*(l1_2)$) -- ($1.25*(l6_2)$);
			\draw[->,shorten >=3pt, shorten <=3pt, color=black!50!green] ($1.25*(l6_2)$) -- ($1.25*(l5_2)$);
			\fill[color=black!50!green, opacity=0.2] ($1.25*(l2_2)$) -- ($1.25*(l3_2)$) -- ($1.25*(l1_2)$);
			\fill[color=black!50!green, opacity=0.2] ($1.25*(l4_2)$) -- ($1.25*(l5_2)$) -- ($1.25*(l6_2)$);
			\fill[color=black, opacity=0.2] ($1.25*(l3_2)$) -- ($1.25*(l4_2)$) -- ($1.25*(l6_2)$) -- ($1.25*(l1_2)$);
			\filldraw[color=blue] ($1.25*(l1_2)$) circle[radius=2pt];
			\filldraw[color=red] ($1.25*(l2_2)$) circle[radius=2pt];
			\draw ($1.25*(l2_2)$) circle[radius=4pt];
			\filldraw[color=blue] ($1.25*(l3_2)$) circle[radius=2pt];
			\filldraw[color=red] ($1.25*(l4_2)$) circle[radius=2pt];
			\filldraw[color=blue] ($1.25*(l5_2)$) circle[radius=2pt];
			\filldraw[color=red] ($1.25*(l6_2)$) circle[radius=2pt];
			\draw[color=blue] ($1.25*(l1_2)$) -- ($1.25*(l3_2)$);
			\draw[color=red] ($1.25*(l6_2)$) -- ($1.25*(l4_2)$);
			
			\draw[->] ($1.1*(l2)$) -- ($0.8*(l2)$);
			
			
			\node (l1_3) at ($1.5*(l3)+0.5*(1.7320508075688774,-1.)$) {};
			\node (l2_3) at ($1.5*(l3)+0.5*(1.7320508075688774,1.)$) {};
			\node (l3_3) at ($1.5*(l3)+0.5*(0,2.)$) {};
			\node (l4_3) at ($1.5*(l3)+0.5*(-1.7320508075688774,1.)$){};
			\node (l5_3) at ($1.5*(l3)+0.5*(-1.7320508075688774,-1.)$) {};
			\node (l6_3) at ($1.5*(l3)+0.5*(0,-2.)$) {};
			
			\draw[->,shorten >=3pt, shorten <=3pt] ($1.25*(l3_3)$) -- ($1.25*(l4_3)$);
			\draw[->,shorten >=3pt, shorten <=3pt, color=black!50!green] ($1.25*(l4_3)$) -- ($1.25*(l5_3)$);
			\draw[->,shorten >=3pt, shorten <=3pt] ($1.25*(l5_3)$) -- ($1.25*(l6_3)$);
			\draw[->,shorten >=3pt, shorten <=3pt] ($1.25*(l3_3)$) -- ($1.25*(l2_3)$);
			\draw[->,shorten >=3pt, shorten <=3pt, color=black!50!green] ($1.25*(l2_3)$) -- ($1.25*(l1_3)$);
			\draw[->,shorten >=3pt, shorten <=3pt] ($1.25*(l1_3)$) -- ($1.25*(l6_3)$);
			\fill[color=black, opacity=0.2] ($1.25*(l3_3)$) -- ($1.25*(l4_3)$) -- ($1.25*(l2_3)$);
			\fill[color=black, opacity=0.2] ($1.25*(l5_3)$) -- ($1.25*(l6_3)$) -- ($1.25*(l1_3)$);
			\fill[color=black!50!green, opacity=0.2] ($1.25*(l4_3)$) -- ($1.25*(l5_3)$) -- ($1.25*(l1_3)$) -- ($1.25*(l2_3)$);
			\filldraw[color=blue] ($1.25*(l1_3)$) circle[radius=2pt];
			\filldraw[color=red] ($1.25*(l2_3)$) circle[radius=2pt];
			\filldraw[color=blue] ($1.25*(l3_3)$) circle[radius=2pt];
			\draw ($1.25*(l3_3)$) circle[radius=4pt];
			\filldraw[color=red] ($1.25*(l4_3)$) circle[radius=2pt];
			\filldraw[color=blue] ($1.25*(l5_3)$) circle[radius=2pt];
			\filldraw[color=red] ($1.25*(l6_3)$) circle[radius=2pt];
			\draw[color=red] ($1.25*(l2_3)$) -- ($1.25*(l4_3)$);
			\draw[color=blue] ($1.25*(l1_3)$) -- ($1.25*(l5_3)$);
			
			\draw[->] ($1.1*(l3)$) -- ($0.8*(l3)$);
			
			
			\node (l1_4) at ($1.5*(l4)+0.5*(1.7320508075688774,-1.)$) {};
			\node (l2_4) at ($1.5*(l4)+0.5*(1.7320508075688774,1.)$) {};
			\node (l3_4) at ($1.5*(l4)+0.5*(0,2.)$) {};
			\node (l4_4) at ($1.5*(l4)+0.5*(-1.7320508075688774,1.)$){};
			\node (l5_4) at ($1.5*(l4)+0.5*(-1.7320508075688774,-1.)$) {};
			\node (l6_4) at ($1.5*(l4)+0.5*(0,-2.)$) {};

			\draw[->,shorten >=3pt, shorten <=3pt, color=black!50!green] ($1.25*(l4_4)$) -- ($1.25*(l5_4)$);
			\draw[->,shorten >=3pt, shorten <=3pt] ($1.25*(l5_4)$) -- ($1.25*(l6_4)$);
			\draw[->,shorten >=3pt, shorten <=3pt, color=black!50!green] ($1.25*(l6_4)$) -- ($1.25*(l1_4)$);
			\draw[->,shorten >=3pt, shorten <=3pt, color=black!50!green] ($1.25*(l4_4)$) -- ($1.25*(l3_4)$);
			\draw[->,shorten >=3pt, shorten <=3pt] ($1.25*(l3_4)$) -- ($1.25*(l2_4)$);
			\draw[->,shorten >=3pt, shorten <=3pt, color=black!50!green] ($1.25*(l2_4)$) -- ($1.25*(l1_4)$);
			\fill[color=black!50!green, opacity=0.2] ($1.25*(l4_4)$) -- ($1.25*(l5_4)$) -- ($1.25*(l3_4)$);
			\fill[color=black!50!green, opacity=0.2] ($1.25*(l6_4)$) -- ($1.25*(l1_4)$) -- ($1.25*(l2_4)$);
			\fill[color=black, opacity=0.2] ($1.25*(l5_4)$) -- ($1.25*(l6_4)$) -- ($1.25*(l2_4)$) -- ($1.25*(l3_4)$);
			\filldraw[color=blue] ($1.25*(l1_4)$) circle[radius=2pt];
			\filldraw[color=red] ($1.25*(l2_4)$) circle[radius=2pt];
			\filldraw[color=blue] ($1.25*(l3_4)$) circle[radius=2pt];
			\filldraw[color=red] ($1.25*(l4_4)$) circle[radius=2pt];
			\draw ($1.25*(l4_4)$) circle[radius=4pt];
			\filldraw[color=blue] ($1.25*(l5_4)$) circle[radius=2pt];
			\filldraw[color=red] ($1.25*(l6_4)$) circle[radius=2pt];
			\draw[color=blue] ($1.25*(l3_4)$) -- ($1.25*(l5_4)$);
			\draw[color=red] ($1.25*(l2_4)$) -- ($1.25*(l6_4)$);
			
			\draw[->] ($1.1*(l4)$) -- ($0.8*(l4)$);
			
			
			\node (l1_5) at ($1.5*(l5)+0.5*(1.7320508075688774,-1.)$) {};
			\node (l2_5) at ($1.5*(l5)+0.5*(1.7320508075688774,1.)$) {};
			\node (l3_5) at ($1.5*(l5)+0.5*(0,2.)$) {};
			\node (l4_5) at ($1.5*(l5)+0.5*(-1.7320508075688774,1.)$){};
			\node (l5_5) at ($1.5*(l5)+0.5*(-1.7320508075688774,-1.)$) {};
			\node (l6_5) at ($1.5*(l5)+0.5*(0,-2.)$) {};
			
			\draw[->,shorten >=3pt, shorten <=3pt] ($1.25*(l5_5)$) -- ($1.25*(l6_5)$);
			\draw[->,shorten >=3pt, shorten <=3pt, color=black!50!green] ($1.25*(l6_5)$) -- ($1.25*(l1_5)$);
			\draw[->,shorten >=3pt, shorten <=3pt] ($1.25*(l1_5)$) -- ($1.25*(l2_5)$);
			\draw[->,shorten >=3pt, shorten <=3pt] ($1.25*(l5_5)$) -- ($1.25*(l4_5)$);
			\draw[->,shorten >=3pt, shorten <=3pt, color=black!50!green] ($1.25*(l4_5)$) -- ($1.25*(l3_5)$);
			\draw[->,shorten >=3pt, shorten <=3pt] ($1.25*(l3_5)$) -- ($1.25*(l2_5)$);
			\fill[color=black, opacity=0.2] ($1.25*(l5_5)$) -- ($1.25*(l6_5)$) -- ($1.25*(l4_5)$);
			\fill[color=black, opacity=0.2] ($1.25*(l1_5)$) -- ($1.25*(l2_5)$) -- ($1.25*(l3_5)$);
			\fill[color=black!50!green, opacity=0.2] ($1.25*(l6_5)$) -- ($1.25*(l1_5)$) -- ($1.25*(l3_5)$) -- ($1.25*(l4_5)$);
			\filldraw[color=blue] ($1.25*(l1_5)$) circle[radius=2pt];
			\filldraw[color=red] ($1.25*(l2_5)$) circle[radius=2pt];
			\filldraw[color=blue] ($1.25*(l3_5)$) circle[radius=2pt];
			\filldraw[color=red] ($1.25*(l4_5)$) circle[radius=2pt];
			\filldraw[color=blue] ($1.25*(l5_5)$) circle[radius=2pt];
			\draw ($1.25*(l5_5)$) circle[radius=4pt];
			\filldraw[color=red] ($1.25*(l6_5)$) circle[radius=2pt];
			\draw[color=red] ($1.25*(l4_5)$) -- ($1.25*(l6_5)$);
			\draw[color=blue] ($1.25*(l3_5)$) -- ($1.25*(l1_5)$);
			
			\draw[->] ($1.1*(l5)$) -- ($0.8*(l5)$);
			
			
			\node (l1_6) at ($1.5*(l6)+0.5*(1.7320508075688774,-1.)$) {};
			\node (l2_6) at ($1.5*(l6)+0.5*(1.7320508075688774,1.)$) {};
			\node (l3_6) at ($1.5*(l6)+0.5*(0,2.)$) {};
			\node (l4_6) at ($1.5*(l6)+0.5*(-1.7320508075688774,1.)$){};
			\node (l5_6) at ($1.5*(l6)+0.5*(-1.7320508075688774,-1.)$) {};
			\node (l6_6) at ($1.5*(l6)+0.5*(0,-2.)$) {};
			
			\draw[->,shorten >=3pt, shorten <=3pt, color=black!50!green] ($1.25*(l6_6)$) -- ($1.25*(l1_6)$);
			\draw[->,shorten >=3pt, shorten <=3pt] ($1.25*(l1_6)$) -- ($1.25*(l2_6)$);
			\draw[->,shorten >=3pt, shorten <=3pt, color=black!50!green] ($1.25*(l2_6)$) -- ($1.25*(l3_6)$);
			\draw[->,shorten >=3pt, shorten <=3pt, color=black!50!green] ($1.25*(l6_6)$) -- ($1.25*(l5_6)$);
			\draw[->,shorten >=3pt, shorten <=3pt] ($1.25*(l5_6)$) -- ($1.25*(l4_6)$);
			\draw[->,shorten >=3pt, shorten <=3pt, color=black!50!green] ($1.25*(l4_6)$) -- ($1.25*(l3_6)$);
			\fill[color=black!50!green, opacity=0.2] ($1.25*(l6_6)$) -- ($1.25*(l1_6)$) -- ($1.25*(l5_6)$);
			\fill[color=black!50!green, opacity=0.2] ($1.25*(l2_6)$) -- ($1.25*(l3_6)$) -- ($1.25*(l4_6)$);
			\fill[color=black, opacity=0.2] ($1.25*(l1_6)$) -- ($1.25*(l2_6)$) -- ($1.25*(l4_6)$) -- ($1.25*(l5_6)$);
			\filldraw[color=blue] ($1.25*(l1_6)$) circle[radius=2pt];
			\filldraw[color=red] ($1.25*(l2_6)$) circle[radius=2pt];
			\filldraw[color=blue] ($1.25*(l3_6)$) circle[radius=2pt];
			\filldraw[color=red] ($1.25*(l4_6)$) circle[radius=2pt];
			\filldraw[color=blue] ($1.25*(l5_6)$) circle[radius=2pt];
			\filldraw[color=red] ($1.25*(l6_6)$) circle[radius=2pt];
			\draw ($1.25*(l6_6)$) circle[radius=4pt];
			\draw[color=blue] ($1.25*(l5_6)$) -- ($1.25*(l1_6)$);
			\draw[color=red] ($1.25*(l4_6)$) -- ($1.25*(l2_6)$);
			
			\draw[->] ($1.1*(l6)$) -- ($0.8*(l6)$);
			
			\node at (0,-6) {\small$\rksdS$};
			
			\node (t) at (8,0) {};;
			\node[color=blue] (m) at ($(-1,0)+(t)$) {\small$(-,-)$};
			\node[color=red] (p) at ($(1,0)+(t)$) {\small{$(+,+)$}};
			\node[color=black!50!green] at ($(t)+(0,1.25)$) {\small{$(0,+)$}};
			\node[color=black] at ($(t)-(0,1.25)$) {\small{$(0,-)$}};
			
			\draw[->,shorten >=3pt, shorten <=3pt, color=black!50!green] ($(p)+(0,0.25)$) .. controls ($0.75*(p)+(0,0.5)+0.25*(m)+(0,0.5)$) and ($0.25*(p)+(0,0.5)+0.75*(m)+(0,0.5)$)  .. ($(m)+(0,0.25)$);;
			\draw[->,shorten >=3pt, shorten <=3pt] ($(m)-(0,0.25)$) .. controls ($0.75*(m)-(0,0.5)+0.25*(p)-(0,0.5)$) and ($0.25*(m)-(0,0.5)+0.75*(p)-(0,0.5)$)  .. ($(p)-(0,0.25)$);;
			
			\node at ($(t)+(0,-1.75)$) {\small$\Cc$};
			
			\draw[->] (5,0) -- (6,0);
			\node at (5.5,0.5) {\small$\wt{Q}$};
		\end{tikzpicture}
		\caption{The combinatorial Minor fibration of the arrangement $\Ac$ in $\RR^2$ with defining polynomial $xy(x-y)$
			on the tope-rank subdivision of the Salvetti complex.}\label{fig:CombMFib}
	\end{figure}
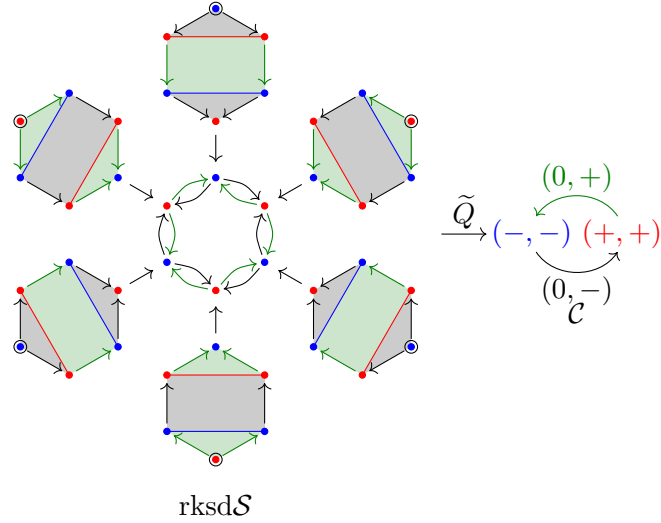
	\begin{example}
		\label{ex:CombMFib}
		Let $\Ac$ be the arrangement in $\RR^2$ with defining polynomial $Q = Q(\Ac) = xy(x-y)$.
		Then the combinatorial Milnor fibration $\wt{Q}$ of $\OM(\Ac)$ is displayed in Figure \ref{fig:CombMFib} where the fibers are colored accordingly.
	\end{example}
	
	\begin{figure}
		\begin{tikzpicture}[scale=1]
			\node (A) at (-4,0) {(\textbf{A})};;
			\node (B) at (0,0) {(\textbf{B})};;
			\node (C) at (4,0) {(\textbf{C})};;
			
			\node at ($(A)+(0,-0.75)$) {\small($0< k <|z(\tau)|-1$)};;
			\node at ($(C)+(0,-0.75)$) {\small where $r=|z(\tau)|$};;

			\node (A1) at ($(A)+(-1,1.0)$) {\small$(\tau^R_k,R)$};
			\node (A2) at ($(A)+(0,3)$) {\small$(\tau^R_{[k,k+1]},R)$};
			\node (A3) at ($(A)+(1,1)$) {\small$(\tau^R_{k+1},R)$};
			
			\node (B1) at ($(B)+(0,1)$) {\small$(\tau^R_1,R)$};
			\node (B2) at ($(B)+(0,3)$) {\small$(\tau^R_{[0,1]},R)$};
			
			\node (C1) at ($(C)+(0,1)$) {\small$(\tau^R_{r-1},R)$};
			\node (C2) at ($(C)+(0,3)$) {\small$(\tau^R_{[r-1,r]},R)$};
			
			\draw (A1) -- (A2) -- (A3);
			\draw (B1) -- (B2); 
			\draw (C1) -- (C2); 
			
		\end{tikzpicture}
		\caption{The three possible cases for a connected component of the fiber 
			$\wt{p}|_{(\wt{Q}\downarrow b)}^{-1}((\tau,R))$, $\dim(\tau)\geq 2$ in the proof of Theorem \ref{thm:MilnorOMPosetQuasiFib}.}
		\label{fig:FibersProofPQF}
	\end{figure}
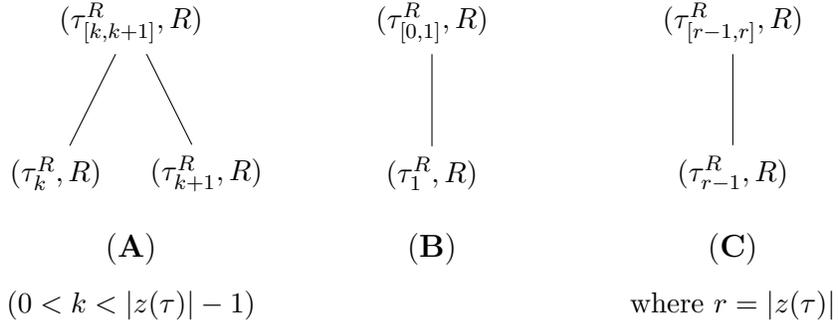
	
	\begin{theorem}
		\label{thm:MilnorOMPosetQuasiFib}
		The map $\wt{Q}:\rksdS \to \Cc$ is a poset quasi-fibration.
	\end{theorem}
	\begin{proof}
		Let $b \in \{(0,+),(0,-)\}$ and $a \in \{(+,+),(-,-)\}$. 
		Without loss of generality, it suffices to show that $(\wt{Q}\downarrow a) \hookrightarrow (\wt{Q}\downarrow b)$ is a homotopy equivalence
		for $a = (+,+)$ and $b = (0,+)$.
		
		The proof is in two steps. 
		In the first step, we construct an acyclic matching $\mtc$ on $(\wt{Q}\downarrow b)$ with Theorem \ref{thm:Patchwork}
		such that its critical cells $C(\mtc)$ constitute a subcomplex of $(\wt{Q}\downarrow b)$ which still properly contains
		$(\wt{Q}\downarrow a)$, but with the property that all cells in $C(\mtc) \cap (\wt{Q}\downarrow (-,-))$ are $0$-dimensional.
		In the second step, we then notice that each such vertex in $C(\mtc) \cap (\wt{Q}\downarrow (-,-))$ is the apex of
		a cone over a contractible subcomplex of $(\wt{Q}\downarrow a)$.
		These two steps together yield two homotopy equivalences $(\wt{Q}\downarrow a) \hookrightarrow C(\mtc) \hookrightarrow (\wt{Q}\downarrow b)$
		which concludes our argument (cf.\ Example \ref{ex:combMFibPQF} and Figure \ref{fig:combMFibPQF}).
		
		We begin with the construction of the matching $\mtc$.
		We construct matchings on each of
		the fibers of the poset map $\wt{p}|_{(\wt{Q}\downarrow b)}:(\wt{Q}\downarrow b) \to \wt{p}((\wt{Q}\downarrow b))$
		and apply Theorem \ref{thm:Patchwork} to obtain a matching $\mtc$ on the whole complex $(\wt{Q}\downarrow b)$.
		So, let $(\tau,R) \in \wt{p}((\wt{Q}\downarrow b)) \subseteq \Sc$ be a Salvetti cell of dimension at least $2$
		which is in the image of our complex $(\wt{Q}\downarrow b)$ under the projection map $\wt{p}:\rksdS \to \Sc$
		(note that all fibers of $\wt{p}$ of cells of dimension at most one are singletons). 
		Then a connected component of $\wt{p}|_{(\wt{Q}\downarrow b)}^{-1}((\tau,R)) \subseteq (\wt{Q}\downarrow b)$ has one of the following
		three forms (\textbf{A}), (\textbf{B}), or (\textbf{C}) depicted in Figure \ref{fig:FibersProofPQF}, 
		by Definitions \ref{def:TopeRankSubdivCovectors}, \ref{def:TopeRankSubdivSalvetti} and \ref{def:MilnorFibrationOM}.
		In case (\textbf{A}), we either have $\wt{Q}((\tau^R_{k},R))=(+,+), \wt{Q}((\tau^R_{k+1},R))=(-,-)$ or the other way around;
		in both situations we can match those two cells in this connected component which are not contained in $(\wt{Q}\downarrow a)$.
		By the Definition \ref{def:MilnorFibrationOM} of the map $\wt{Q}$, in case (\textbf{B}), we must have $Q(R)=+$, so
		$\wt{Q}((\tau^R_{1},R)) = (-,-)$ and both cells are in $(\wt{Q}\downarrow b) \setminus (\wt{Q}\downarrow a)$ and can be matched for $\mtc$.
		In the third case (\textbf{C}), we then have $\wt{Q}((\tau^R_{r-1},R)) = (+,+)$, so in this situation we can not match anything,
		and are left with a cell $(\tau^R_{[r-1,r]},R) \in (\wt{Q}\downarrow b) \setminus C(\mtc)$ which has an adjacent vertex 
		$(\tau\circ(-R),\tau\circ(-R)) \in (\wt{Q}\downarrow b) \setminus C(\mtc)$.
		
		After, for each such fiber, all cells are matched in the above way, by Theorem \ref{thm:Patchwork} we obtain our desired acyclic matching $\mtc$
		whose critical cells $C(\mtc)$ constitute a subcomplex of $(\wt{Q}\downarrow b)$,
		and by Theorem \ref{thm:MainThmDMTRegularSubcomplex} we get a first homotopy equivalence $C(\mtc) \hookrightarrow (\wt{Q}\downarrow b)$.
		By construction, our matching has the following critical cells:
		\begin{align*}
			C(\mtc) = & \quad(\wt{Q}\downarrow a)\\
			&\cup  \{(\tau^R_{[|z(\tau)|-1,|z(\tau)|]},R) \mid Q(\tau\circ(-R))=-, \tau \in \Lc^\dual_{\dim \geq 1}\}\\
			&\cup  \{(T,T) \mid Q(T)=-\}.
		\end{align*}
		But it is clear that each $(\tau^R_{[|z(\tau)|-1,|z(\tau)|]},R)$ in the second set is
		contained in a unique maximal cell $(\Zero^T_{[n-1,n]},T)$ with $Q(-T) = -$,
		which is moreover a cone with apex $(-T,-T)$ over the contractible subcomplex of $(\wt{Q}\downarrow a)$
		given by the closure of $(\Zero^T_{n-1},T)$.
		Finally, ``pushing in'' all these cones yields 
		the second homotopy equivalence $$(\wt{Q}\downarrow a) \hookrightarrow C(\mtc)$$ and finishes the proof.
	\end{proof}

	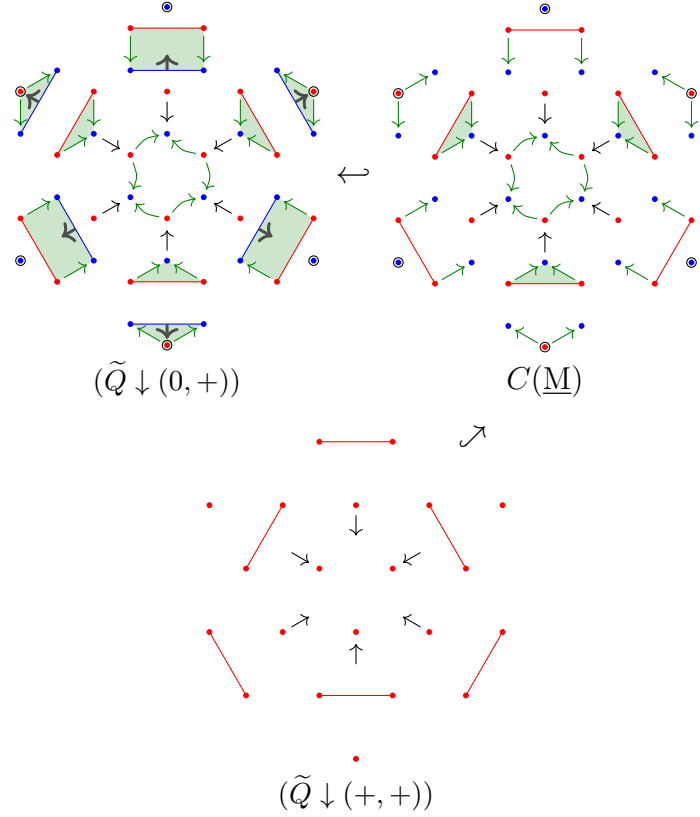
\begin{figure}
		\def\sc{0.45}
		\begin{tikzpicture}[scale=\sc]		
			\node (l1) at (1.7320508075688774,-1.) {};
			\node (l2) at (1.7320508075688774,1.) {};
			\node (l3) at (0,2.) {};
			\node (l4) at (-1.7320508075688774,1.) {};
			\node (l5) at (-1.7320508075688774,-1.) {};
			\node (l6) at (0,-2.) {};
			\node (zero) at (0,0) {};
			
			
			\node (l1_0) at ($0.5*(1.7320508075688774,-1.)$) {};
			\node (l2_0) at ($0.5*(1.7320508075688774,1.)$) {};
			\node (l3_0) at ($0.5*(0,2.)$) {};
			\node (l4_0) at ($0.5*(-1.7320508075688774,1.)$){};
			\node (l5_0) at ($0.5*(-1.7320508075688774,-1.)$) {};
			\node (l6_0) at ($0.5*(0,-2.)$) {};
			
			\draw[->,shorten >=3pt, shorten <=3pt, color=black!50!green] ($1.25*(l2_0)$) .. controls ($0.75*(l1_0)+0.75*(l2_0)$)  .. ($1.25*(l1_0)$);;
			\draw[->,shorten >=3pt, shorten <=3pt, color=black!50!green] ($1.25*(l2_0)$) .. controls ($0.5*(l2_0)+0.5*(l3_0)$)  .. ($1.25*(l3_0)$);;
			\draw[->,shorten >=3pt, shorten <=3pt, color=black!50!green] ($1.25*(l4_0)$) .. controls ($0.75*(l3_0)+0.75*(l4_0)$)  .. ($1.25*(l3_0)$);;
			\draw[->,shorten >=3pt, shorten <=3pt, color=black!50!green] ($1.25*(l4_0)$) .. controls ($0.5*(l4_0)+0.5*(l5_0)$)  .. ($1.25*(l5_0)$);;
			\draw[->,shorten >=3pt, shorten <=3pt, color=black!50!green] ($1.25*(l6_0)$) .. controls ($0.75*(l5_0)+0.75*(l6_0)$)  .. ($1.25*(l5_0)$);;
			\draw[->,shorten >=3pt, shorten <=3pt, color=black!50!green] ($1.25*(l6_0)$) .. controls ($0.5*(l6_0)+0.5*(l1_0)$)  .. ($1.25*(l1_0)$);;
			\filldraw[color=blue] ($1.25*(l1_0)$) circle[radius=2pt];
			\filldraw[color=red] ($1.25*(l2_0)$) circle[radius=2pt];
			\filldraw[color=blue] ($1.25*(l3_0)$) circle[radius=2pt];
			\filldraw[color=red] ($1.25*(l4_0)$) circle[radius=2pt];
			\filldraw[color=blue] ($1.25*(l5_0)$) circle[radius=2pt];
			\filldraw[color=red] ($1.25*(l6_0)$) circle[radius=2pt];
			
			
			\node (l1_1) at ($1.5*(l1)+0.5*(1.7320508075688774,-1.)$) {};
			\node (l2_1) at ($1.5*(l1)+0.5*(1.7320508075688774,1.)$) {};
			\node (l3_1) at ($1.5*(l1)+0.5*(0,2.)$) {};
			\node (l4_1) at ($1.5*(l1)+0.5*(-1.7320508075688774,1.)$){};
			\node (l5_1) at ($1.5*(l1)+0.5*(-1.7320508075688774,-1.)$) {};
			\node (l6_1) at ($1.5*(l1)+0.5*(0,-2.)$) {};
			
			\draw[->,shorten >=3pt, shorten <=3pt, color=black!50!green] ($1.25*(l2_1)$) -- ($1.25*(l3_1)$);
			\draw[->,shorten >=3pt, shorten <=3pt, color=black!50!green] ($1.25*(l6_1)$) -- ($1.25*(l5_1)$);
			\fill[color=black!50!green, opacity=0.2] ($1.25*(l2_1)$) -- ($1.25*(l3_1)$) -- ($1.25*(l5_1)$) -- ($1.25*(l6_1)$);
			\filldraw[color=blue] ($1.25*(l1_1)$) circle[radius=2pt];
			\draw ($1.25*(l1_1)$) circle[radius=4pt];
			\filldraw[color=red] ($1.25*(l2_1)$) circle[radius=2pt];
			\filldraw[color=blue] ($1.25*(l3_1)$) circle[radius=2pt];
			\filldraw[color=red] ($1.25*(l4_1)$) circle[radius=2pt];
			\filldraw[color=blue] ($1.25*(l5_1)$) circle[radius=2pt];
			\filldraw[color=red] ($1.25*(l6_1)$) circle[radius=2pt];
			\draw[color=red] ($1.25*(l6_1)$) -- ($1.25*(l2_1)$);
			\draw[color=blue] ($1.25*(l5_1)$) -- ($1.25*(l3_1)$);
			\node (a) at ($0.5*1.25*(l5_1)+0.5*1.25*(l3_1)$) {};;
			\draw[color=black!70, line width=0.4mm, ->] ($(a)$) -- ($1.15*(a)$);		
			\draw[->] ($1.1*(l1)$) -- ($0.8*(l1)$);
			
			
			\node (l1_2) at ($1.5*(l2)+0.5*(1.7320508075688774,-1.)$) {};
			\node (l2_2) at ($1.5*(l2)+0.5*(1.7320508075688774,1.)$) {};
			\node (l3_2) at ($1.5*(l2)+0.5*(0,2.)$) {};
			\node (l4_2) at ($1.5*(l2)+0.5*(-1.7320508075688774,1.)$){};
			\node (l5_2) at ($1.5*(l2)+0.5*(-1.7320508075688774,-1.)$) {};
			\node (l6_2) at ($1.5*(l2)+0.5*(0,-2.)$) {};
			
			\draw[->,shorten >=3pt, shorten <=3pt, color=black!50!green] ($1.25*(l2_2)$) -- ($1.25*(l3_2)$);
			\draw[->,shorten >=3pt, shorten <=3pt, color=black!50!green] ($1.25*(l4_2)$) -- ($1.25*(l5_2)$);
			\draw[->,shorten >=3pt, shorten <=3pt, color=black!50!green] ($1.25*(l2_2)$) -- ($1.25*(l1_2)$);
			\draw[->,shorten >=3pt, shorten <=3pt, color=black!50!green] ($1.25*(l6_2)$) -- ($1.25*(l5_2)$);
			\fill[color=black!50!green, opacity=0.2] ($1.25*(l2_2)$) -- ($1.25*(l3_2)$) -- ($1.25*(l1_2)$);
			\fill[color=black!50!green, opacity=0.2] ($1.25*(l4_2)$) -- ($1.25*(l5_2)$) -- ($1.25*(l6_2)$);
			\filldraw[color=blue] ($1.25*(l1_2)$) circle[radius=2pt];
			\filldraw[color=red] ($1.25*(l2_2)$) circle[radius=2pt];
			\draw ($1.25*(l2_2)$) circle[radius=4pt];
			\filldraw[color=blue] ($1.25*(l3_2)$) circle[radius=2pt];
			\filldraw[color=red] ($1.25*(l4_2)$) circle[radius=2pt];
			\filldraw[color=blue] ($1.25*(l5_2)$) circle[radius=2pt];
			\filldraw[color=red] ($1.25*(l6_2)$) circle[radius=2pt];
			\draw[color=blue] ($1.25*(l1_2)$) -- ($1.25*(l3_2)$);
			\node (a) at ($0.5*1.25*(l1_2)+0.5*1.25*(l3_2)$) {};;
			\draw[color=black!70, line width=0.4mm, ->] ($(a)$) -- ($1.1*(a)$);	
			\draw[color=red] ($1.25*(l6_2)$) -- ($1.25*(l4_2)$);
			
			\draw[->] ($1.1*(l2)$) -- ($0.8*(l2)$);
			
			
			\node (l1_3) at ($1.5*(l3)+0.5*(1.7320508075688774,-1.)$) {};
			\node (l2_3) at ($1.5*(l3)+0.5*(1.7320508075688774,1.)$) {};
			\node (l3_3) at ($1.5*(l3)+0.5*(0,2.)$) {};
			\node (l4_3) at ($1.5*(l3)+0.5*(-1.7320508075688774,1.)$){};
			\node (l5_3) at ($1.5*(l3)+0.5*(-1.7320508075688774,-1.)$) {};
			\node (l6_3) at ($1.5*(l3)+0.5*(0,-2.)$) {};
			
			\draw[->,shorten >=3pt, shorten <=3pt, color=black!50!green] ($1.25*(l4_3)$) -- ($1.25*(l5_3)$);
			\draw[->,shorten >=3pt, shorten <=3pt, color=black!50!green] ($1.25*(l2_3)$) -- ($1.25*(l1_3)$);
			\fill[color=black!50!green, opacity=0.2] ($1.25*(l4_3)$) -- ($1.25*(l5_3)$) -- ($1.25*(l1_3)$) -- ($1.25*(l2_3)$);
			\filldraw[color=blue] ($1.25*(l1_3)$) circle[radius=2pt];
			\filldraw[color=red] ($1.25*(l2_3)$) circle[radius=2pt];
			\filldraw[color=blue] ($1.25*(l3_3)$) circle[radius=2pt];
			\draw ($1.25*(l3_3)$) circle[radius=4pt];
			\filldraw[color=red] ($1.25*(l4_3)$) circle[radius=2pt];
			\filldraw[color=blue] ($1.25*(l5_3)$) circle[radius=2pt];
			\filldraw[color=red] ($1.25*(l6_3)$) circle[radius=2pt];
			\draw[color=red] ($1.25*(l2_3)$) -- ($1.25*(l4_3)$);
			\draw[color=blue] ($1.25*(l1_3)$) -- ($1.25*(l5_3)$);
			\node (a) at ($0.5*1.25*(l1_3)+0.5*1.25*(l5_3)$) {};;
			\draw[color=black!70, line width=0.4mm, ->] ($(a)$) -- ($1.15*(a)$);	
			
			\draw[->] ($1.1*(l3)$) -- ($0.8*(l3)$);
			
			
			\node (l1_4) at ($1.5*(l4)+0.5*(1.7320508075688774,-1.)$) {};
			\node (l2_4) at ($1.5*(l4)+0.5*(1.7320508075688774,1.)$) {};
			\node (l3_4) at ($1.5*(l4)+0.5*(0,2.)$) {};
			\node (l4_4) at ($1.5*(l4)+0.5*(-1.7320508075688774,1.)$){};
			\node (l5_4) at ($1.5*(l4)+0.5*(-1.7320508075688774,-1.)$) {};
			\node (l6_4) at ($1.5*(l4)+0.5*(0,-2.)$) {};

			\draw[->,shorten >=3pt, shorten <=3pt, color=black!50!green] ($1.25*(l4_4)$) -- ($1.25*(l5_4)$);
			\draw[->,shorten >=3pt, shorten <=3pt, color=black!50!green] ($1.25*(l6_4)$) -- ($1.25*(l1_4)$);
			\draw[->,shorten >=3pt, shorten <=3pt, color=black!50!green] ($1.25*(l4_4)$) -- ($1.25*(l3_4)$);
			\draw[->,shorten >=3pt, shorten <=3pt, color=black!50!green] ($1.25*(l2_4)$) -- ($1.25*(l1_4)$);
			\fill[color=black!50!green, opacity=0.2] ($1.25*(l4_4)$) -- ($1.25*(l5_4)$) -- ($1.25*(l3_4)$);
			\fill[color=black!50!green, opacity=0.2] ($1.25*(l6_4)$) -- ($1.25*(l1_4)$) -- ($1.25*(l2_4)$);
			\filldraw[color=blue] ($1.25*(l1_4)$) circle[radius=2pt];
			\filldraw[color=red] ($1.25*(l2_4)$) circle[radius=2pt];
			\filldraw[color=blue] ($1.25*(l3_4)$) circle[radius=2pt];
			\filldraw[color=red] ($1.25*(l4_4)$) circle[radius=2pt];
			\draw ($1.25*(l4_4)$) circle[radius=4pt];
			\filldraw[color=blue] ($1.25*(l5_4)$) circle[radius=2pt];
			\filldraw[color=red] ($1.25*(l6_4)$) circle[radius=2pt];
			\draw[color=blue] ($1.25*(l3_4)$) -- ($1.25*(l5_4)$);
			\node (a) at ($0.5*1.25*(l3_4)+0.5*1.25*(l5_4)$) {};;
			\draw[color=black!70, line width=0.4mm, ->] ($(a)$) -- ($1.1*(a)$);	
			\draw[color=red] ($1.25*(l2_4)$) -- ($1.25*(l6_4)$);
			
			\draw[->] ($1.1*(l4)$) -- ($0.8*(l4)$);
			
			
			\node (l1_5) at ($1.5*(l5)+0.5*(1.7320508075688774,-1.)$) {};
			\node (l2_5) at ($1.5*(l5)+0.5*(1.7320508075688774,1.)$) {};
			\node (l3_5) at ($1.5*(l5)+0.5*(0,2.)$) {};
			\node (l4_5) at ($1.5*(l5)+0.5*(-1.7320508075688774,1.)$){};
			\node (l5_5) at ($1.5*(l5)+0.5*(-1.7320508075688774,-1.)$) {};
			\node (l6_5) at ($1.5*(l5)+0.5*(0,-2.)$) {};
			
			\draw[->,shorten >=3pt, shorten <=3pt, color=black!50!green] ($1.25*(l6_5)$) -- ($1.25*(l1_5)$);
			\draw[->,shorten >=3pt, shorten <=3pt, color=black!50!green] ($1.25*(l4_5)$) -- ($1.25*(l3_5)$);
			\fill[color=black!50!green, opacity=0.2] ($1.25*(l6_5)$) -- ($1.25*(l1_5)$) -- ($1.25*(l3_5)$) -- ($1.25*(l4_5)$);
			\filldraw[color=blue] ($1.25*(l1_5)$) circle[radius=2pt];
			\filldraw[color=red] ($1.25*(l2_5)$) circle[radius=2pt];
			\filldraw[color=blue] ($1.25*(l3_5)$) circle[radius=2pt];
			\filldraw[color=red] ($1.25*(l4_5)$) circle[radius=2pt];
			\filldraw[color=blue] ($1.25*(l5_5)$) circle[radius=2pt];
			\draw ($1.25*(l5_5)$) circle[radius=4pt];
			\filldraw[color=red] ($1.25*(l6_5)$) circle[radius=2pt];
			\draw[color=red] ($1.25*(l4_5)$) -- ($1.25*(l6_5)$);
			\draw[color=blue] ($1.25*(l3_5)$) -- ($1.25*(l1_5)$);
			\node (a) at ($0.5*1.25*(l3_5)+0.5*1.25*(l1_5)$) {};;
			\draw[color=black!70, line width=0.4mm, ->] ($(a)$) -- ($1.15*(a)$);	
			
			\draw[->] ($1.1*(l5)$) -- ($0.8*(l5)$);
			
			
			\node (l1_6) at ($1.5*(l6)+0.5*(1.7320508075688774,-1.)$) {};
			\node (l2_6) at ($1.5*(l6)+0.5*(1.7320508075688774,1.)$) {};
			\node (l3_6) at ($1.5*(l6)+0.5*(0,2.)$) {};
			\node (l4_6) at ($1.5*(l6)+0.5*(-1.7320508075688774,1.)$){};
			\node (l5_6) at ($1.5*(l6)+0.5*(-1.7320508075688774,-1.)$) {};
			\node (l6_6) at ($1.5*(l6)+0.5*(0,-2.)$) {};
			
			\draw[->,shorten >=3pt, shorten <=3pt, color=black!50!green] ($1.25*(l6_6)$) -- ($1.25*(l1_6)$);
			\draw[->,shorten >=3pt, shorten <=3pt, color=black!50!green] ($1.25*(l2_6)$) -- ($1.25*(l3_6)$);
			\draw[->,shorten >=3pt, shorten <=3pt, color=black!50!green] ($1.25*(l6_6)$) -- ($1.25*(l5_6)$);
			\draw[->,shorten >=3pt, shorten <=3pt, color=black!50!green] ($1.25*(l4_6)$) -- ($1.25*(l3_6)$);
			\fill[color=black!50!green, opacity=0.2] ($1.25*(l6_6)$) -- ($1.25*(l1_6)$) -- ($1.25*(l5_6)$);
			\fill[color=black!50!green, opacity=0.2] ($1.25*(l2_6)$) -- ($1.25*(l3_6)$) -- ($1.25*(l4_6)$);
			\filldraw[color=blue] ($1.25*(l1_6)$) circle[radius=2pt];
			\filldraw[color=red] ($1.25*(l2_6)$) circle[radius=2pt];
			\filldraw[color=blue] ($1.25*(l3_6)$) circle[radius=2pt];
			\filldraw[color=red] ($1.25*(l4_6)$) circle[radius=2pt];
			\filldraw[color=blue] ($1.25*(l5_6)$) circle[radius=2pt];
			\filldraw[color=red] ($1.25*(l6_6)$) circle[radius=2pt];
			\draw ($1.25*(l6_6)$) circle[radius=4pt];
			\draw[color=blue] ($1.25*(l5_6)$) -- ($1.25*(l1_6)$);
			\node (a) at ($0.5*1.25*(l5_6)+0.5*1.25*(l1_6)$) {};;
			\draw[color=black!70, line width=0.4mm, ->] ($(a)$) -- ($1.1*(a)$);	
			\draw[color=red] ($1.25*(l4_6)$) -- ($1.25*(l2_6)$);
			
			\draw[->] ($1.1*(l6)$) -- ($0.8*(l6)$); 
			\node at (5.5,0) {$\hookleftarrow$};
			
			\node at (0,-6) {\small$(\wt{Q}\downarrow (0,+))$};
		\end{tikzpicture}
		\begin{tikzpicture}[scale=\sc]		
			\node (l1) at (1.7320508075688774,-1.) {};
			\node (l2) at (1.7320508075688774,1.) {};
			\node (l3) at (0,2.) {};
			\node (l4) at (-1.7320508075688774,1.) {};
			\node (l5) at (-1.7320508075688774,-1.) {};
			\node (l6) at (0,-2.) {};
			\node (zero) at (0,0) {};
			
			
			\node (l1_0) at ($0.5*(1.7320508075688774,-1.)$) {};
			\node (l2_0) at ($0.5*(1.7320508075688774,1.)$) {};
			\node (l3_0) at ($0.5*(0,2.)$) {};
			\node (l4_0) at ($0.5*(-1.7320508075688774,1.)$){};
			\node (l5_0) at ($0.5*(-1.7320508075688774,-1.)$) {};
			\node (l6_0) at ($0.5*(0,-2.)$) {};
			
			\draw[->,shorten >=3pt, shorten <=3pt, color=black!50!green] ($1.25*(l2_0)$) .. controls ($0.75*(l1_0)+0.75*(l2_0)$)  .. ($1.25*(l1_0)$);;
			\draw[->,shorten >=3pt, shorten <=3pt, color=black!50!green] ($1.25*(l2_0)$) .. controls ($0.5*(l2_0)+0.5*(l3_0)$)  .. ($1.25*(l3_0)$);;
			\draw[->,shorten >=3pt, shorten <=3pt, color=black!50!green] ($1.25*(l4_0)$) .. controls ($0.75*(l3_0)+0.75*(l4_0)$)  .. ($1.25*(l3_0)$);;
			\draw[->,shorten >=3pt, shorten <=3pt, color=black!50!green] ($1.25*(l4_0)$) .. controls ($0.5*(l4_0)+0.5*(l5_0)$)  .. ($1.25*(l5_0)$);;
			\draw[->,shorten >=3pt, shorten <=3pt, color=black!50!green] ($1.25*(l6_0)$) .. controls ($0.75*(l5_0)+0.75*(l6_0)$)  .. ($1.25*(l5_0)$);;
			\draw[->,shorten >=3pt, shorten <=3pt, color=black!50!green] ($1.25*(l6_0)$) .. controls ($0.5*(l6_0)+0.5*(l1_0)$)  .. ($1.25*(l1_0)$);;
			\filldraw[color=blue] ($1.25*(l1_0)$) circle[radius=2pt];
			\filldraw[color=red] ($1.25*(l2_0)$) circle[radius=2pt];
			\filldraw[color=blue] ($1.25*(l3_0)$) circle[radius=2pt];
			\filldraw[color=red] ($1.25*(l4_0)$) circle[radius=2pt];
			\filldraw[color=blue] ($1.25*(l5_0)$) circle[radius=2pt];
			\filldraw[color=red] ($1.25*(l6_0)$) circle[radius=2pt];
			
			
			\node (l1_1) at ($1.5*(l1)+0.5*(1.7320508075688774,-1.)$) {};
			\node (l2_1) at ($1.5*(l1)+0.5*(1.7320508075688774,1.)$) {};
			\node (l3_1) at ($1.5*(l1)+0.5*(0,2.)$) {};
			\node (l4_1) at ($1.5*(l1)+0.5*(-1.7320508075688774,1.)$){};
			\node (l5_1) at ($1.5*(l1)+0.5*(-1.7320508075688774,-1.)$) {};
			\node (l6_1) at ($1.5*(l1)+0.5*(0,-2.)$) {};
			
			\draw[->,shorten >=3pt, shorten <=3pt, color=black!50!green] ($1.25*(l2_1)$) -- ($1.25*(l3_1)$);
			\draw[->,shorten >=3pt, shorten <=3pt, color=black!50!green] ($1.25*(l6_1)$) -- ($1.25*(l5_1)$);
			\filldraw[color=blue] ($1.25*(l1_1)$) circle[radius=2pt];
			\draw ($1.25*(l1_1)$) circle[radius=4pt];
			\filldraw[color=red] ($1.25*(l2_1)$) circle[radius=2pt];
			\filldraw[color=blue] ($1.25*(l3_1)$) circle[radius=2pt];
			\filldraw[color=red] ($1.25*(l4_1)$) circle[radius=2pt];
			\filldraw[color=blue] ($1.25*(l5_1)$) circle[radius=2pt];
			\filldraw[color=red] ($1.25*(l6_1)$) circle[radius=2pt];
			\draw[color=red] ($1.25*(l6_1)$) -- ($1.25*(l2_1)$);
			\draw[->] ($1.1*(l1)$) -- ($0.8*(l1)$);
			
			
			\node (l1_2) at ($1.5*(l2)+0.5*(1.7320508075688774,-1.)$) {};
			\node (l2_2) at ($1.5*(l2)+0.5*(1.7320508075688774,1.)$) {};
			\node (l3_2) at ($1.5*(l2)+0.5*(0,2.)$) {};
			\node (l4_2) at ($1.5*(l2)+0.5*(-1.7320508075688774,1.)$){};
			\node (l5_2) at ($1.5*(l2)+0.5*(-1.7320508075688774,-1.)$) {};
			\node (l6_2) at ($1.5*(l2)+0.5*(0,-2.)$) {};
			\
			\draw[->,shorten >=3pt, shorten <=3pt, color=black!50!green] ($1.25*(l2_2)$) -- ($1.25*(l3_2)$);
			\draw[->,shorten >=3pt, shorten <=3pt, color=black!50!green] ($1.25*(l4_2)$) -- ($1.25*(l5_2)$);
			\draw[->,shorten >=3pt, shorten <=3pt, color=black!50!green] ($1.25*(l2_2)$) -- ($1.25*(l1_2)$);
			\draw[->,shorten >=3pt, shorten <=3pt, color=black!50!green] ($1.25*(l6_2)$) -- ($1.25*(l5_2)$);
			\fill[color=black!50!green, opacity=0.2] ($1.25*(l4_2)$) -- ($1.25*(l5_2)$) -- ($1.25*(l6_2)$);
			\filldraw[color=blue] ($1.25*(l1_2)$) circle[radius=2pt];
			\filldraw[color=red] ($1.25*(l2_2)$) circle[radius=2pt];
			\draw ($1.25*(l2_2)$) circle[radius=4pt];
			\filldraw[color=blue] ($1.25*(l3_2)$) circle[radius=2pt];
			\filldraw[color=red] ($1.25*(l4_2)$) circle[radius=2pt];
			\filldraw[color=blue] ($1.25*(l5_2)$) circle[radius=2pt];
			\filldraw[color=red] ($1.25*(l6_2)$) circle[radius=2pt];
			\draw[color=red] ($1.25*(l6_2)$) -- ($1.25*(l4_2)$);
			
			\draw[->] ($1.1*(l2)$) -- ($0.8*(l2)$);
			
			
			\node (l1_3) at ($1.5*(l3)+0.5*(1.7320508075688774,-1.)$) {};
			\node (l2_3) at ($1.5*(l3)+0.5*(1.7320508075688774,1.)$) {};
			\node (l3_3) at ($1.5*(l3)+0.5*(0,2.)$) {};
			\node (l4_3) at ($1.5*(l3)+0.5*(-1.7320508075688774,1.)$){};
			\node (l5_3) at ($1.5*(l3)+0.5*(-1.7320508075688774,-1.)$) {};
			\node (l6_3) at ($1.5*(l3)+0.5*(0,-2.)$) {};
			\
			\draw[->,shorten >=3pt, shorten <=3pt, color=black!50!green] ($1.25*(l4_3)$) -- ($1.25*(l5_3)$);
			\draw[->,shorten >=3pt, shorten <=3pt, color=black!50!green] ($1.25*(l2_3)$) -- ($1.25*(l1_3)$);
			\filldraw[color=blue] ($1.25*(l1_3)$) circle[radius=2pt];
			\filldraw[color=red] ($1.25*(l2_3)$) circle[radius=2pt];
			\filldraw[color=blue] ($1.25*(l3_3)$) circle[radius=2pt];
			\draw ($1.25*(l3_3)$) circle[radius=4pt];
			\filldraw[color=red] ($1.25*(l4_3)$) circle[radius=2pt];
			\filldraw[color=blue] ($1.25*(l5_3)$) circle[radius=2pt];
			\filldraw[color=red] ($1.25*(l6_3)$) circle[radius=2pt];
			\draw[color=red] ($1.25*(l2_3)$) -- ($1.25*(l4_3)$);
			
			\draw[->] ($1.1*(l3)$) -- ($0.8*(l3)$);
			
			
			\node (l1_4) at ($1.5*(l4)+0.5*(1.7320508075688774,-1.)$) {};
			\node (l2_4) at ($1.5*(l4)+0.5*(1.7320508075688774,1.)$) {};
			\node (l3_4) at ($1.5*(l4)+0.5*(0,2.)$) {};
			\node (l4_4) at ($1.5*(l4)+0.5*(-1.7320508075688774,1.)$){};
			\node (l5_4) at ($1.5*(l4)+0.5*(-1.7320508075688774,-1.)$) {};
			\node (l6_4) at ($1.5*(l4)+0.5*(0,-2.)$) {};

			\draw[->,shorten >=3pt, shorten <=3pt, color=black!50!green] ($1.25*(l4_4)$) -- ($1.25*(l5_4)$);
			\draw[->,shorten >=3pt, shorten <=3pt, color=black!50!green] ($1.25*(l6_4)$) -- ($1.25*(l1_4)$);
			\draw[->,shorten >=3pt, shorten <=3pt, color=black!50!green] ($1.25*(l4_4)$) -- ($1.25*(l3_4)$);
			\draw[->,shorten >=3pt, shorten <=3pt, color=black!50!green] ($1.25*(l2_4)$) -- ($1.25*(l1_4)$);
			\fill[color=black!50!green, opacity=0.2] ($1.25*(l6_4)$) -- ($1.25*(l1_4)$) -- ($1.25*(l2_4)$);
			\filldraw[color=blue] ($1.25*(l1_4)$) circle[radius=2pt];
			\filldraw[color=red] ($1.25*(l2_4)$) circle[radius=2pt];
			\filldraw[color=blue] ($1.25*(l3_4)$) circle[radius=2pt];
			\filldraw[color=red] ($1.25*(l4_4)$) circle[radius=2pt];
			\draw ($1.25*(l4_4)$) circle[radius=4pt];
			\filldraw[color=blue] ($1.25*(l5_4)$) circle[radius=2pt];
			\filldraw[color=red] ($1.25*(l6_4)$) circle[radius=2pt];
			\draw[color=red] ($1.25*(l2_4)$) -- ($1.25*(l6_4)$);
			
			\draw[->] ($1.1*(l4)$) -- ($0.8*(l4)$);
			
			
			\node (l1_5) at ($1.5*(l5)+0.5*(1.7320508075688774,-1.)$) {};
			\node (l2_5) at ($1.5*(l5)+0.5*(1.7320508075688774,1.)$) {};
			\node (l3_5) at ($1.5*(l5)+0.5*(0,2.)$) {};
			\node (l4_5) at ($1.5*(l5)+0.5*(-1.7320508075688774,1.)$){};
			\node (l5_5) at ($1.5*(l5)+0.5*(-1.7320508075688774,-1.)$) {};
			\node (l6_5) at ($1.5*(l5)+0.5*(0,-2.)$) {};
			
			\draw[->,shorten >=3pt, shorten <=3pt, color=black!50!green] ($1.25*(l6_5)$) -- ($1.25*(l1_5)$);
			\draw[->,shorten >=3pt, shorten <=3pt, color=black!50!green] ($1.25*(l4_5)$) -- ($1.25*(l3_5)$);
			\filldraw[color=blue] ($1.25*(l1_5)$) circle[radius=2pt];
			\filldraw[color=red] ($1.25*(l2_5)$) circle[radius=2pt];
			\filldraw[color=blue] ($1.25*(l3_5)$) circle[radius=2pt];
			\filldraw[color=red] ($1.25*(l4_5)$) circle[radius=2pt];
			\filldraw[color=blue] ($1.25*(l5_5)$) circle[radius=2pt];
			\draw ($1.25*(l5_5)$) circle[radius=4pt];
			\filldraw[color=red] ($1.25*(l6_5)$) circle[radius=2pt];
			\draw[color=red] ($1.25*(l4_5)$) -- ($1.25*(l6_5)$);
			
			\draw[->] ($1.1*(l5)$) -- ($0.8*(l5)$);
			
			
			\node (l1_6) at ($1.5*(l6)+0.5*(1.7320508075688774,-1.)$) {};
			\node (l2_6) at ($1.5*(l6)+0.5*(1.7320508075688774,1.)$) {};
			\node (l3_6) at ($1.5*(l6)+0.5*(0,2.)$) {};
			\node (l4_6) at ($1.5*(l6)+0.5*(-1.7320508075688774,1.)$){};
			\node (l5_6) at ($1.5*(l6)+0.5*(-1.7320508075688774,-1.)$) {};
			\node (l6_6) at ($1.5*(l6)+0.5*(0,-2.)$) {};
			
			\draw[->,shorten >=3pt, shorten <=3pt, color=black!50!green] ($1.25*(l6_6)$) -- ($1.25*(l1_6)$);
			\draw[->,shorten >=3pt, shorten <=3pt, color=black!50!green] ($1.25*(l2_6)$) -- ($1.25*(l3_6)$);
			\draw[->,shorten >=3pt, shorten <=3pt, color=black!50!green] ($1.25*(l6_6)$) -- ($1.25*(l5_6)$);
			\draw[->,shorten >=3pt, shorten <=3pt, color=black!50!green] ($1.25*(l4_6)$) -- ($1.25*(l3_6)$);
			\fill[color=black!50!green, opacity=0.2] ($1.25*(l2_6)$) -- ($1.25*(l3_6)$) -- ($1.25*(l4_6)$);
			\filldraw[color=blue] ($1.25*(l1_6)$) circle[radius=2pt];
			\filldraw[color=red] ($1.25*(l2_6)$) circle[radius=2pt];
			\filldraw[color=blue] ($1.25*(l3_6)$) circle[radius=2pt];
			\filldraw[color=red] ($1.25*(l4_6)$) circle[radius=2pt];
			\filldraw[color=blue] ($1.25*(l5_6)$) circle[radius=2pt];
			\filldraw[color=red] ($1.25*(l6_6)$) circle[radius=2pt];
			\draw ($1.25*(l6_6)$) circle[radius=4pt];
			\draw[color=red] ($1.25*(l4_6)$) -- ($1.25*(l2_6)$);
			
			\draw[->] ($1.1*(l6)$) -- ($0.8*(l6)$); 
			%
			
			\node at (0,-6) {$C(\mtc)$};
		\end{tikzpicture}
		
		\begin{tikzpicture}[scale=\sc]	
			
			\node at (3.5,4.5) {\rotatebox[origin=c]{45}{$\hookrightarrow$}};	
			\node (l1) at (1.7320508075688774,-1.) {};
			\node (l2) at (1.7320508075688774,1.) {};
			\node (l3) at (0,2.) {};
			\node (l4) at (-1.7320508075688774,1.) {};
			\node (l5) at (-1.7320508075688774,-1.) {};
			\node (l6) at (0,-2.) {};
			\node (zero) at (0,0) {};
			
			
			\node (l1_0) at ($0.5*(1.7320508075688774,-1.)$) {};
			\node (l2_0) at ($0.5*(1.7320508075688774,1.)$) {};
			\node (l3_0) at ($0.5*(0,2.)$) {};
			\node (l4_0) at ($0.5*(-1.7320508075688774,1.)$){};
			\node (l5_0) at ($0.5*(-1.7320508075688774,-1.)$) {};
			\node (l6_0) at ($0.5*(0,-2.)$) {};

			\filldraw[color=red] ($1.25*(l2_0)$) circle[radius=2pt];
			\filldraw[color=red] ($1.25*(l4_0)$) circle[radius=2pt];
			\filldraw[color=red] ($1.25*(l6_0)$) circle[radius=2pt];
			
			
			\node (l1_1) at ($1.5*(l1)+0.5*(1.7320508075688774,-1.)$) {};
			\node (l2_1) at ($1.5*(l1)+0.5*(1.7320508075688774,1.)$) {};
			\node (l3_1) at ($1.5*(l1)+0.5*(0,2.)$) {};
			\node (l4_1) at ($1.5*(l1)+0.5*(-1.7320508075688774,1.)$){};
			\node (l5_1) at ($1.5*(l1)+0.5*(-1.7320508075688774,-1.)$) {};
			\node (l6_1) at ($1.5*(l1)+0.5*(0,-2.)$) {};
			
			\filldraw[color=red] ($1.25*(l2_1)$) circle[radius=2pt];
			\filldraw[color=red] ($1.25*(l4_1)$) circle[radius=2pt];
			\filldraw[color=red] ($1.25*(l6_1)$) circle[radius=2pt];
			\draw[color=red] ($1.25*(l6_1)$) -- ($1.25*(l2_1)$);
			\draw[->] ($1.1*(l1)$) -- ($0.8*(l1)$);
			
			
			\node (l1_2) at ($1.5*(l2)+0.5*(1.7320508075688774,-1.)$) {};
			\node (l2_2) at ($1.5*(l2)+0.5*(1.7320508075688774,1.)$) {};
			\node (l3_2) at ($1.5*(l2)+0.5*(0,2.)$) {};
			\node (l4_2) at ($1.5*(l2)+0.5*(-1.7320508075688774,1.)$){};
			\node (l5_2) at ($1.5*(l2)+0.5*(-1.7320508075688774,-1.)$) {};
			\node (l6_2) at ($1.5*(l2)+0.5*(0,-2.)$) {};
			
			\filldraw[color=red] ($1.25*(l2_2)$) circle[radius=2pt];
			\filldraw[color=red] ($1.25*(l4_2)$) circle[radius=2pt];
			\filldraw[color=red] ($1.25*(l6_2)$) circle[radius=2pt];
			\draw[color=red] ($1.25*(l6_2)$) -- ($1.25*(l4_2)$);
			
			\draw[->] ($1.1*(l2)$) -- ($0.8*(l2)$);
			
			
			\node (l1_3) at ($1.5*(l3)+0.5*(1.7320508075688774,-1.)$) {};
			\node (l2_3) at ($1.5*(l3)+0.5*(1.7320508075688774,1.)$) {};
			\node (l3_3) at ($1.5*(l3)+0.5*(0,2.)$) {};
			\node (l4_3) at ($1.5*(l3)+0.5*(-1.7320508075688774,1.)$){};
			\node (l5_3) at ($1.5*(l3)+0.5*(-1.7320508075688774,-1.)$) {};
			\node (l6_3) at ($1.5*(l3)+0.5*(0,-2.)$) {};
			
			\filldraw[color=red] ($1.25*(l2_3)$) circle[radius=2pt];
			\filldraw[color=red] ($1.25*(l4_3)$) circle[radius=2pt];
			\filldraw[color=red] ($1.25*(l6_3)$) circle[radius=2pt];
			\draw[color=red] ($1.25*(l2_3)$) -- ($1.25*(l4_3)$);
			
			\draw[->] ($1.1*(l3)$) -- ($0.8*(l3)$);
			
			
			\node (l1_4) at ($1.5*(l4)+0.5*(1.7320508075688774,-1.)$) {};
			\node (l2_4) at ($1.5*(l4)+0.5*(1.7320508075688774,1.)$) {};
			\node (l3_4) at ($1.5*(l4)+0.5*(0,2.)$) {};
			\node (l4_4) at ($1.5*(l4)+0.5*(-1.7320508075688774,1.)$){};
			\node (l5_4) at ($1.5*(l4)+0.5*(-1.7320508075688774,-1.)$) {};
			\node (l6_4) at ($1.5*(l4)+0.5*(0,-2.)$) {};

			\filldraw[color=red] ($1.25*(l2_4)$) circle[radius=2pt];
			\filldraw[color=red] ($1.25*(l4_4)$) circle[radius=2pt];
			\filldraw[color=red] ($1.25*(l6_4)$) circle[radius=2pt];
			\draw[color=red] ($1.25*(l2_4)$) -- ($1.25*(l6_4)$);
			
			\draw[->] ($1.1*(l4)$) -- ($0.8*(l4)$);
			
			
			\node (l1_5) at ($1.5*(l5)+0.5*(1.7320508075688774,-1.)$) {};
			\node (l2_5) at ($1.5*(l5)+0.5*(1.7320508075688774,1.)$) {};
			\node (l3_5) at ($1.5*(l5)+0.5*(0,2.)$) {};
			\node (l4_5) at ($1.5*(l5)+0.5*(-1.7320508075688774,1.)$){};
			\node (l5_5) at ($1.5*(l5)+0.5*(-1.7320508075688774,-1.)$) {};
			\node (l6_5) at ($1.5*(l5)+0.5*(0,-2.)$) {};
			
			\filldraw[color=red] ($1.25*(l2_5)$) circle[radius=2pt];
			\filldraw[color=red] ($1.25*(l4_5)$) circle[radius=2pt];
			\filldraw[color=red] ($1.25*(l6_5)$) circle[radius=2pt];
			\draw[color=red] ($1.25*(l4_5)$) -- ($1.25*(l6_5)$);
			
			\draw[->] ($1.1*(l5)$) -- ($0.8*(l5)$);
			
			
			\node (l1_6) at ($1.5*(l6)+0.5*(1.7320508075688774,-1.)$) {};
			\node (l2_6) at ($1.5*(l6)+0.5*(1.7320508075688774,1.)$) {};
			\node (l3_6) at ($1.5*(l6)+0.5*(0,2.)$) {};
			\node (l4_6) at ($1.5*(l6)+0.5*(-1.7320508075688774,1.)$){};
			\node (l5_6) at ($1.5*(l6)+0.5*(-1.7320508075688774,-1.)$) {};
			\node (l6_6) at ($1.5*(l6)+0.5*(0,-2.)$) {};
			
			\filldraw[color=red] ($1.25*(l2_6)$) circle[radius=2pt];
			\filldraw[color=red] ($1.25*(l4_6)$) circle[radius=2pt];
			\filldraw[color=red] ($1.25*(l6_6)$) circle[radius=2pt];
			\draw[color=red] ($1.25*(l4_6)$) -- ($1.25*(l2_6)$);
			
			\draw[->] ($1.1*(l6)$) -- ($0.8*(l6)$);

			\node at (0,-6) {\small$(\wt{Q}\downarrow (+,+))$};
		\end{tikzpicture}
		\caption{The two homotopy equivalences from the proof of Theorem \ref{thm:MilnorOMPosetQuasiFib}.}\label{fig:combMFibPQF}
		
	\end{figure}
	
	We illustrate the steps in the previous proof with the following example.
	\begin{example}
		\label{ex:combMFibPQF}
		Let $\Ac$ be the arrangement in $\RR^2$ with defining polynomial $Q = Q(\Ac) = xy(x-y)$.
		Now the poset fiber over $(0,+) \in \Cc$ is displayed on the right hand side in Figure \ref{fig:combMFibPQF} and the acyclic matching $\mtc$ on
		its face poset constructed in the proof of Theorem \ref{thm:MilnorOMPosetQuasiFib} is illustrated by the gray arrows. This gives the homotopy equivalence
		with the complex displayed in the middle.
		Now, three cones over blue vertices in $(\wt{Q}\downarrow(0,-))$ remain and pushing them in yields the second homotopy equivalence with $(\wt{Q}\downarrow(+,+))$
		displayed on the left hand side.
	\end{example}
	
	Now, let $\Ac$ be a real hyperplane arrangement, $\OM = (\Ac,\Lc(\Ac))$ the associated simple oriented matroid,
	$\Xf(\Ac)$ the complex complement, and $Q:\Xf(\Ac) \to \CC^\times := \CC \setminus \{0\}, z \mapsto \prod_{H \in \Ac} \alpha_H(z)$ the Milnor fibration.
	
	Our following central theorem yields our combinatorial map $\wt{Q}$ as a concrete model for
	the Milnor fibration of the complexified arrangement.
	\begin{theorem}
		\label{thm:MilnorOMHomotopyEquivToMilnorGeom}
		There is a (homotopy) commutative square
		\begin{center}
			\begin{tikzcd}[column sep=16mm, scale cd=1]
				\vert\rksdS(\Ac)\vert \ar[r, "|\wt{Q}|"]\ar[d,"\isom", swap] &\vert\Cc\vert \ar[d,"\isom"]\\
				\Xf(\Ac) \ar[r,"Q"] &\CC^\times,
			\end{tikzcd}
		\end{center}
		where the vertical maps are homotopy equivalences, i.e.\ $\wt{Q}$ is a combinatorial model for the Milnor fibration of $\Ac$.
	\end{theorem}
	
	To proof the theorem, firstly assume that $|\Ac|=n$ and let $\Bc_n$ be the Boolean arrangement of rank $n$. 
	We split up the diagram into smaller parts as follows, each of which will be shown
	to commute (up to homotopy).
	\begin{center}
		\begin{tikzcd}[column sep=16mm, scale cd=1]
			\vert\rksdS(\Ac)\vert \ar[r,hook, "i'"] \ar[rr, bend left, "\vert\wt{Q}(\Ac)\vert"]\ar[d,"|\wt{p}_\Ac|"]\ar[dd,bend right = 40] 
			&\vert\rksdS(\Bc_n)\vert \ar[r,"\vert\wt{Q}(\Bc_n)\vert"] \ar[d, "|\wt{p}_{\Bc_n}|", swap] \ar[dd, bend left = 40, "\wt{\varphi}"] &\vert\Cc\vert \ar[dd,"\isom"]\\
			\vert\Sc(\Ac)\vert \ar[r, hook, "i"] \ar[d, "\isom"] &\vert\Sc(\Bc_n)\vert \ar[d, "\isom", swap] & \\
			\Xf(\Ac) \ar[r, hook]\ar[rr,bend right, "Q(\Ac)"] & (\CC^\times)^n \ar[r,"Q(\Bc_n)"] &\CC^\times.
		\end{tikzcd}
	\end{center}
	Firstly, we note that the upper left hand square commutes by the definition of the maps $\wt{p}_\Ac$ respectively $\wt{p}_{\Bc_n}$ (see Definition \ref{def:TopeRankSubdivSalvetti}),
	where $i$ and $i'$ are the realizations of corresponding poset-inclusions.
	
	Now, we consider the lower left hand square, i.e. a relative version of Theorem \ref{thm:SalvettiHoEquiv}.
	
	\begin{proposition}
		\label{prop:RelSalvetti}
		Assume that $\Ac$ is essential, i.e.\ we have $\bigcap_{H \in \Ac}H = \{0\}$.
		We have natural inclusions $\alpha: |\Sc(\Ac)| \hookrightarrow \Xf(\Ac)$ and 
		$\beta: |\Sc(\Bc_n)| \hookrightarrow (\CC^\times)^n$
		which are homotopy equivalences and make the following diagram commute:
		\begin{center}
			\begin{tikzcd}[column sep=16mm, scale cd=1]
				\vert\Sc(\Ac)\vert \ar[r, "i" hook] \ar[d, "\alpha", hook] &\vert\Sc(\Bc_n)\vert \ar[d, "\beta", hook] \\
				\Xf(\Ac) \ar[r, "j" hook] & (\CC^\times)^n,
			\end{tikzcd}
		\end{center}
		where $i$ is the realization of the natural inclusion of the posets $\Sc(\Ac) \hookrightarrow \Sc(\Bc_n)$
		and $j$ is the injective linear map 
		\begin{eqnarray*}
			& \CC^\ell\isom V &\to \CC^n \\
			& v &\mapsto (\alpha_1(v),\ldots,\alpha_n(v)),
		\end{eqnarray*}
		given by defining linear forms $\alpha_i \in V^*$ of $\Ac$, restricted to $\Xf(\Ac)$. 
	\end{proposition}
	\begin{proof}
		First, we note that $\Xf(\Ac)$ deformation retracts to $\Xf(\Ac) \cap S^{2\ell-1}$, its intersection with the unit sphere
		$S^{2\ell-1} \subseteq V$, as does $\Xf(\Bc_n) = (\CC^\times)^n$ to $(\CC^\times)^n\cap S^{2n-1}$ (up to a scaling factor).
		Moreover, after a suitable scaling of the defining linear forms $\alpha_i$ $(1\leq i \leq n)$, we may assume that $j(S^{2\ell-1}) \subseteq S^{2n-1}$.
		Note further that we have $j(\Xf(\Ac)) = j(V)\cap (\CC^\times)^n$ respectively 
		$j(\Xf(\Ac)\cap S^{2\ell-1}) = j(V) \cap (\CC^\times)^n\cap S^{2n-1}$.
		
		To construct the inclusions $\alpha, \beta$, we consider the description of the Salvetti complex via complex sign vectors 
		by Bj\"orner and Ziegler \cite{BjoeZie1992_CombStrat}. 
		In the following, we use a twisted version of their $s^{(1)}$-stratification
		to later ensure compatibility with our parametrization in Example \ref{ex:SalBooleanParam}.
		Let $s:\CC \to \{i,j,+,-,0\}$ be the map defined by
		\[
		s(z = x+iy) := \begin{cases}
			0, \text{ if } z=0,\\
			i, \text{ if } x=0 \text{ and }y>0,\\
			j, \text{ if } x=0 \text{ and }y<0,\\
			+, \text{ if } x>0,\\
			-, \text{ if } x<0.
		\end{cases}
		\]
		Then $s_\Ac:V \to \{i,j,+,-,0\}^n, v \mapsto (s(\alpha_1(v)),\ldots,s(\alpha_n(v))$ induces a PL regular cell decompositions
		of $S^{2\ell-1} \subseteq V$ as $\Sigma_\Ac = \bigcup_{\sigma \in s_\Ac(S^{2\ell-1})}s_\Ac^{-1}(\sigma)$
		whose face poset is isomorphic to the subposet of $\{i,j,+,-,0\}^n$ where vectors are ordered component-wise according to
		
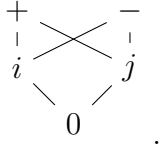
\begin{center}
			\begin{tikzpicture}[scale=0.75]
				\node (0) at (0,0) {$0$};
				\node (i) at (-1,1) {$i$};
				\node (j) at (1,1) {$j$};
				\node (+) at (-1,2) {$+$};
				\node (-) at (1,2) {$-$};
				\draw (0) -- (i);			
				\draw (0) -- (j);			
				\draw (i) -- (+);			
				\draw (i) -- (-);			
				\draw (j) -- (+);			
				\draw (j) -- (-);
			\end{tikzpicture}.
		\end{center}
		Now, $(s_\Ac^{-1}(\Xf(\Ac)\cap S^{2\ell-1}))^\dual \isom |\Sc(\Ac)|$ and yields $|\Sc(\Ac)|$ as a deformation retract of $\Xf(\Ac)$
		which provides the map $\alpha$, cf.\ \cite[Sec.~5]{BjoeZie1992_CombStrat}.
		Moreover, the same is true for the boolean arrangement $\Bc_n$ and for the corresponding cell complex $|\Sc(\Bc_n)|$ with its inclusion
		$\beta$ into $(\CC^\times)^n$ and we thus have $\beta(|S(\Bc_n)|) \cap j(V) = j(\alpha(|S(\Ac)|))$ which makes our diagram of inclusions commutative.	
	\end{proof}
	
	\begin{remark}
		\label{rem:BZexpParam}
		Recall the inclusion map $\beta:|S(\Bc_n)| = (S^1)^n \hookrightarrow \CC^n$ from Proposition \ref{prop:RelSalvetti}.
		and the parametrization map $\varphi: S(\Bc_n) \to (S^1)^n$ of the cells from Example \ref{ex:SalBooleanParam}.
		Then one easily sees that we have $\beta(|\sigma|) = |\varphi(\sigma)|$ for all $\sigma \in S(\Bc_n)$.
	\end{remark}
	
	To establish the commutativity of the right-hand square amounts to verifying the statement of Theorem \ref{thm:MilnorOMHomotopyEquivToMilnorGeom} for the Boolean arrangement as follows.
	\begin{proposition}
		We have a diagram
		\begin{center}
			\begin{tikzcd}[column sep=16mm, scale cd=1]
				\vert\rksdS(\Bc_n)\vert \ar[r,"\wt{Q}(\Bc_n)"] \ar[d, "|\wt{p}|"] \ar[dd, bend left=40, "\wt{\varphi}" right, hook] &\vert\Cc\vert \ar[dd,"\isom", hook]\\
				\vert\Sc(\Bc_n)\vert \ar[d, "\varphi" , hook] & \\
				(\CC^\times)^n \ar[r,"f_{\Bc_n}"] &\CC^\times.
			\end{tikzcd}
		\end{center}
		where the vertical maps are homotopy equivalences, $\wt{\varphi}$ is homotopic to $\varphi\circ|\wt{p}|$
		and the right square including $\wt{\varphi}$ commutes.
	\end{proposition}
	\begin{proof}
		Recall the explicit parametrization of the maps $\varphi$ and $\wt{\varphi}$ from examples \ref{ex:SalBooleanParam} and \ref{ex:TopeRkSubdivBoolean}.
		Then clearly, the right hand square commutes on the nose.
		
		But $\wt{\varphi} = \varphi \circ \psi$ where $\psi$ is the homeomorphism given by Theorem \ref{thm:rksdSubdivMap}.
		Further, $|\wt{p}|$ is homotopic to $\psi$ by the classical Carrier-Theorem, see \cite[Thm.~9.2]{LW1969_TopCWcpxs}.
	\end{proof}
	
	Combining all of the above yields the proof for Theorem \ref{thm:MilnorOMHomotopyEquivToMilnorGeom}.
	
	Moreover, by Lemma \ref{lem:CombFiberHtEquiv}, we directly obtain the following.
	
	\begin{theorem}
		\label{thm:MilnorFiberHtEquiv}
		The combinatorial Milnor fiber $\wt{\Ff}(\OM)$ is homotopy equivalent to the geometric Milnor fiber $\Ff(\Ac)$.
	\end{theorem}
	
	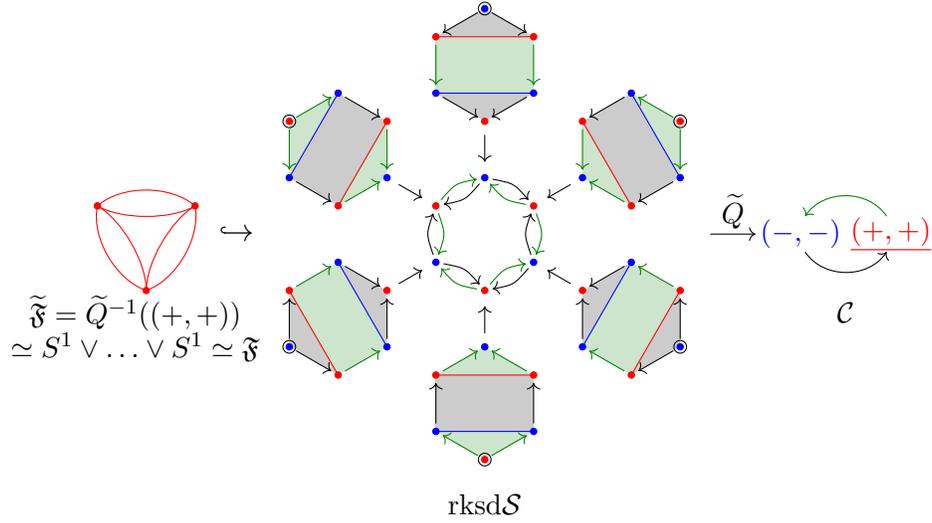
\begin{figure}
		\begin{tikzpicture}[scale=0.6]		
			
			\node (l1) at (1.7320508075688774,-1.) {};
			\node (l2) at (1.7320508075688774,1.) {};
			\node (l3) at (0,2.) {};
			\node (l4) at (-1.7320508075688774,1.) {};
			\node (l5) at (-1.7320508075688774,-1.) {};
			\node (l6) at (0,-2.) {};
			\node (zero) at (0,0) {};
			
			
			\node (l1_0) at ($0.5*(1.7320508075688774,-1.)$) {};
			\node (l2_0) at ($0.5*(1.7320508075688774,1.)$) {};
			\node (l3_0) at ($0.5*(0,2.)$) {};
			\node (l4_0) at ($0.5*(-1.7320508075688774,1.)$){};
			\node (l5_0) at ($0.5*(-1.7320508075688774,-1.)$) {};
			\node (l6_0) at ($0.5*(0,-2.)$) {};
			
			\draw[->,shorten >=3pt, shorten <=3pt] ($1.25*(l1_0)$) .. controls ($0.5*(l1_0)+0.5*(l2_0)$)  .. ($1.25*(l2_0)$);;
			\draw[->,shorten >=3pt, shorten <=3pt, color=black!50!green] ($1.25*(l2_0)$) .. controls ($0.75*(l1_0)+0.75*(l2_0)$)  .. ($1.25*(l1_0)$);;
			\draw[->,shorten >=3pt, shorten <=3pt, color=black!50!green] ($1.25*(l2_0)$) .. controls ($0.5*(l2_0)+0.5*(l3_0)$)  .. ($1.25*(l3_0)$);;
			\draw[->,shorten >=3pt, shorten <=3pt] ($1.25*(l3_0)$) .. controls ($0.75*(l2_0)+0.75*(l3_0)$)  .. ($1.25*(l2_0)$);;
			\draw[->,shorten >=3pt, shorten <=3pt] ($1.25*(l3_0)$) .. controls ($0.5*(l3_0)+0.5*(l4_0)$)  .. ($1.25*(l4_0)$);;
			\draw[->,shorten >=3pt, shorten <=3pt, color=black!50!green] ($1.25*(l4_0)$) .. controls ($0.75*(l3_0)+0.75*(l4_0)$)  .. ($1.25*(l3_0)$);;
			\draw[->,shorten >=3pt, shorten <=3pt, color=black!50!green] ($1.25*(l4_0)$) .. controls ($0.5*(l4_0)+0.5*(l5_0)$)  .. ($1.25*(l5_0)$);;
			\draw[->,shorten >=3pt, shorten <=3pt] ($1.25*(l5_0)$) .. controls ($0.75*(l4_0)+0.75*(l5_0)$)  .. ($1.25*(l4_0)$);;
			\draw[->,shorten >=3pt, shorten <=3pt] ($1.25*(l5_0)$) .. controls ($0.5*(l5_0)+0.5*(l6_0)$)  .. ($1.25*(l6_0)$);;
			\draw[->,shorten >=3pt, shorten <=3pt, color=black!50!green] ($1.25*(l6_0)$) .. controls ($0.75*(l5_0)+0.75*(l6_0)$)  .. ($1.25*(l5_0)$);;
			\draw[->,shorten >=3pt, shorten <=3pt, color=black!50!green] ($1.25*(l6_0)$) .. controls ($0.5*(l6_0)+0.5*(l1_0)$)  .. ($1.25*(l1_0)$);;
			\draw[->,shorten >=3pt, shorten <=3pt] ($1.25*(l1_0)$) .. controls ($0.75*(l6_0)+0.75*(l1_0)$)  .. ($1.25*(l6_0)$);;
			\filldraw[color=blue] ($1.25*(l1_0)$) circle[radius=2pt];
			\filldraw[color=red] ($1.25*(l2_0)$) circle[radius=2pt];
			\filldraw[color=blue] ($1.25*(l3_0)$) circle[radius=2pt];
			\filldraw[color=red] ($1.25*(l4_0)$) circle[radius=2pt];
			\filldraw[color=blue] ($1.25*(l5_0)$) circle[radius=2pt];
			\filldraw[color=red] ($1.25*(l6_0)$) circle[radius=2pt];
			
			
			\node (l1_1) at ($1.5*(l1)+0.5*(1.7320508075688774,-1.)$) {};
			\node (l2_1) at ($1.5*(l1)+0.5*(1.7320508075688774,1.)$) {};
			\node (l3_1) at ($1.5*(l1)+0.5*(0,2.)$) {};
			\node (l4_1) at ($1.5*(l1)+0.5*(-1.7320508075688774,1.)$){};
			\node (l5_1) at ($1.5*(l1)+0.5*(-1.7320508075688774,-1.)$) {};
			\node (l6_1) at ($1.5*(l1)+0.5*(0,-2.)$) {};
			
			\draw[->,shorten >=3pt, shorten <=3pt] ($1.25*(l1_1)$) -- ($1.25*(l2_1)$);
			\draw[->,shorten >=3pt, shorten <=3pt, color=black!50!green] ($1.25*(l2_1)$) -- ($1.25*(l3_1)$);
			\draw[->,shorten >=3pt, shorten <=3pt] ($1.25*(l3_1)$) -- ($1.25*(l4_1)$);
			\draw[->,shorten >=3pt, shorten <=3pt] ($1.25*(l1_1)$) -- ($1.25*(l6_1)$);
			\draw[->,shorten >=3pt, shorten <=3pt, color=black!50!green] ($1.25*(l6_1)$) -- ($1.25*(l5_1)$);
			\draw[->,shorten >=3pt, shorten <=3pt] ($1.25*(l5_1)$) -- ($1.25*(l4_1)$);
			\fill[color=black!50!green, opacity=0.2] ($1.25*(l2_1)$) -- ($1.25*(l3_1)$) -- ($1.25*(l5_1)$) -- ($1.25*(l6_1)$);
			\fill[color=black, opacity=0.2] ($1.25*(l1_1)$) -- ($1.25*(l2_1)$) -- ($1.25*(l6_1)$);
			\fill[color=black, opacity=0.2] ($1.25*(l3_1)$) -- ($1.25*(l4_1)$) -- ($1.25*(l5_1)$);
			\filldraw[color=blue] ($1.25*(l1_1)$) circle[radius=2pt];
			\draw ($1.25*(l1_1)$) circle[radius=4pt];
			\filldraw[color=red] ($1.25*(l2_1)$) circle[radius=2pt];
			\filldraw[color=blue] ($1.25*(l3_1)$) circle[radius=2pt];
			\filldraw[color=red] ($1.25*(l4_1)$) circle[radius=2pt];
			\filldraw[color=blue] ($1.25*(l5_1)$) circle[radius=2pt];
			\filldraw[color=red] ($1.25*(l6_1)$) circle[radius=2pt];
			\draw[color=red] ($1.25*(l6_1)$) -- ($1.25*(l2_1)$);
			\draw[color=blue] ($1.25*(l5_1)$) -- ($1.25*(l3_1)$);		
			
			%
			\draw[->] ($1.1*(l1)$) -- ($0.8*(l1)$);
			
			
			\node (l1_2) at ($1.5*(l2)+0.5*(1.7320508075688774,-1.)$) {};
			\node (l2_2) at ($1.5*(l2)+0.5*(1.7320508075688774,1.)$) {};
			\node (l3_2) at ($1.5*(l2)+0.5*(0,2.)$) {};
			\node (l4_2) at ($1.5*(l2)+0.5*(-1.7320508075688774,1.)$){};
			\node (l5_2) at ($1.5*(l2)+0.5*(-1.7320508075688774,-1.)$) {};
			\node (l6_2) at ($1.5*(l2)+0.5*(0,-2.)$) {};
			
			\draw[->,shorten >=3pt, shorten <=3pt, color=black!50!green] ($1.25*(l2_2)$) -- ($1.25*(l3_2)$);
			\draw[->,shorten >=3pt, shorten <=3pt] ($1.25*(l3_2)$) -- ($1.25*(l4_2)$);
			\draw[->,shorten >=3pt, shorten <=3pt, color=black!50!green] ($1.25*(l4_2)$) -- ($1.25*(l5_2)$);
			\draw[->,shorten >=3pt, shorten <=3pt, color=black!50!green] ($1.25*(l2_2)$) -- ($1.25*(l1_2)$);
			\draw[->,shorten >=3pt, shorten <=3pt] ($1.25*(l1_2)$) -- ($1.25*(l6_2)$);
			\draw[->,shorten >=3pt, shorten <=3pt, color=black!50!green] ($1.25*(l6_2)$) -- ($1.25*(l5_2)$);
			\fill[color=black!50!green, opacity=0.2] ($1.25*(l2_2)$) -- ($1.25*(l3_2)$) -- ($1.25*(l1_2)$);
			\fill[color=black!50!green, opacity=0.2] ($1.25*(l4_2)$) -- ($1.25*(l5_2)$) -- ($1.25*(l6_2)$);
			\fill[color=black, opacity=0.2] ($1.25*(l3_2)$) -- ($1.25*(l4_2)$) -- ($1.25*(l6_2)$) -- ($1.25*(l1_2)$);
			\filldraw[color=blue] ($1.25*(l1_2)$) circle[radius=2pt];
			\filldraw[color=red] ($1.25*(l2_2)$) circle[radius=2pt];
			\draw ($1.25*(l2_2)$) circle[radius=4pt];
			\filldraw[color=blue] ($1.25*(l3_2)$) circle[radius=2pt];
			\filldraw[color=red] ($1.25*(l4_2)$) circle[radius=2pt];
			\filldraw[color=blue] ($1.25*(l5_2)$) circle[radius=2pt];
			\filldraw[color=red] ($1.25*(l6_2)$) circle[radius=2pt];
			\draw[color=blue] ($1.25*(l1_2)$) -- ($1.25*(l3_2)$);
			\draw[color=red] ($1.25*(l6_2)$) -- ($1.25*(l4_2)$);
			
			\draw[->] ($1.1*(l2)$) -- ($0.8*(l2)$);
			
			
			\node (l1_3) at ($1.5*(l3)+0.5*(1.7320508075688774,-1.)$) {};
			\node (l2_3) at ($1.5*(l3)+0.5*(1.7320508075688774,1.)$) {};
			\node (l3_3) at ($1.5*(l3)+0.5*(0,2.)$) {};
			\node (l4_3) at ($1.5*(l3)+0.5*(-1.7320508075688774,1.)$){};
			\node (l5_3) at ($1.5*(l3)+0.5*(-1.7320508075688774,-1.)$) {};
			\node (l6_3) at ($1.5*(l3)+0.5*(0,-2.)$) {};
			
			\draw[->,shorten >=3pt, shorten <=3pt] ($1.25*(l3_3)$) -- ($1.25*(l4_3)$);
			\draw[->,shorten >=3pt, shorten <=3pt, color=black!50!green] ($1.25*(l4_3)$) -- ($1.25*(l5_3)$);
			\draw[->,shorten >=3pt, shorten <=3pt] ($1.25*(l5_3)$) -- ($1.25*(l6_3)$);
			\draw[->,shorten >=3pt, shorten <=3pt] ($1.25*(l3_3)$) -- ($1.25*(l2_3)$);
			\draw[->,shorten >=3pt, shorten <=3pt, color=black!50!green] ($1.25*(l2_3)$) -- ($1.25*(l1_3)$);
			\draw[->,shorten >=3pt, shorten <=3pt] ($1.25*(l1_3)$) -- ($1.25*(l6_3)$);
			\fill[color=black, opacity=0.2] ($1.25*(l3_3)$) -- ($1.25*(l4_3)$) -- ($1.25*(l2_3)$);
			\fill[color=black, opacity=0.2] ($1.25*(l5_3)$) -- ($1.25*(l6_3)$) -- ($1.25*(l1_3)$);
			\fill[color=black!50!green, opacity=0.2] ($1.25*(l4_3)$) -- ($1.25*(l5_3)$) -- ($1.25*(l1_3)$) -- ($1.25*(l2_3)$);
			\filldraw[color=blue] ($1.25*(l1_3)$) circle[radius=2pt];
			\filldraw[color=red] ($1.25*(l2_3)$) circle[radius=2pt];
			\filldraw[color=blue] ($1.25*(l3_3)$) circle[radius=2pt];
			\draw ($1.25*(l3_3)$) circle[radius=4pt];
			\filldraw[color=red] ($1.25*(l4_3)$) circle[radius=2pt];
			\filldraw[color=blue] ($1.25*(l5_3)$) circle[radius=2pt];
			\filldraw[color=red] ($1.25*(l6_3)$) circle[radius=2pt];
			\draw[color=red] ($1.25*(l2_3)$) -- ($1.25*(l4_3)$);
			\draw[color=blue] ($1.25*(l1_3)$) -- ($1.25*(l5_3)$);
			
			\draw[->] ($1.1*(l3)$) -- ($0.8*(l3)$);
			
			
			\node (l1_4) at ($1.5*(l4)+0.5*(1.7320508075688774,-1.)$) {};
			\node (l2_4) at ($1.5*(l4)+0.5*(1.7320508075688774,1.)$) {};
			\node (l3_4) at ($1.5*(l4)+0.5*(0,2.)$) {};
			\node (l4_4) at ($1.5*(l4)+0.5*(-1.7320508075688774,1.)$){};
			\node (l5_4) at ($1.5*(l4)+0.5*(-1.7320508075688774,-1.)$) {};
			\node (l6_4) at ($1.5*(l4)+0.5*(0,-2.)$) {};
			
			
			\draw[->,shorten >=3pt, shorten <=3pt, color=black!50!green] ($1.25*(l4_4)$) -- ($1.25*(l5_4)$);
			\draw[->,shorten >=3pt, shorten <=3pt] ($1.25*(l5_4)$) -- ($1.25*(l6_4)$);
			\draw[->,shorten >=3pt, shorten <=3pt, color=black!50!green] ($1.25*(l6_4)$) -- ($1.25*(l1_4)$);
			\draw[->,shorten >=3pt, shorten <=3pt, color=black!50!green] ($1.25*(l4_4)$) -- ($1.25*(l3_4)$);
			\draw[->,shorten >=3pt, shorten <=3pt] ($1.25*(l3_4)$) -- ($1.25*(l2_4)$);
			\draw[->,shorten >=3pt, shorten <=3pt, color=black!50!green] ($1.25*(l2_4)$) -- ($1.25*(l1_4)$);
			\fill[color=black!50!green, opacity=0.2] ($1.25*(l4_4)$) -- ($1.25*(l5_4)$) -- ($1.25*(l3_4)$);
			\fill[color=black!50!green, opacity=0.2] ($1.25*(l6_4)$) -- ($1.25*(l1_4)$) -- ($1.25*(l2_4)$);
			\fill[color=black, opacity=0.2] ($1.25*(l5_4)$) -- ($1.25*(l6_4)$) -- ($1.25*(l2_4)$) -- ($1.25*(l3_4)$);
			\filldraw[color=blue] ($1.25*(l1_4)$) circle[radius=2pt];
			\filldraw[color=red] ($1.25*(l2_4)$) circle[radius=2pt];
			\filldraw[color=blue] ($1.25*(l3_4)$) circle[radius=2pt];
			\filldraw[color=red] ($1.25*(l4_4)$) circle[radius=2pt];
			\draw ($1.25*(l4_4)$) circle[radius=4pt];
			\filldraw[color=blue] ($1.25*(l5_4)$) circle[radius=2pt];
			\filldraw[color=red] ($1.25*(l6_4)$) circle[radius=2pt];
			\draw[color=blue] ($1.25*(l3_4)$) -- ($1.25*(l5_4)$);
			\draw[color=red] ($1.25*(l2_4)$) -- ($1.25*(l6_4)$);
			
			\draw[->] ($1.1*(l4)$) -- ($0.8*(l4)$);
			
			
			\node (l1_5) at ($1.5*(l5)+0.5*(1.7320508075688774,-1.)$) {};
			\node (l2_5) at ($1.5*(l5)+0.5*(1.7320508075688774,1.)$) {};
			\node (l3_5) at ($1.5*(l5)+0.5*(0,2.)$) {};
			\node (l4_5) at ($1.5*(l5)+0.5*(-1.7320508075688774,1.)$){};
			\node (l5_5) at ($1.5*(l5)+0.5*(-1.7320508075688774,-1.)$) {};
			\node (l6_5) at ($1.5*(l5)+0.5*(0,-2.)$) {};
			
			\draw[->,shorten >=3pt, shorten <=3pt] ($1.25*(l5_5)$) -- ($1.25*(l6_5)$);
			\draw[->,shorten >=3pt, shorten <=3pt, color=black!50!green] ($1.25*(l6_5)$) -- ($1.25*(l1_5)$);
			\draw[->,shorten >=3pt, shorten <=3pt] ($1.25*(l1_5)$) -- ($1.25*(l2_5)$);
			\draw[->,shorten >=3pt, shorten <=3pt] ($1.25*(l5_5)$) -- ($1.25*(l4_5)$);
			\draw[->,shorten >=3pt, shorten <=3pt, color=black!50!green] ($1.25*(l4_5)$) -- ($1.25*(l3_5)$);
			\draw[->,shorten >=3pt, shorten <=3pt] ($1.25*(l3_5)$) -- ($1.25*(l2_5)$);
			\fill[color=black, opacity=0.2] ($1.25*(l5_5)$) -- ($1.25*(l6_5)$) -- ($1.25*(l4_5)$);
			\fill[color=black, opacity=0.2] ($1.25*(l1_5)$) -- ($1.25*(l2_5)$) -- ($1.25*(l3_5)$);
			\fill[color=black!50!green, opacity=0.2] ($1.25*(l6_5)$) -- ($1.25*(l1_5)$) -- ($1.25*(l3_5)$) -- ($1.25*(l4_5)$);
			\filldraw[color=blue] ($1.25*(l1_5)$) circle[radius=2pt];
			\filldraw[color=red] ($1.25*(l2_5)$) circle[radius=2pt];
			\filldraw[color=blue] ($1.25*(l3_5)$) circle[radius=2pt];
			\filldraw[color=red] ($1.25*(l4_5)$) circle[radius=2pt];
			\filldraw[color=blue] ($1.25*(l5_5)$) circle[radius=2pt];
			\draw ($1.25*(l5_5)$) circle[radius=4pt];
			\filldraw[color=red] ($1.25*(l6_5)$) circle[radius=2pt];
			\draw[color=red] ($1.25*(l4_5)$) -- ($1.25*(l6_5)$);
			\draw[color=blue] ($1.25*(l3_5)$) -- ($1.25*(l1_5)$);
			
			\draw[->] ($1.1*(l5)$) -- ($0.8*(l5)$);
			
			
			\node (l1_6) at ($1.5*(l6)+0.5*(1.7320508075688774,-1.)$) {};
			\node (l2_6) at ($1.5*(l6)+0.5*(1.7320508075688774,1.)$) {};
			\node (l3_6) at ($1.5*(l6)+0.5*(0,2.)$) {};
			\node (l4_6) at ($1.5*(l6)+0.5*(-1.7320508075688774,1.)$){};
			\node (l5_6) at ($1.5*(l6)+0.5*(-1.7320508075688774,-1.)$) {};
			\node (l6_6) at ($1.5*(l6)+0.5*(0,-2.)$) {};
			
			\draw[->,shorten >=3pt, shorten <=3pt, color=black!50!green] ($1.25*(l6_6)$) -- ($1.25*(l1_6)$);
			\draw[->,shorten >=3pt, shorten <=3pt] ($1.25*(l1_6)$) -- ($1.25*(l2_6)$);
			\draw[->,shorten >=3pt, shorten <=3pt, color=black!50!green] ($1.25*(l2_6)$) -- ($1.25*(l3_6)$);
			\draw[->,shorten >=3pt, shorten <=3pt, color=black!50!green] ($1.25*(l6_6)$) -- ($1.25*(l5_6)$);
			\draw[->,shorten >=3pt, shorten <=3pt] ($1.25*(l5_6)$) -- ($1.25*(l4_6)$);
			\draw[->,shorten >=3pt, shorten <=3pt, color=black!50!green] ($1.25*(l4_6)$) -- ($1.25*(l3_6)$);
			\fill[color=black!50!green, opacity=0.2] ($1.25*(l6_6)$) -- ($1.25*(l1_6)$) -- ($1.25*(l5_6)$);
			\fill[color=black!50!green, opacity=0.2] ($1.25*(l2_6)$) -- ($1.25*(l3_6)$) -- ($1.25*(l4_6)$);
			\fill[color=black, opacity=0.2] ($1.25*(l1_6)$) -- ($1.25*(l2_6)$) -- ($1.25*(l4_6)$) -- ($1.25*(l5_6)$);
			\filldraw[color=blue] ($1.25*(l1_6)$) circle[radius=2pt];
			\filldraw[color=red] ($1.25*(l2_6)$) circle[radius=2pt];
			\filldraw[color=blue] ($1.25*(l3_6)$) circle[radius=2pt];
			\filldraw[color=red] ($1.25*(l4_6)$) circle[radius=2pt];
			\filldraw[color=blue] ($1.25*(l5_6)$) circle[radius=2pt];
			\filldraw[color=red] ($1.25*(l6_6)$) circle[radius=2pt];
			\draw ($1.25*(l6_6)$) circle[radius=4pt];
			\draw[color=blue] ($1.25*(l5_6)$) -- ($1.25*(l1_6)$);
			\draw[color=red] ($1.25*(l4_6)$) -- ($1.25*(l2_6)$);
			
			\draw[->] ($1.1*(l6)$) -- ($0.8*(l6)$);
			
			\node at (0,-6) {\small$\rksdS$};
			
			\node (t) at (8,0) {};;
			\node[color=blue] (m) at ($(-1,0)+(t)$) {\small$(-,-)$};
			\node[color=red] (p) at ($(1,0)+(t)$) {\small\underline{$(+,+)$}};
			
			\draw[->,shorten >=3pt, shorten <=3pt, color=black!50!green] ($(p)+(0,0.25)$) .. controls ($0.75*(p)+(0,0.5)+0.25*(m)+(0,0.5)$) and ($0.25*(p)+(0,0.5)+0.75*(m)+(0,0.5)$)  .. ($(m)+(0,0.25)$);;
			\draw[->,shorten >=3pt, shorten <=3pt] ($(m)-(0,0.25)$) .. controls ($0.75*(m)-(0,0.5)+0.25*(p)-(0,0.5)$) and ($0.25*(m)-(0,0.5)+0.75*(p)-(0,0.5)$)  .. ($(p)-(0,0.25)$);;
			
			\node at ($(t)+(0,-1.75)$) {\small$\Cc$};
			
			\draw[->] (5,0) -- (6,0);
			\node at (5.5,0.5) {\small$\wt{Q}$};
			
			\node (l1_F) at ($0.5*(l1)-(6,0)$) {};
			\node (l2_F) at ($0.5*(l2)-(6,0)$) {};
			\node (l3_F) at ($0.5*(l3)-(6,0)$) {};
			\node (l4_F) at ($0.5*(l4)-(6,0)$){};
			\node (l5_F) at ($0.5*(l5)-(6,0)$) {};
			\node (l6_F) at ($0.5*(l6)-(6,0)$) {};
			
			\filldraw[color=red] ($1.25*(l2_F)$) circle[radius=2pt];
			\filldraw[color=red] ($1.25*(l4_F)$) circle[radius=2pt];
			\filldraw[color=red] ($1.25*(l6_F)$) circle[radius=2pt];
			
			\node (c) at  ($0.333*1.25*(l2_F) + 0.333*1.25*(l4_F) + 0.333*1.25*(l6_F)$) {};;
			\node (v1) at ($0.5*1.25*(l2_F) + 0.5*1.25*(l4_F) - (c)$) {};;
			\node (v2) at ($0.5*1.25*(l4_F) + 0.5*1.25*(l6_F) - (c)$) {};;
			\node (v3) at ($0.5*1.25*(l2_F) + 0.5*1.25*(l6_F) - (c)$) {};;
			
			%
			\draw[color=red] ($1.25*(l2_F)$) .. controls ($0.75*1.25*(l2_F)+0.25*1.25*(l4_F)+0.75*(v1)$) and ($0.25*1.25*(l2_F)+0.75*1.25*(l4_F)+0.75*(v1)$) .. ($1.25*(l4_F)$);;
			\draw[color=red] ($1.25*(l4_F)$) .. controls ($0.25*1.25*(l2_F)+0.75*1.25*(l4_F)-0.5*(v1)$) and ($0.75*1.25*(l2_F)+0.25*1.25*(l4_F)-0.5*(v1)$) .. ($1.25*(l2_F)$);;
			\draw[color=red] ($1.25*(l4_F)$) .. controls ($0.75*1.25*(l4_F)+0.25*1.25*(l6_F)+0.75*(v2)$) and ($0.25*1.25*(l4_F)+0.75*1.25*(l6_F)+0.75*(v2)$) .. ($1.25*(l6_F)$);;
			\draw[color=red] ($1.25*(l6_F)$) .. controls ($0.25*1.25*(l4_F)+0.75*1.25*(l6_F)-0.5*(v2)$) and ($0.75*1.25*(l4_F)+0.25*1.25*(l6_F)-0.5*(v2)$) .. ($1.25*(l4_F)$);;
			\draw[color=red] ($1.25*(l6_F)$) .. controls ($0.75*1.25*(l6_F)+0.25*1.25*(l2_F)+0.75*(v3)$) and ($0.25*1.25*(l6_F)+0.75*1.25*(l2_F)+0.75*(v3)$) .. ($1.25*(l2_F)$);;
			\draw[color=red] ($1.25*(l2_F)$) .. controls ($0.25*1.25*(l6_F)+0.75*1.25*(l2_F)-0.5*(v3)$) and ($0.75*1.25*(l6_F)+0.25*1.25*(l2_F)-0.5*(v3)$) .. ($1.25*(l6_F)$);;
			
			\node at ($(c) +(-0.25,-1.75)$) {\small$\wt{\Ff} = \wt{Q}^{-1}((+,+))$};
			\node at ($(c) +(-0.25,-2.5)$) {\small$\isom S^1\vee\ldots\vee S^1 \isom \Ff$};  
			
			\node at (-5.5,0) {$\hookrightarrow$};
			
		\end{tikzpicture}
		\caption{The combinatorial Minor fibration of the arrangement $\Ac$ in $\RR^2$ with defining polynomial $xy(x-y)$
			and its fiber.}\label{fig:ExCombMilnorA2}
	\end{figure}
	
	To conclude, we illustrate this central result with the following example.
	\begin{example}
		\label{ex:CombMilnorFibA2}
		Let $\Ac$ be again the arrangement in $\RR^2$ with defining polynomial $Q = Q(\Ac) = xy(x-y)$.
		Then it is not hard to see that the Milnor fiber $\Ff = Q^{-1}(1)$ is homotopy equivalent
		to a wedge of $4$ circles.
		Figure \ref{fig:ExCombMilnorA2} displays the tope-rank subdivision of the Salvetti complex
		of $\Ac$ and the combinatorial Milnor fibration map $\wt{Q}$ where the preimages of cells in $\Cc$ are
		colored accordingly.
		We clearly see that the combinatorial Milnor fiber $\wt{\Ff}$ is homotopy equivalent to
		a wedge of $4$ circles as well, in accordance with Theorem \ref{thm:MilnorFiberHtEquiv}.
	\end{example}

	
	\section{Conclusion}
	
	As seen in the previous sections, the finite regular cell complex $\wt{\Ff}$ can be constructed
	from oriented matroid data and thus be used for computations. 
	We implemented the complex in the computer algebra system GAP4 \cite{GAP4}. 
	It is available as part of the
	\emph{HypArr package} \cite{GAP_hyparr_pkg} for computations with hyperplane arrangements, oriented matroids and their invariants.
	Using our implementation and the further functionality of the \emph{HAP package} by G.\ Ellis \cite{HAP}, 
	we can readily confirm previous known examples, e.g.,
	the computation of Milnor fiber Betti numbers for some reflection arrangements
	in \cite{CS1995_MilnorFib} and the appearance of $2$-torsion in $H_1(F,\ZZ)$
	for the \emph{Icosidodecahedral arrangement} first discovered in \cite{Yoshinaga2020_Milnor2Torsion},
	providing a counterexample to a conjecture by Papadima--Suciu \cite{PS2017_modular}.
	
	Moreover, our complex is defined for any oriented matroid disregarding realizability.
	To this end, the above implementation may now be used for a systematic study
	of more examples, in particular considering also non-realizable oriented matroids
	with respect to the behavior of topological invariants of the complex $\wt{\Ff}$.
	We propose the following problems.
	\begin{problem}\hfill
		\label{prob:MFnon-real}
		\begin{enumerate}[1)]
			\item Do homotopy invariants of $\wt{\Ff}$ behave differently for non-realizable oriented matroids?
			
			\item Are there pairs of oriented matroids $\OM_1,\OM_2$ with the same underlying matroid
			such that the Betti numbers of $\wt{\Ff}(\OM_1)$ and $\wt{\Ff}(\OM_2)$ differ?
			
			\item Are there pairs of oriented matroids $\OM_1,\OM_2$ with the same underlying matroid
			such that $H_\bullet(\wt{\Ff}(\OM_1),\ZZ)$ and $H_\bullet(\wt{\Ff}(\OM_2))$ 
			have different torsion subgroups?
		\end{enumerate}
		
	\end{problem}

	
	\section*{Acknowledgments}
	We would like to thank Emanuele Delucchi for helpful discussions.
	We also thank James F.~Davis for pointing out the references \cite{AD2002_CombGrassmanians} and \cite{Bab1993_CombFlagSpace}.
	The first author was
	funded by the Deutsche Forschungsgemeinschaft (DFG, German Research Foundation), DFG Grant \#MU 5286/1-1 (project number: 539874788) to P.~M\"ucksch. 
	The second author was supported by JSPS KAKENHI JP23H00081.
	The first author is also thankful to the JSPS for a JSPS fellowship for foreign researchers in Japan during 2022/23
	where parts of this project first startet.
	
	
	\newcommand{\etalchar}[1]{$^{#1}$}
	\providecommand{\bysame}{\leavevmode\hbox to3em{\hrulefill}\thinspace}
	\providecommand{\MR}{\relax\ifhmode\unskip\space\fi MR }
	\providecommand{\MRhref}[2]{%
		\href{http://www.ams.org/mathscinet-getitem?mr=#1}{#2}
	}
	\providecommand{\href}[2]{#2}


\end{document}